\documentclass{amsart}%[oneside, 12pt]{amsart}
\usepackage{tikz}
\usepackage{xcolor}
\usepackage{amssymb,latexsym,amsmath,extarrows}
\usepackage{graphicx,mathrsfs,comment}
\usepackage{hyperref,url}
\usepackage{pict2e}
\usepackage{enumerate}

\usepackage{cancel}

\usepackage{amstext}
\usepackage{bbm} 

\numberwithin{equation}{section}

\usepackage{esint}
\newcommand{\cFl}{\mathcal{F}^{\lambda}}
\newcommand{\cpar}{c_{\text{par}}}
\newcommand{\gal}{\gamma^{\lambda}}
\newcommand{\phl}{\phi^{\lambda}}
\newcommand{\cF}{\mathcal{F}}
\newcommand{\cI}{\mathcal I}

\newtheorem{theorem}{Theorem}[section]
\newtheorem{lemma}[theorem]{Lemma}

\newtheorem{proposition}[theorem]{Proposition}
\newtheorem{remark}[theorem]{Remark}
\newtheorem*{remark*}{Remark}

\newtheorem{definition}[theorem]{Definition}
\newtheorem{corollary}[theorem]{Corollary}

\newtheorem{conjecture}[theorem]{Conjecture}

\newtheorem{List}{List}

\newcommand{\cE}{\mathcal{E}}

\newcommand{\W}{\mathbb{W}}

\makeatletter
\newcommand{\barredsum}{%
  \DOTSB\mathop{\mathpalette\@barredsum\relax}\slimits@
}
\newcommand{\@barredsum}[2]{%
  \begingroup
  \sbox\z@{$#1\sum$}%
  \setlength{\unitlength}{\dimexpr2pt+\ht\z@+\dp\z@\relax}%
  \@barredsumthickness{#1}%
  \vphantom{\@barredsumbar}%
  \ooalign{$\m@th#1\sum$\cr\hidewidth$#1\@barredsumbar$\hidewidth\cr}%
  \endgroup
}
\newcommand{\@barredsumbar}{%
  \vcenter{\hbox{\begin{picture}(0,1)\roundcap\Line(0,0)(0,1)\end{picture}}}%
}
\newcommand{\@barredsumthickness}[1]{% see https://tex.stackexchange.com/a/477200/
  \linethickness{%
    1.25\fontdimen8
      \ifx#1\displaystyle\textfont\else
      \ifx#1\textstyle\textfont\else
      \ifx#1\scriptstyle\scriptfont\else
      \scriptscriptfont\fi\fi\fi 3
  }%
}
\makeatother

\newcommand{\al}{\alpha}
\newcommand{\be}{\beta}
\newcommand{\ga}{\gamma}
\newcommand{\Ga}{\Gamma}
\newcommand{\de}{\delta}

\newcommand{\De}{\Delta}
\newcommand{\e}{\varepsilon}

\newcommand{\ka}{\kappa}
\newcommand{\la}{\lambda}

\newcommand{\si}{\sigma}
\newcommand{\Si}{\Sigma}

\newcommand{\om}{\omega}
\newcommand{\Om}{\Omega}

\newcommand{\cq}{\mathcal Q}

\newcommand{\cG}{\mathcal {G}}

\newcommand{\cb}{\mathcal B}

\newcommand{\ce}{\mathcal E}

\newcommand{\cf}{\mathcal F}

\newcommand{\f}{\frac}

\newcommand{\wt}{\widetilde}
\newcommand{\wh}{\widehat}

\newcommand{\nf}{\infty}

\newcommand{\ZR}{\mathbb{R}}
\newcommand{\ZT}{\mathbb{T}}
\newcommand{\ZH}{\mathbb{H}}
\newcommand{\ZZ}{\mathbb{Z}}

\newcommand{\ZS}{\mathbb{S}}

\newcommand{\Id}{{\bf{1}}}

\newcommand{\bA}{{\bf A}}
\newcommand{\bB}{{\bf B}}

\newcommand{\bT}{{\bf T}}

\newcommand{\Tau}{\mathcal{T}}

\newcommand{\Kc}{K_\circ}

\newcommand{\cB}{{\mathcal B}}
\newcommand{\cT}{{\mathcal T}}
\newcommand{\cQ}{{\mathcal Q}}

\newcommand{\cg}{\mathcal G}

\newcommand{\A}{\mathbb A}
\renewcommand{\S}{\mathbb{S}}

\newcommand{\bbr}{\textup{br}}

\newcommand{\R}{\mathbb{R}}

\newcommand{\T}{\mathbb{T}  }
\newcommand{\B}{\mathbb{B}}

\newcommand{\N}{\mathbb{N}}

\newcommand{\bfone}{\mathbf{1}}

\newcommand{\rap}{{\rm RapDec}}
\newcommand{\dist}{{\rm dist}}

\newcommand{\BL}{\textup{BL}}

\newcommand{\ang}{\measuredangle}

\newcommand{\supp}{{\rm supp}}
\newcommand{\proj}{{\rm{Proj}}}

\newcommand{\cW}{\mathcal{W}}

\newcommand{\hw}{e^{it\sqrt{-\Delta}}}

\newcommand{\smax}{s_{\text{max}}}

\begin{document}

\title[Local smoothing for wave equations]{On local smoothing estimates for wave equations}

\date{}

\author{Shengwen Gan} \address{ Shengwen Gan\\  Department of Mathematics\\ University of Wisconsin-Madison, USA} \email{sgan7@wisc.edu}

\author{Danqing He}
\address{Danqing He, School of Mathematical Sciences,
Fudan University, People's Republic of China}
\email{hedanqing@fudan.edu.cn}

 \author{Xiaochun Li}
 \address{Xiaochun Li, Department of Mathematics, University of Illinois at Urbana-Champaign, Urbana, IL, 61801, USA}
 \email{xcli@illinois.edu}

\author{Shukun Wu} \address{ Shukun Wu\\  Department of Mathematics\\ Indiana University Bloomington, USA} \email{shukwu@iu.edu}

\begin{abstract}

We prove sharp local smoothing estimates for wave equations on compact Riemannian manifolds in $n+1$ dimensions for odd $n$ and obtain improved estimates in even dimensions.
This is achieved by deriving local smoothing estimates for certain Fourier integral operators.
We also obtain improved local smoothing estimates for wave equations in Euclidean spaces.

\end{abstract}
\maketitle

%\tableofcontents

\section{Introduction}

\subsection{Overview}

The study of local smoothing originated in the late 1980s with independent work of Sjölin \cite{Sjo} and Vega \cite{vega}, who observed that solutions to the Schr\"odinger equation gain fractional regularity when localized in space and averaged in time. 
Sogge \cite{sogge1991propagation} subsequently identified analogous phenomena for the wave equation, formulated the local smoothing problem within the framework of Fourier integral operators, and emphasized its geometric character. 
In the 1990s, Wolff \cite{wolff2000local} revealed that sharp local smoothing estimates are closely linked to Kakeya-type phenomena, showing that geometric obstructions arise from configurations of thin tubes. 
This connection placed local smoothing at the intersection of dispersive PDE, harmonic analysis, and geometric measure theory.

\medskip

This paper establishes local smoothing estimates for wave equations on both Euclidean spaces and general compact Riemannian manifolds.
Let $(M, g)$ denote either $\mathbb{R}^n$ or a smooth, compact $n$-dimensional Riemannian manifold equipped with the corresponding Laplace--Beltrami operator $\Delta_g$. Consider $u$ as the solution to the Cauchy problem:
%This paper proves local smoothing estimates for wave equations on Euclidean spaces and on general compact Riemannian manifold.
%Let $(M,g)$ be $\ZR^n$ or a smooth, compact $n$-dimensional Riemannian manifold with the associated Laplace–Beltrami operator $\De_g$.
%Let $u$ be the solution of the Cauchy problem
\begin{equation}
\label{cauchy-problem}
    \begin{cases}
        (\partial^2_{t}-\Delta_g)u(x,t)=0, \hspace{.5cm}(x,t)\in M\times \R,\\[.5ex]
        u(x,0)=f_0(x), \hspace{1cm}
        \partial_t u(x,0)=f_1(x).
    \end{cases}
\end{equation}

\smallskip

When the underlying manifold is $M=\mathbb{R}^n$ with the flat metric, Peral \cite{peral1980lp} showed that for each fixed time $t$,  
$$
\|u(\cdot, t)\|_{L^p(M)} \lesssim_{M,g} \|f_0\|_{L^p_{s_p}(M)} + \|f_1\|_{L^p_{s_p-1}(M)},
$$  
for $p \in (1,\infty)$, where $s_p = (n-1)\big|\tfrac{1}{2}-\tfrac{1}{p}\big|$. The corresponding estimate for a general Riemannian manifold $(M,g)$ was later proved by Seeger-Sogge-Stein \cite{Seeger-Sogge-Stein}.   In subsequent work, Sogge \cite{sogge1991propagation} observed that averaging the solution over time intervals $t \sim 1$ yields an $L^p$ regularity gain for $2 < p < \infty$. In the same paper, he formulated the following {\it local smoothing} conjecture on $\mathbb R^n$:

%When $M=\ZR^n$ with flat metric, Peral \cite{peral1980lp} showed that for fixed time $t$,
  %  \[\|u(\cdot, t)\|_{L^p(M)}\lesssim_{M,g} \| f_0 \|_{L^p_{s_p}(M)}+\| f_1 \|_{L^p_{s_p-1}(M)},  \]
%for $p\in(1,\infty)$ and $s_p=(n-1)|\frac12-\frac1p|$ (for general $(M,g)$, this was proved by Seeger-Sogge-Stein in  \cite{Seeger-Sogge-Stein}). 
%Sogge \cite{sogge1991propagation} observed that with an averaging over $t\sim 1$, there is a gain of regularity in $L^p$ for $2<p<\infty$. 
%In the same paper, he made the following local smoothing conjecture:

\begin{conjecture}
\label{smoothing-conj-1}
Let $M=\ZR^n$ equipped with the flat metric, let $s_p=\frac{(n-1)}{2}-\frac{(n-1)}{p}$, and let $u$ solve the Cauchy problem \eqref{cauchy-problem}.
Then for all $p\geq 2+\frac{2}{n-1}$ and $\si<\frac1p$,
\begin{equation}
\label{smoothing-esti-1}
    \| u\|_{L^p(\R^n\times [1,2])}\le C_{p,\si}\big(\|f_0\|_{L^p_{s_p-\si}(\R^n)}+\|f_1\|_{L^p_{s_p-1-\si}(\R^n)}\big).
\end{equation}
\end{conjecture}

\smallskip

The local smoothing phenomenon also occurs on a general compact manifold $(M,g)$.
However, the range of admissible exponents is quite different from the flat case in Conjecture \ref{smoothing-conj-1}. 
Miniccozi and Sogge \cite{Minicozzi-Sogge} constructed examples to show that \eqref{smoothing-esti-1} fails for all $\si<\frac{1}{p}$ when $p<p_n^+$, where
\begin{equation}
\nonumber
    p_n^+:=
    \begin{cases}
        2+\frac{8}{3n-3}, \hspace{.3cm}\text{if $n$ is odd},\\[.5ex]
        2+\frac{8}{3n-2}, \hspace{.3cm}\text{if $n$ is even}.
    \end{cases}
\end{equation}
This observation naturally leads to the following conjecture:

\begin{conjecture}
\label{smoothing-conj-2}
For $n\geq2$, let $(M,g)$ be an $n$-dimensional compact Riemannian manifold, and let $u$ solve the Cauchy problem \eqref{cauchy-problem}.
Then for all $p\geq p_n^+$ and $\si<\frac1p$,
\begin{equation}
\label{smoothing-esti-2}
    \| u\|_{L^p(M\times [1,2])}\le C_{p,\si}\big(\|f_0\|_{L^p_{s_p-\si}(M)}+\|f_1\|_{L^p_{s_p-1-\si}(M)}\big).
\end{equation}
\end{conjecture}

Our main result is the following theorem:

\begin{theorem}\label{mainthm}
Conjectures \ref{smoothing-conj-1} and \ref{smoothing-conj-2} are true when $n\geq3$ and $p\geq p(n)$, where 
\begin{equation}
\label{p(n)}
    p(n):=
    \left\{\begin{array}{ll} 
    p_n^+&\textrm{$n$ is odd}\\[1ex]
    2+\frac 8{3n-2-\om(n)}&\textrm{$n$ is even.}
\end{array}
\right.
\end{equation}
Here, $\omega(n)=\frac{4(7n-2)}{3n^2+5n+14}=O(n^{-1})$.

\end{theorem}

Theorem \ref{mainthm} confirms Conjecture \ref{smoothing-conj-2} for all odd dimensions.
Note that $p(n) =p_n^+ +O(n^{-3})$ for all $n$ even.
The even-dimensional case appears to require new ideas and techniques.

\smallskip

\begin{remark}
    {\rm
    
     One may form a stronger conjecture by including the endpoint $\si= \frac1p$.
     For example, it was proved by Heo, Nazarov, and Seeger  \cite{heo2011radial} that Conjecture \ref{smoothing-conj-1} is true when $p\geq 2+\frac{4}{n-3}$ and $\si=\frac{1}{p}$ for $n\ge 4$. 
     It is open whether there are local smoothing estimates at the endpoint $\si=\frac1p$ for dimensions $n=2,3$.
    }
\end{remark}

\medskip

To study Conjecture \ref{smoothing-conj-1}, Wolff \cite{wolff2000local} initiated a program based on decoupling inequalities.
The investigation of decoupling inequalities culminated in Bourgain-Demeter's resolution on the $\ell^2$-decoupling theorem \cite{Bourgain-Demeter}.
As a direct corollary, they showed that Conjecture \ref{smoothing-conj-1} is true when $p>2+\frac{4}{n-1}$, where $2+\frac{4}{n-1}$ is the decoupling exponent for the light cone.
The approach was adapted to general compact Riemannian manifolds by Beltran-Hickman-Sogge \cite{beltran2020variable}, where they proved the same range of $p$ for Conjecture \ref{smoothing-conj-2}.
We refer to \cite{beltran2021sharp} for more historical information.

Conjecture \ref{smoothing-conj-1} was verified when $n=2$ by Guth-Wang-Zhang \cite{guth2020sharp}, however, via a different approach.
The authors there proved a sharp reverse square function estimate, which implies Conjecture \ref{smoothing-conj-1}.
This approach was modified by Gao-Liu-Miao-Xi \cite{gao2023square} to prove Conjecture \ref{smoothing-conj-2} in the same dimensions.

\smallskip

In this paper, we follow Wolff's footsteps and use refined decoupling inequalities to study the local smoothing problem.
Prior to this work, the decoupling range $p>2+\frac{4}{n-1}$ is the best-known result for both Conjecture \ref{smoothing-conj-1} and \ref{smoothing-conj-2} when $n\geq3$. 
Further discussion of our methods will be provided after an introduction to the Fourier integral operator, a standard tool for the local smoothing problem on manifolds.

\bigskip

\subsection{Fourier integral operator (FIO)}

Through a standard parametrix construction of the half wave operator $e^{it\sqrt{-\De_g}}$ (see \cite{beltran2021sharp} for details), the study of Conjecture \ref{smoothing-conj-2} boils down to the analysis of certain Fourier integral operators:
\begin{equation}\label{FIO}
    \cF f(x,t):=\int_{\R^n}e^{i\phi(x,t;\xi)}a(x,t;\xi)\wh f(\xi)\mathrm{d}\xi.
\end{equation}
Here, the symbol $a\in S^{0}(\R^{n+1}\times \R^n)$ is supported in $\{(x,t;\xi):|x|^2+t^2\le 1\}$. 
The phase function $\phi(x,t;\xi)$ is homogeneous of degree $1$ in $\xi$ and smooth away from $\xi=0$. Also, $\phi$ satisfies

\medskip

\noindent
$\bullet$ \textbf{(H1) Non-degeneracy condition.}

\textup{rank} $\partial^2_{x\xi}\phi(x,t;\xi)=n$ for all $(x,t;\xi)\in\supp\ a$.

\medskip
\noindent
$\bullet$ \textbf{(H2) Positive definiteness condition.}

Consider the Gauss map $G: \supp\ a\rightarrow S^n$ by $G(z;\xi):=\frac{G_0(z;\xi)}{|G_0(z;\xi)|}$, where
    \[G_0(z;\xi):=\bigwedge_{j=1}^n\partial_{\xi_j}\partial_{z}\phi(z;\xi). \]
Then for all $(z,\xi_0)\in\supp\ a$,
\begin{equation}
\nonumber
    \partial^2_{\xi\xi}\langle \partial_z \phi(z,\xi),G(z,\xi_0) \rangle|_{\xi=\xi_0} 
\end{equation} 
has rank $n-1$ with $n-1$ positive eigenvalues.

\smallskip

A prototypical example for the phase function is given by $\phi(x,t;\xi)=\langle x,\xi\rangle+t|\xi|$. 
The Fourier integral operator defined using this phase function corresponds to the classical half-wave propagator $\hw f$.
One may also keep in mind another example  $\phi(x,t;\xi)=\langle x,\xi\rangle + t|\xi'|^2/\xi_n$. Here $\xi'=(\xi_1,\dots,\xi_{n-1})$. 

\smallskip

As shown in \cite{beltran2021sharp}, to prove Conjecture \ref{smoothing-conj-2}, it suffices to prove the following conjecture for FIOs satisfying the two conditions \textbf{(H1)} and \textbf{(H2)}.

\begin{conjecture}\label{FIO-conj}
Let $\cF$ be defined in \eqref{FIO}.
Suppose $\cF$ satisfies conditions \textbf{\textup{(H1)}} and \textbf{\textup{(H2)}}.
Then for $p\geq p_n^+$ and $\si>(n-1)(\frac12-\frac1p)-\frac1p$, 
\begin{equation}
\label{FIO-esti}
    \|\cF f\|_{L^p(\R^n\times [1,2])}\lesssim \|f\|_{L^p_\si(\R^n)}.
\end{equation}
\end{conjecture}
\begin{remark}
\rm
Via a similar argument, to prove Conjecture \ref{smoothing-conj-1}, it suffices to prove \eqref{FIO-esti} with $\cf=e^{it\sqrt{-\De}}$ for the given range of $p$.
\end{remark}

Similarly, to establish Theorem \ref{mainthm}, we only need to prove

\begin{theorem}
\label{FIOthm}
Conjecture \ref{FIO-conj} is true when $p\geq p(n)$, where $p(n)$ is given in \eqref{p(n)}.
\end{theorem}

\bigskip

\subsection{Main ideas}

To explain our ideas, we take $\cf=e^{it\sqrt{-\De}}$ in Conjecture \ref{FIO-conj}.
The same ideas apply to general FIOs.

Since the kernel of the operator $e^{it\sqrt{-\De}}$ is essentially the Fourier inversion of the surface measure of the cone $\{(\xi,|\xi|): \xi\in\R^n\}$, it is natural to connect the local smoothing problem with Fourier restriction theory.
This connection was shown to be successful by Rogers \cite{rogers2008local} in the study of the Schr\"odinger equations (with the half-wave operator $e^{it\sqrt{-\De}}$ being replaced by the Schr\"odinger operator $e^{it\De}$).
In fact, Rogers showed that the local smoothing conjecture for Schr\"odinger equations is equivalent to the Fourier restriction conjecture for paraboloid.
Moreover, since the tools developed for the parabolic restriction conjecture are versatile and apply equally well to hypersurfaces with a positive definite second fundamental form, it is reasonable to believe that the same ideas will extend to the local smoothing problem of the same type.
Successful examples can be found in \cite{gan2022note}, where fractional Schr\"odinger equations with the operators $e^{it(-\De)^{\al/2}},\, \al>1,$ are investigated.

\smallskip

However, there is a significant difference when one attempts to relate the local smoothing problem for wave equations to the conic restriction problem.
Indeed, the local smoothing problem is generally believed to be strictly harder than the conic and parabolic restriction problem (see \cite{Tao-B-R-restriction}).
For example, the exponent $2+\frac{4}{n-1}$, also known as the Stein-Tomas exponent, was established for the restriction problem back in the 1970s. 
In contrast, for the local smoothing problem, this exponent was only obtained much later by Bourgain and Demeter in 2014.

One key difference between paraboloid and light cone is that the former possesses a better symmetry: the natural rescaling for paraboloid, the parabolic rescaling (see \cite{guth2016restriction}), is an isotropic rescaling when restricting to the initial data, whereas the natural rescaling for the light cone, the Lorentz rescaling (see Section \ref{Lorentz-rescaling-section}), is non-isotropic.
As a consequence, although induction is well-established in the study of the parabolic restriction problem and hence the related local smoothing problem for Schr\"odinger equations, there seems to be no satisfactory induction on scales scheme for wave equations.

\medskip

In this paper, we propose a possible candidate for induction on scales for the wave equations.
To better explain our idea, we will compare the half-wave operator $e^{it\sqrt{-\De}}$ with the Schr\"odinger operator $e^{it\De}$.
After standard reductions, the desired local smoothing estimates for the two operators and the parabolic restriction problem are reduced to the following local forms: 

\medskip

\noindent
\textit{Local smoothing estimate for wave equations}: 
For $\supp\wh f\subset B^n_1\setminus B^n_{1/2}$,
\begin{equation}\label{localwave}
    \|\hw f\|_{L^p(B^{n+1}_R)}\lessapprox R^{(n-1)(\frac12-\frac1p)}\|f\|_{p}.
\end{equation}  
\smallskip

\noindent
\textit{Local smoothing estimate for Schr\"odinger equations}:
For $\supp\wh f\subset B^{n}_1$,
\begin{equation}\label{localschr}
    \|e^{it\Delta} f\|_{L^p(B^{n+1}_R)}\lessapprox R^{n(\frac12-\frac1p)}\|f\|_{p}.
\end{equation} 
\smallskip

\noindent
\textit{Fourier restriction estimate for paraboloid}:
For $\supp\wh f\subset B^{n}_1$,
\begin{equation}\label{restriction}
    \|e^{it\Delta} f\|_{L^p(B^{n+1}_R)}\lessapprox \|\wh f\|_{p}.
\end{equation} 

\medskip

To study the restriction estimate \eqref{restriction} via induction on scales, Guth \cite{guth2016restriction} introduced the mixed norm 
\begin{equation}\label{mimic}
    \|\wh f\|_{2}^{\frac2p} \sup_{\theta}\|\wh f_\theta\|_{L^2_{avg}(\theta)}^{1-\frac2p}
\end{equation}
in place of the $L^p$-norm $\|\wh f\|_{L^p(B^{n}_1)}$, where $\|\wh f_\theta\|_{L^2_{avg}(\theta)}^2=|\theta|^{-1}\|\wh f_\theta\|_{L^2(\theta)}^2$, and $\theta$ ranges over $R^{-1/2}$-balls in $B^n_1$.
The mixed norm \eqref{mimic} can be viewed as the $L^p$-norm by  interpolation, in which the quantity $\sup_{\theta}\|\wh f_\theta\|_{L^2_{avg}(\theta)}$ plays the roll of $\|\wh f\|_{L^\infty}$. 
It is crucial that $\sup_{\theta}\|\wh f_\theta\|_{L^2_{avg}(\theta)}$ is an $L^2$-average.
If this is simply replaced by $\|\wh f\|_{L^\infty}$, inductions will generally fail.

\smallskip

Given the connection between the parabolic restriction conjecture and the local smoothing conjecture for Schr\"odinger equations, it is thus natural to consider a suitable mixed norm similar to \eqref{mimic} to investigate \eqref{localschr}.
One subtlety is that the right-hand side of \eqref{localschr} lies in the physical space, while the right-hand side of \eqref{restriction} lies in the frequency space.
This leads to a different mixed norm, which may be chosen as follows:
    \[R^{n(\frac12-\frac1p)}\|f\|_{2}^{\frac2p}\sup_{B}\|f\|_{L^2_{avg}(B)}^{1-\frac2p}, \]
where $B$ ranges over $R^{1/2}$-balls in $B^n_R$.
Using this mixed norm, it is possible that the recent framework developed by Wang and the second author in \cite{wang2024restriction} for the restriction conjecture can lead to parallel results for local smoothing estimate for Schr\"odinger equations \eqref{localschr}.

Based on the above discussion, it seems natural to study the local smoothing estimate for wave equations \eqref{localwave} by considering a similar mixed norm:
\begin{equation}\label{modify}
    R^{(n-1)(\frac12-\frac1p)}\|f\|_{2}^{\frac2p}\sup_{B}\|f\|_{L^2_{avg}(B)}^{1-\frac2p},
\end{equation} 
where $B$ ranges over $R^{1/2}$-balls in $B^n_R$.
However, this naive choice of the mixed norm does not work, since the form in \eqref{modify} is not invariant under Lorentz rescaling, a non-isotropic one.

\medskip

Inspired by the work \cite{guth2020sharp}, we introduce a mixed norm 
\begin{equation}
\label{mixed-norm-wpd}
    R^{(n-1)(\frac12-\frac1p)}\|f\|_{2}^{\frac2p}\cW( f,B_R^{n+1})^{1-\frac2p},
\end{equation}
where the \textit{wave packet density} $\cW(f, B_R^{n+1})$ is a rescaling-invariant form of an $L^2$ average that serves as an analogue of $\|f\|_\infty$.
The precise definition of $\cW(f,B_R^{n+1})$ is given in Definition \ref{wpd-def}.
Roughly speaking, it is the supreme of all the $L^2$-average over conic wave packets and their rescaled counterparts under Lorentz rescaling.

\smallskip

The introduction of the wave packet density is one of the main ideas of this work.
With the mixed norm \eqref{mixed-norm-wpd} in hand, induction on scales becomes possible, and we would like to adapt the framework developed in \cite{wang2024restriction} to study \eqref{localwave}.
After the standard wave packet decomposition and several steps of dyadic pigeonholing, our goal is to establish the following: 
Suppose $f=\sum_{T\in\ZT}f_T$ is a wave packet decomposition of $f$, where $\ZT$ is a family of $1\times R^{1/2}\times\cdots\times R^{1/2}\times R$-planks.
Given a union of unit balls $X\subset B^{n+1}_R$ such that each unit ball $B\subset X$ intersects with $\lesssim \mu$ planks in $\ZT$, prove that when $p=p(n)$,
\begin{equation}\label{wave-mixed-norm}
    \|\hw f\|_{L^p(X)}\lessapprox R^{(n-1)(\frac12-\frac1p)}\|f\|_{2}^{\frac2p}\cW( f,B_R^{n+1})^{1-\frac2p}.
\end{equation}  
As in \cite{wang2024restriction}, we aim to
\begin{enumerate}
    \item Prove a refined decoupling inequality for the cone to bound $\|\hw f\|_{L^{q_n}(X)}$ with $q_n=\frac{2(n+1)}{n-1}$ being the decoupling exponent.
    \item Explore orthogonality in the $L^2$ space.
    That is, estimate $\|\hw f\|_{L^2(X)}$.
\end{enumerate}
It is fairly standard to establish item (1), and we prove the refined decoupling inequality for general FIOs in the appendix.
However, it is not clear what the optimal $L^2$-orthogonality for $\|\hw f\|_{L^2(X)}$ is, as $X$ is a union of unit balls. 
One simple approach is to enlarge $X$ to its $R^{1/2}$-neighborhood and then use $L^2$ orthogonality on the $R^{1/2}$-balls.
As we will see below, this already allows us to obtain some useful $L^2$ estimates.

\smallskip

We introduce a parameter $l$ that measures how many $R^{1/2}$-balls interact with a plank $T$.
Specifically, for each plank $T\in\ZT$, define a shading $Y(T)=X\cap T$, and assume that $Y(T)$ intersects with $\sim l$ many $R^{1/2}$-balls in $\ZR^{n+1}$.
Thus, by using $L^2$-orthogonality inside each $R^{1/2}$-ball in $N_{R^{1/2}}(X)$, we have the following $L^2$-estimate
\begin{equation}
\label{L2-1}
    \|\hw f\|_{L^2(X)}^2\lesssim \|\hw f\|_{L^2(N_{R^{1/2}}(X))}^2\lesssim(lR^{1/2})\|f\|_2^2.
\end{equation}

Next, at the decoupling endpoint $p=\frac{2(n+1)}{n-1}$, we use the refined decoupling inequality for the cone to have
\begin{equation}
\nonumber
    \|\hw f\|_{L^p(X)}^p\lessapprox \mu^{\frac{2}{n-1}}\sum_{T\in\ZT}\|\hw f_{T}\|_p^p,
\end{equation}
After normalization, we assume that $|\hw f_{T}|$ is essentially the characteristic function of the $1\times R^{1/2}\times\cdots\times R^{1/2}\times R$-plank $T$.
Thus, when $p=\frac{2(n+1)}{n-1}$,
\begin{equation}
\label{after-dec-1}
    \|\hw f\|_{L^p(X)}^p\lessapprox \mu^{\frac{2}{n-1}}(R^{\frac{n+1}{2}}\#\ZT).
\end{equation}
It remains to find an upper bound for $\mu^{\frac{2}{n-1}}\#\ZT$, which is an incidence problem between the unit balls in $X$ and the $1\times R^{1/2}\times\cdots\times R^{1/2}\times R$-planks in $\ZT$.

Now we are going to use the power of induction on scales.
Broadly speaking, induction on scales allows us to impose a two-ends condition (which is a non-concentration condition) on the shading $Y(T)$. 
Also, the wave packet density $\cW(f,B_R^{n})$ imposes a non-concentration condition on the collection of planks $\ZT$ (see Remark \ref{wpd-cwa-rmk-1}), which, in particular, implies $\cW(f,B_R^{n+1})^2\gtrsim R^{-\frac{n+1}{4}}(\#\ZT)^{1/2}$.
Therefore, by using the hairbrush structure (see Figure \ref{hairbrush1}) and a ``2-broad" assumption on each unit ball in $X$, we have
\begin{equation}
\label{mu-intro}
    \mu\lesssim l^{-1/2}(\#\ZT)^{1/2}\lesssim l^{-1/2} R^\frac{n+1}{4}\cW(f,B_R^{n+1})^2.
\end{equation}
We remark that the hairbrush structure for the cone has the following special property: all the light rays intersecting a fixed line are essentially disjoint.

\smallskip

Finally, we plug the incidence bound back to  \eqref{after-dec-1} to get when $p=\frac{2(n+1)}{n-1}$,
\begin{equation}
\label{L-decoupling}
    \|\hw f\|_{L^p(X)}^p\lessapprox  R^2(R^{-\frac{n-3}{2(n-1)}}l^{-\frac{1}{n-1}}) \|f\|_2^2\,\cW( f,B_R^{n+1})^{p-2}.
\end{equation}
Interpolate this with \eqref{L2-1} and use the fact that $l\lesssim R^{1/2}$ to conclude \eqref{localwave} for $p=2+\frac{8}{3n-4}$, which already is a strong bound.

\begin{remark}
\rm
The wave packet density $\cW(f,B_R^{n})$ is quite versatile.
Modifications to its definition will allow us to impose stronger non-concentration conditions on the set of planks $\ZT$.
We refer the readers to Remark \ref{wpd-cwa-rmk-2} for further discussions.
A similar object is also considered recently in \cite{GWX}.
\end{remark}

\bigskip 

\subsection{Sharp exponents in odd dimensions}

To further extend the admissible range of $p$, we refine the aforementioned incidence and $L^2$ estimates.
Together with some $k$-broad estimates (which is also used in a recent work \cite{FHL}), we achieve a complete solution to the problem in odd dimensions.
Let us briefly outline the method here. 
We refer to Section \ref{section-sketch} for a sketch of the proof.

By the standard broad-narrow argument, the inequality \eqref{wave-mixed-norm} at $p = p(n)$ can be reduced to a $k_n$-broad norm, where $k_n$ is given by
\begin{equation}
\label{defofkn}
    k_n=
    \left\{
    \begin{array}{ll} 
    \tfrac{n+5}2; &\textrm{$n\ge 3$ is odd},\\[1ex]
    \f{n+4}{2};& \textrm{$n\ge 2$ is even}.%\\
%2&n=4
    \end{array}
    \right.
\end{equation}
By pigeonholing, we may assume that $X$ can be covered by a collection of $R^{1/2}$-balls $Q$, each of which 
satisfies
\begin{equation}
    \big|X\cap Q\big| \sim \sigma, \text{ and }\, \#\mathbb T(Q) \sim m \,,
\end{equation}
where $\mathbb T(Q):= \{T\in\mathbb T: X\cap Q\cap T\neq \emptyset\}$, and $m,\sigma$ are some dyadic numbers.

To obtain a refined incidence estimate, on the one hand, using a hairbrush argument, we have (note that $m\geq\mu$)
\begin{equation}
\label{incidence}
    |X| \gtrsim l m \lambda.
\end{equation}
On the other hand, a double-counting argument yields
\begin{equation}
\label{dbcon}
    \mu |X| \sim \lambda (\#\mathbb{T}).
\end{equation}
Combining\eqref{incidence} and \eqref{dbcon}, we obtain at a refinement of \eqref{mu-intro} given by 
\begin{equation}
\label{mu-bdd}
    \mu \lesssim l^{-1} m^{-1} (\#\mathbb{T}).
\end{equation}

Clearly, the larger $m$ is, the stronger \eqref{mu-bdd} becomes, while the refinement of the $L^2$ estimates depends inversely on $m$: the smaller $m$ is, the better $L^2$ estimate we can obtain.

\smallskip 

Note that inside each $R^{1/2}$ ball $Q$, $T\cap Q$ is a slab of dimensions $1\times R^{1/2}\times\cdots\times R^{1/2}$ for each $T\in\ZT(Q)$.
Let $\ZT(B)=\{T\in\ZT(Q):B\cap T\not=\varnothing\}$ for each unit ball $B\subset X\cap Q$.
From the broad-narrow reduction, we know that the set of planks $\ZT(B)$ is $k_n$-broad; that is, the directions of $\ZT$ are $k_n$-broad.
One might expect this to imply that the set of slabs $\{T\cap Q: T\in\ZT(B)\}$ is also $k_n$-broad, meaning that the normal directions of $\ZT(B)$ are $k_n$-broad.
If this expectation holds, a favorable lower bound for 
$m$ can be obtained using a $k_n$-linear estimate for the slabs $\{T\cap Q: T\in\ZT(Q)\}$.
This is indeed the case when the FIO $\cf$ is the half-wave operator ${\hw}$; however, it fails for general FIOs.
Our corresponding treatment is dimension-sensitive, and we outline the corresponding arguments separately.
The cases $n=3$ and $n=5$ turn out to be the most interesting.

\medskip

\subsubsection{When $n=3$}

We use the approach developed in \cite{Li-Wu} to prove a refined $L^2$ estimate in place of \eqref{L2-1}.
Via the covering lemma, Lemma \ref{regular-lem}, there is a set of $1\times R^{1/2}\times R^{1/2}\times R^{1/2}$-slabs $\ZS(Q)$ such that the following is true: 
\begin{enumerate}
    \item $\ZS(Q)$ essentially forms a cover of $X\cap Q$.
    \item $\#\ZS(Q)\sim \si/\rho$.
    \item For any arbitrary $1\times R^{1/2}\times R^{1/2}\times R^{1/2}$-slab $S\subset Q$, $|S\cap X|\lesssim\rho$.
\end{enumerate}
Let $f_Q=\sum_{T\in\ZT_Q}f_T$.
The key observation is that if each $B\subset X\cap Q$ is ``4-broad" (note that $k_3=4)$, then we have the $L^2$ estimate 
\begin{equation}
\nonumber
    \|\hw f\|_{L^2(X\cap Q)}^2\lesssim\sum_{S\in\ZS(Q)}\|\hw f\|_{L^2(X\cap S)}^2\lesssim\#\ZS(Q)\|f_Q\|_2^2\sim(\si/\rho)\|f_Q\|_2^2.
\end{equation}
Sum up all $R^{1/2}$-balls to obtain the refined $L^2$ estimate
\begin{equation}
\label{refined-L2-1}
    \|\hw f\|_{L^2(X)}^2\lesssim l\si/\rho\|f\|_2^2.
\end{equation}
Note that \eqref{refined-L2-1} is stronger than \eqref{L2-1} when $\si/\rho$ is smaller than $R^{1/2}$.

\smallskip

Next, observe that $\{T\cap Q\}_{T\in\ZT(Q)}$ are distinct $1\times R^{1/2}\times R^{1/2}\times R^{1/2}$-slabs. 
Since $|S\cap X|\lesssim\rho$ for any arbitrary $1\times R^{1/2}\times R^{1/2}\times R^{1/2}$-slab $S\subset Q$, by double counting the unit balls in $X\cap Q$ and planks in $\ZT(Q)$, we have $\#\ZT(Q)\sim m\gtrsim \mu(\si/\rho)$.
Put this back to \eqref{mu-bdd} so that
\begin{equation}
\nonumber
    \mu\lesssim \rho^{1/2}(l\si)^{-1/2}(\#\ZT)^{1/2}\lesssim\rho^{1/2}(l\si)^{-1/2}R\cdot\cW(f,B_R^{4})^2.
\end{equation}
Plugging this back to \eqref{after-dec-1} and interpolating with \eqref{refined-L2-1}, we prove Theorem \ref{FIOthm} when $n=3$.
A sketch of our argument will be given in Section \ref{section-sketch-n=3}.

\medskip

\subsubsection{When $n=5$}

Although the slabs $\{T \cap Q : T \in \ZT(Q)\}$ do not exhibit strong $5$-broadness (recall that $k_5 = 5$), they do satisfy a weaker $3$-broadness property.
Thus, we can use a trilinear estimate for these slabs and obtain (with $n=5, k=3$)
\begin{equation}
\nonumber
    \#\mathbb{T}(Q) \sim m \gtrsim \mu \left( \frac{\sigma}{R^{\frac{n+1-k}{2}}} \right)^{\frac{1}{k}},
\end{equation}
which, together with \eqref{mu-bdd}, gives
\begin{equation}
\label{mu-bdd-2}
    \mu \lesssim l^{-1/2}  (\#\mathbb{T})^{1/2}\Big( \frac{R^{3/2}}{\sigma} \Big)^{1/6} .
\end{equation}

By locally using the $k_n$-broad estimate established in \cite{Schippa} inside each $R^{1/2}$ ball $Q$, we can refine \eqref{L2-1} as (with $n=5$)
\begin{equation}
\label{l2-intro-2}
     \|\sum_{T\in\ZT}\hw f_T\|_{L^2(X)}\lesssim l^{1/2}\min (R^{1/4},\si^{\frac{1}{n+k_n+1}})\|f\|_{L^2}.
\end{equation}
Put \eqref{mu-bdd-2} back to \eqref{after-dec-1} to obtain an improved $L^{\frac{2(n+1)}{n-1}}$ estimate.
The desired estimates at $p(n)$ are then obtained by interpolating between this $L^{\frac{2(n+1)}{n-1}}$ estimate, the $L^2$ estimate \eqref{l2-intro-2}, and $k_n$-broad estimate by \cite{Schippa} in the $L^{\frac{2(n+k_n+1)}{n+k_n-1}}$ space.
See Section \ref{section-sketch-ngeq5} for a sketch of its proof.

\medskip

\subsubsection{When $n\geq7$}

In this case, no refinement on incidence or $L^2$ estimates is needed.
The desired estimates at $p(n)$ are then obtained by interpolating between the $L^{\frac{2(n+1)}{n-1}}$ estimate \eqref{L-decoupling}, the $L^2$ estimate \eqref{L2-1}, and $k_n$-broad estimate  in the $L^{\frac{2(n+k_n+1)}{n+k_n-1}}$ space.

\bigskip

\subsection{A broad \& two-ends algorithm}

We conclude the introduction with a comment on a technical difficulty encountered during the proof.
Let $\hw f=\sum_{T\in\ZT}\hw f_T$ be the wave packet decomposition.
Via a standard broad-narrow reduction, it suffices to consider the ``broad" norm $\|\sum_{T\in\ZT}\hw f_T\|_{\BL^p(X)}$.
That is, for each unit ball $B\subset X$, $|\sum_{T\in\ZT}\hw f_T|\Id_B$ is ``broad", which essentially means that its contribution comes from groups of wave packets with quantitatively linearly independent directions.
To utilize the ``broad" information and the two-ends condition on the shading $Y(T)=X\cap T$, we need the following two conditions to hold simultaneously:
\begin{enumerate}
    \item For each unit ball $B\subset X$, $|\sum_{T\in\ZT}\hw f_T|\Id_B$ is broad.
    \item For each plank $T\in\ZT$, $Y(T)=X\cap T$ satisfies a two-ends condition.
\end{enumerate}

It turns out that achieving the two conditions requires considerable effort, mostly because the sum in $|\sum_{T\in\ZT}\hw f_T|\Id_B$ is a sum of oscillatory functions.
To compare, if we were to consider $\sum_{T\in\ZT}|\hw f_T|\Id_B$, where oscillation is removed, then several steps of dyadic pigeonholing would suffice to ensure the two required conditions hold simultaneously.
When there is oscillation, as is the case here, we essentially prove that if the two conditions do not hold simultaneously, then the total incidence $\sum_{B\subset X}\#\{T\in\ZT:T\cap B\not=\varnothing\}$ must decrease by a factor of $R^\ka$ for some absolute $\ka>0$.
As a result, we develop an algorithm to show that the two conditions must hold simultaneously after finite steps of refinement.
See Section \ref{section-algorithm}, Remark \ref{algorithm-remark}, and Section \ref{section-iteration} for details.

\bigskip

\subsection{Structure of the paper}
In Section 2, we give a sketch of the proof of our main result.
Section 3 discusses preliminaries, followed by the introduction of wave packet density in Section 4.
Section 5 contains some geometric results.
In Section 6, we introduce an algorithm that, after being iterated in Section 7, helps us realize the broad \& two-ends reduction on our operator.
Finally, we prove the main theorem in Sections 8 and 9.  
The appendix includes a proof of the refined decoupling theorem for the cone in the variable coefficient setting.

\bigskip

\noindent {\bf Notation:} 
Throughout the paper, we use $\# E$ to denote the cardinality of a finite set.
If $\ce$ is a family of sets in $\ZR^n$, we use $\cup_\ce$ to denote $\cup_{E\in\ce}E$.
For $A,B\geq 0$, we use $A\lesssim B$ to mean $A\leq CB$ for a (big) constant $C$, and use $A\sim B$ to mean $A\lesssim B$ and $B\lesssim A$.
For a given $R>1$, we use $ A \lessapprox B$ to denote $A\leq c_\eta R^{\eta} B$ for all $\eta>0$.
For two finite sets $E,F$, we say $E$ is a {\bf refinement} of $F$, if $E\subset F$ and  $\#E\gtrapprox \#F$.
We use $\mathrm{RapDec}(R)$ to denote the quantity such that for all $R \geq 1$ and all $N\ge 1$,
    \[\mathrm{RapDec}(R) \leq C_{N}R^{-N}. \]
We use $B^{m}_r(z)$ to denote the ball of radius $r$ centered at $z$ in $\R^m$. When the center is not specified, $B^m_r$ typically refers to the ball of radius $r$ centered at the origin, unless stated otherwise.
We use $N_r(X)$ to denote the $r$-neighborhood of $X$.

\bigskip

\noindent {\bf Choice of parameters:} $0<\e<1$, $R\in[1,\la^{1-\e}]$, $K=R^{\e^{50}}, \Kc=R^{\e^{100}}, \ka=R^{\e^{500}}, \de=\e^{1000}$. 
Under such a choice, we have
\[ R^\de\ll \kappa\ll \Kc\ll K\ll R\ll\la. \]

\bigskip

\noindent
{\bf Acknowledgment.} Danqing He is supported by National Key R$\&$D Program of China (No.2021YFA1002500), NNSF of China (No.12322105), Natural Science Foundation of Shanghai (No.23QA1400300), and the New Cornerstone Science Foundation.
Xiaochun Li is partially supported by NSF2350101.
Shukun Wu is partially supported by NSF2453583 and would like to express his gratitude to Chuanwei Gao for bringing a useful reference to his attention.

\bigskip

\section{A sketch of the proof for Theorem \ref{FIOthm} for odd \texorpdfstring{$n$}{}}
\label{section-sketch}

In this section, we give a sketch of our proof of Theorem \ref{FIOthm}, mostly focusing on the numerology.
For simplicity, we take $\cf=e^{it\sqrt{-\De}}$ in \eqref{FIO-esti}.
After standard global-to-local reductions, it suffices to prove that for $p=p(n)$ and a function $f$ defined on $\ZR^n$ with $\supp\wh f\subset B^n_1\setminus B^n_{1/2}$,
\begin{equation}
\nonumber
    \|\hw f\|_{L^p(B^{n+1}_R)}\lessapprox R^{(n-1)(\frac12-\frac1p)}\|f\|_{p}.
\end{equation}
Let $f=\sum_{T\in\ZT}f_T$ be the wave packet decomposition, where $\ZT$ is a family of $1\times R^{1/2}\times\cdots\times R^{1/2}\times R$-planks.
Let $q_n=\frac{2(n+1)}{n-1}$ be the decoupling exponent.
By dyadic pigeonholing and rescaling, we can assume each wave packet $\hw f_T$ has amplitude $\sim 1$, and hence $\|\hw f_T\|_{L^{q_n}(B_R)}^{q_n}\sim R^{\frac{n+1}{2}}$ for all $T\in\ZT$.
For the purpose of induction, we replace the $L^p$-norm $\|f\|_{p}$ by a mixed norm $\|f\|_{2}^{\frac2p}\cW( f,B_R^{n+1})^{1-\frac2p}$. 
Roughly speaking, the wave packet density $\cW( f,B_R^{n+1})$ is defined by 
    \[ \cW( f,B_R^{n+1}):=\sup_{U}\bigg(\frac{1}{|U|}\int_{\R^{n}} \sum_{T\in\T[U]}|\hw f_T|^2\bigg)^{1/2}. \]
Here, $U$ ranges over all $Rs^2\times Rs\times \dots\times Rs\times R$-planks where $s$ ranges over dyadic numbers in $[R^{-1/2},1]$. 
We also require the longest direction of $U$ points to a light ray direction, and the shortest direction of $U$ points to the normal direction of the cone at that light ray. 
$\T[U]$ is the set of planks contained in $U$ that share the same direction as $U$.
We remark that the $U$ defined here is a higher-dimensional generalization of the one in \cite{guth2020sharp}.

\smallskip

It suffices to show that when $p=p(n)$,
\begin{equation}\label{localwave-2}
   \big \|\sum_{T\in\ZT}\hw f_T\big\|_{L^p(B^{n+1}_R)}^p\lessapprox R^{(n-1)(\frac p2-1)}\|f\|_{2}^{2}\,\cW( f,B_R^{n+1})^{p-2}.
\end{equation}
By testing with $U=T$ and $U=B_R^{n+1}$, we have the following lower bound for $\cW(f,B_R^{n+1})$: 
\begin{equation}
\label{wpd-1}
    \cW(f,B_R^{n+1})^2\gtrsim \max\{ 1, R^{-\frac{n+1}{2}}(\#\T)\}\gtrsim R^{-\frac{n+1}{4}}(\#\ZT)^{1/2}.
\end{equation}

We also record the refined decoupling inequality for cones, whose variable coefficient version is proved in Theorem \ref{refdecthm}.

\begin{theorem}
\label{refdecthm0}      
Suppose we have the wave packet decomposition for $\hw f$ in $B_R^{n+1}$:
    \[\hw f=\sum_{T\in \W} \hw f_T. \]
Assume that $\|f_T\|_{L^p}$ are about the same for all $T\in\W$. Let $Y$ be a disjoint union of unit balls in $B_R^{n+1}$, each of which intersects $\le M$ many $T\in \W$. Then for $2\le p\le \frac{2(n+1)}{n-1}$ and any $\e>0$,
\begin{equation}
   \big \|\hw f\|_{L^p(Y)}\lesssim R^\e M^{\frac12-\frac1p}\Big(\sum_{T\in\W} \|\hw f_T\big\|_{L^p(w_{B_R^{n+1} })}^p\Big)^{1/p}. 
\end{equation}
\end{theorem}

\begin{figure}[ht]
\centering
\includegraphics[width=10cm]{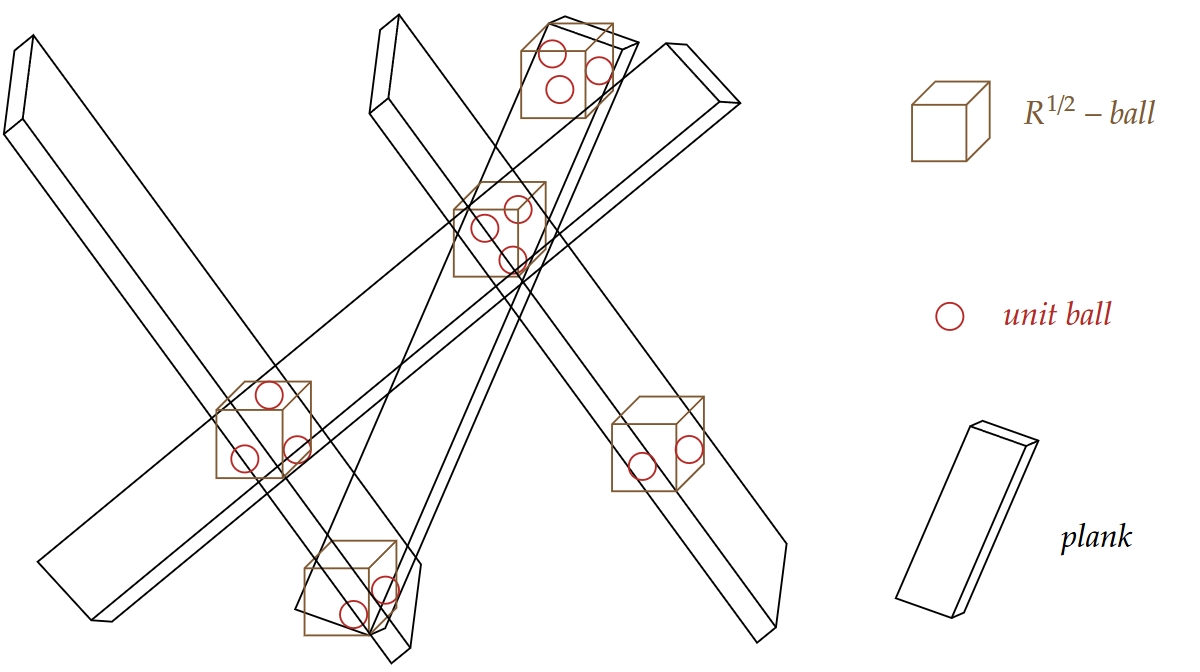}
\caption{}
\label{shading1}
\end{figure}

\smallskip

The refined decoupling result stated in Theorem 2.1 serves as an important endpoint in our interpolation. 
When \(n = 3\), the desired estimates follow from interpolation between the refined decoupling exponent $q_n=\frac{2(n+1)}{n-1}$ and the refined $L^2$ estimates, which are obtained via improved incidence bounds based on both two-end and bilinear structures. 
In higher dimensions, besides the $L^2$ and $L^{q_n}$ spaces, we need to use an additional endpoint $L^{p_{n,k}}$, where the exponent $p_{n,k}$ is associated with the $k_n$-broad estimates. 
It is noteworthy that the cases $n \ge 7$ are the simplest, as no improved incidence estimates are required.
For $n = 5$, an improved incidence estimate arising from 3-broad analysis need to be employed.

\medskip

Let us start with our outline of our proof for the two most difficult cases $n=3$ and $5$.  
By several steps of dyadic pigeonholing, we can obtain 
\begin{enumerate}
    \item[(a)] a union of unit balls $X\subset B^{n+1}_R$,
    \item[(b)] a number $\mu\in[1, R^{(n-1)/2}]$, and a number $m\in[1, R^{n/2}]$,
    \item[(c)] a number $l\in[1,R^{1/2}]$,
    \item[(d)] two numbers $\la,\si\in[1, R^{(n+1)/2}]$,
\end{enumerate}
such that the following is true (see Figure \ref{shading1}):
\begin{enumerate}
    \item $\|\sum_{T\in\ZT}\hw f_T\|_{L^p(B)}$ are about the same for all $B\subset X$, and
    \begin{equation}
    \label{after-pigeonholing-1}
        \|\sum_{T\in\ZT}\hw f_T\|_{L^p(B^{n+1}_R)}^p\lessapprox\|\sum_{T\in\ZT}\hw f_T\|_{L^p(X)}^p.
    \end{equation}
    \item Each $B\subset X$ intersects with $\sim\mu$ many planks in $\ZT$.
    \item For each $R^{1/2}$-ball $Q$ intersecting $X$, $\#\ZT(Q)\sim m$, where $\ZT(Q)\subset\ZT$ is the set of planks intersects $X\cap Q$.
    \item For each plank $T\in\ZT$, the shading $Y(T):=X\cap T$ intersects with $\sim l$ many $R^{1/2}$-balls. 
    \item For each $T\in\ZT$, $|Y(T)|\sim\la$.
    \item For each $R^{1/2}$-ball $Q$ intersecting $X$, $|X\cap Q|\sim \si$.
\end{enumerate}
With the help of the broad \& two-ends algorithm discussed in Sections \ref{section-algorithm} and \ref{section-iteration}, we can further assume that (recall \eqref{defofkn} for $k_n$)
\begin{enumerate}
    \item[(7)] The shading $Y(T)$ satisfies a two-ends condition for all $T\in\ZT$. 
    \item[(8)] For each unit ball $B\subset X$, $|\sum_{T\in\ZT}\hw f_T|\Id_B$ is ``$k_n$-broad", which implies the following: 
    There exist $k_n$ refinements $\ZT_1(B),\cdots,\ZT_{k_n}(B)$ of $\ZT(B)$ such that $T_1,\dots, T_{k_n}$ are quantitatively transverse for all $T_i\in\ZT_i(B)$, $1\le i\le k_n$. %$, T_2\in\ZT_2(B)$.
    Here $\ZT(B):=\{T\in\ZT: Y(T)\cap B\not=\varnothing\}$.
\end{enumerate}
For the case $n=3$, by several more steps of dyadic pigeonholing and a covering lemma, Lemma \ref{regular-lem}, we can additionally obtain 
\begin{enumerate}
    \item[(e)] a number $\rho\in[1,R^{n/2}]$,
\end{enumerate}
such that
\begin{enumerate}
    \item[(9)] For any $Q\in\cQ$, $X\cap Q$ is \textbf{$\rho$-regular} with respect to $1\times R^{1/2}\times R^{1/2}\times R^{1/2}$-slabs, in the following sense: For any $1\times R^{1/2}\times R^{1/2}\times R^{1/2}$-slab $S\subset Q$, 
    \[|X\cap S|\le  \rho,\] 
    and there are 
    \[\sim |X\cap Q|/\rho\sim 
    \si/\rho\]
    many $1\times R^{1/2}\times R^{1/2}\times R^{1/2}$-slabs that cover $X\cap Q$. 
    Denote by $\ZS(Q)=\{S\}$ these slabs that cover $X\cap Q$, so $\#\ZS(Q)\sim\si/\rho$.
\end{enumerate}

Apply the refined decoupling theorem, Theorem \ref{refdecthm0}, to the right-hand side of \eqref{after-pigeonholing-1} so that (recall $q_n=\frac{2(n+1)}{n-1}$),
\begin{equation}
\nonumber
    \|\sum_{T\in\ZT}\hw f_T\|_{L^{q_n}(X)}^{q_n}\lessapprox\mu^{\frac{2}{n-1}}\sum_{T\in\ZT}\|\hw f_T\|_{L^{q_n}(B_R^{n+1})}^{q_n}.
\end{equation}
Since $R^\frac{n+1}{2}
\sim \|\hw f_T\|_{L^{q_n}(B_R^{n+1})}^{q_n}\lesssim R^{-\frac{n+1}{n-1}}\|\hw f_T\|_{L^2(B_R^{n+1})}^{q_n}\lesssim \|f_T\|_2^{q_n}$, by \eqref{wpd-1}, we have
\begin{align}
\nonumber
    \|\sum_{T\in\ZT}\hw f_T\|_{L^{q_n}(X)}^{q_n}&\lessapprox\mu^{\frac{2}{n-1}}(R^{\frac{n+1}{2}}\#\ZT)\\ \label{after-dec-2}
    &\lesssim R^2(R^{-\frac{n-3}{2(n-1)}}\mu^{\frac{2}{n-1}}(\#\ZT)^{-\frac{1}{n-1}})\|f\|_2^2\,\cW(f,B_R^{n+1})^{q_n-2}.
\end{align}

To obtain an upper on $\mu$, we will use the hairbrush structure (see Figure \ref{hairbrush}).
On the one hand, by double-counting, we have
\begin{equation}
\label{l1-1}
    |X|\sim \mu^{-1}\la(\#\ZT).
\end{equation}
On the other hand, pick a plank $T_0\in\ZT$ and define a hairbrush
\begin{equation}
\nonumber
    \ZH(T_0)=\{T\in\ZT: Y(T)\cap Y(T_0)\not=\varnothing, \text{ $T$ and $T_0$ are quantitatively transverse} \}.
\end{equation}
One sees that the planks in $\ZH(T_0)$ are morally disjoint.
Therefore,
\begin{equation}
\label{hairbrush-esti-1}
    |X|\ge |\bigcup_{T\in\ZH(T_0)}Y(T)| \gtrsim\sum_{T\in\ZH(T_0)}|Y(T)| \gtrsim lm\la,
\end{equation}
where we used $|Y(T)|\sim \la$ and $\#\ZH(T_0)\sim lm$.

Combining \eqref{l1-1} and \eqref{hairbrush-esti-1}, we get 
\begin{equation}
\label{mu-1}
    \mu\lesssim l^{-1}m^{-1}(\#\ZT).
\end{equation}
Since $m\geq\mu$, this also implies
\begin{equation}
\label{mu-final-0}
    \mu\lesssim l^{-1/2}(\#\ZT)^{1/2}.
\end{equation}

\medskip

\subsection{When \texorpdfstring{$n=3$}{}}
\label{section-sketch-n=3}

Let $Q\in\cq$ be an $R^{1/2}$-ball.
Recall that $\#\ZT(Q)\sim m$, and $|T\cap Q|\lesssim \rho$ for all $T\in\ZT(Q)$.
By double-counting the incidence between the unit balls in $X\cap Q$ and the $1\times R^{1/2}\times R^{1/2}\times R^{1/2}$-slabs $\{T\cap Q: T\in\ZT(Q)\}$, we obtain
\begin{equation}
\label{L1-esti}
    \mu \si\lesssim \rho m.
\end{equation}

Note that for all $B\subset X$, $|\sum_{T\in\ZT}\hw f_T|\Id_B$ is ``4-broad".
Hence, there exist 4 quantitatively linearly independent directional caps $\{\tau_j\}_{j=1,\ldots, 4}$ such that for all $j=1,\ldots,4$,
\begin{equation}
\nonumber
    \big|\sum_{T\in\ZT}\hw f_T\big|\Id_B\lesssim |\sum_{T\in\ZT_{\tau_j}}\hw f_T|\Id_B.
\end{equation}
Here $\ZT_{\tau_j}:=\{T\in\ZT:\text{ the direction of $T$ is contained in $\tau_j$} \}$.

For each $R^{1/2}$-ball $Q\subset\cq$, recall that $X\cap Q$ has a cover $\ZS(Q)$ by $1\times R^{1/2}\times R^{1/2}\times R^{1/2}$-slabs and $\#\ZS(Q)\sim\si/\rho$.
We claim that for all $S\in\ZS(Q)$
\begin{equation}
\label{l2-broad}
    \bigg\|\sum_{T\in\ZT}\hw f_T\bigg\|_{L^2(S\cap X)}^2\lesssim \sum_{T\in\ZT(Q)}\|f\|_2^2.
\end{equation}
\eqref{l2-broad} is analogous to equation (2.8) in \cite{Li-Wu}.
The proof of \eqref{l2-broad} involves several steps of pigeonholing.
Intuitively, it is based on the following observation:
Since $|\sum_{T\in\ZT}\hw f_T|\Id_B$ is ``4-broad" for each unit ball $B\subset S\cap X$, we can choose a directional cap $\tau_j$ such that $\tau_j$ is quantitatively transverse to the slab $S$.
Let us assume that the cap $\tau_j$ is uniform for all $B\subset S\cap X$.
Thus, as the planks in $\ZT_{\tau_j}$ are all quantitatively transverse to $S$ (see Figure \ref{slab}), by $L^2$-orthogonality, we have
\begin{align}
\nonumber
    &\|\sum_{T\in\ZT}\hw f_T\|_{L^2(S\cap X)}^2\lesssim\|\sum_{T\in\ZT(Q)}\hw f_T\|_{L^2(S\cap X)}^2\\  \nonumber
    \lesssim&\,\|\sum_{T\in\ZT_{\tau_j}\cap\ZT(Q)}\hw f_T\|_{L^2(S\cap X)}^2 \lesssim \|\sum_{T\in\ZT_{\tau_j}\cap\ZT(Q)}\hw f_T\|_{L^2(S)}^2\\ \nonumber
    \lesssim&\,\sum_{T\in\ZT_{\tau_j}\cap\ZT(Q)}\|f_T\|_2^2\lesssim\sum_{T\in\ZT(Q)}\|f_T\|_2^2.
\end{align}
This proves \eqref{l2-broad}.

\begin{figure}[ht]
\centering
\includegraphics[width=5cm]{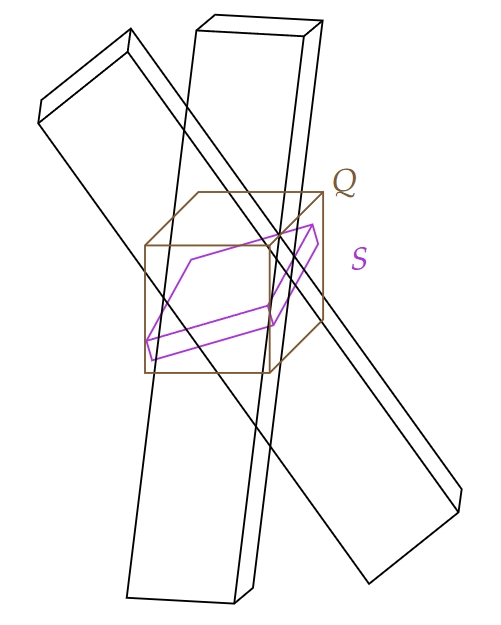}
\caption{}
\label{slab}
\end{figure}

\smallskip

Apply \eqref{l2-broad} to all $S\in\ZS(Q)$ and all $Q\in\cq$.
Since $\#\ZS(Q)\sim\si/\rho$ and since each $T\in\ZT$ intersects $\lesssim l$ many $R^{1/2}$-balls in $\cq$, we have
\begin{equation}
\nonumber
    \|e^{it\sqrt{-\De}}f\|_{L^2(X)}^2\lesssim \sum_{Q\in\cq}\sum_{S\in\ZS(Q)}\|\hw f\|_{L^2(S\cap X)}^2\lesssim l(\si/\rho)\sum_{T\in\ZT}\|f_T\|_2^2,
\end{equation}
which implies the $L^2$ estimate
\begin{equation}
\label{l2-3}
    \|e^{it\sqrt{-\De}}f\|_{L^2(X)}^2\lesssim l(\si/\rho)\|f\|_2^2.
\end{equation}
Interpolate \eqref{after-dec-2} and \eqref{l2-3} by using \eqref{mu-1} and \eqref{L1-esti} to obtain \eqref{localwave-2} for $p=\frac{10}{3}$.

\medskip

\subsection{When \texorpdfstring{$n\geq5$}{}}
\label{section-sketch-ngeq5}
We will obtain \eqref{localwave-2} by interpolating between the cases $p=p_{n,k}=\tfrac{2(n+k_n+1)}{n+k_n-1}$ (recall \eqref{defofkn}), $p=2$,  and $p=q_n=\f{2(n+1)}{n-1}$.

First, we consider the case $p_{n,k}$.
We apply the broad-norm estimates for cones  (Theorem~\ref{k-brd}) obtained in  \cite{Schippa} to essentially establish (see \eqref{k-brd1})
\begin{equation}\label{k-br-in}
    \|\sum_{T\in\ZT}\hw f_T\|_{L^{p_{n,k}}(X)}\lesssim  R^{n(\f12-\f1{p_{n,k}})} \|f\|_{L^2(\mathbb R^n)}^{\f2{p_{n,k}}}\cW(f,B^{n+1}_R)^{1-\f2{p_{n,k}}}.
\end{equation}

Next, we consider the case $p=2$.
For each $R^{1/2}$-ball $Q$, on the one hand we have a simple $L^2$ estimate (see Lemma \ref{local-L2-lem-2})
\begin{equation}
\nonumber
    \|\sum_{T\in\ZT}\hw f_T\|_{L^2(X\cap Q)}^2\lesssim\|\sum_{T\in\ZT}\hw f_T\|_{L^2(Q)}^2\lesssim R^{1/2}\|f\|_2^2
\end{equation}
On the other hand, using $|X\cap Q|\sim \si$, by H\"older's inequality and by applying Theorem \ref{k-brd} locally in $Q$ with $k=k_n$, we have
\begin{equation}
\nonumber
     \|\sum_{T\in\ZT}\hw f_T\|_{L^2(X)}\lesssim \si^{\f1{n+k_n+1}}\|f\|_2.
\end{equation}
Since $Y(T)$ intersects with $\lesssim l$ many $R^{1/2}$-balls for each plank $T\in\ZT$, the above two estimates yield 
\begin{equation}\label{l2-2}
     \|\sum_{T\in\ZT}\hw f_T\|_{L^2(X)}\lesssim l^{1/2}\min (R^{1/4},\si^{\f 1{n+k_n+1}})\|f\|_{L^2}.
\end{equation}

Finally, we discuss the case $q_n=2\f{n+1}{n-1}$.
We will use two types of incidence estimates. 
\eqref{mu-1} is our first upper bound for $\mu$.

As for the second upper bound, we focus on each $R^{1/2}$-ball $Q$.
Recall that $\#\ZT(Q)\sim m$ and $|X\cap Q|\sim\si$ whenever $X\cap Q\not=\varnothing$.
Since each $B\subset X\cap Q$ is $3$-broad regarding the slabs $\{T\cap Q: T\in\ZT(Q)\}$, we have the trilinear  estimate
\begin{equation}
\label{mu-2}
    \mu^{3} \si \lesssim     \sum_{T_1,T_2,T_3\text{ transverse}} \big| T_1\cap T_2\cap T_3\cap Q\big|\lesssim  R^{\frac{n+1-3}{2}}m^3.
\end{equation}
Now \eqref{mu-1} and \eqref{mu-2} give our second estimate on $\mu$:
\begin{equation}
\label{mu-final}
    \mu\lesssim l^{-1/2}(\#\ZT)^{1/2} \bigg(  \frac{R^{\frac{n+1-3}{2}}}{\sigma}\bigg)^{\frac{1}{6}} .
\end{equation}
 
Plug the two estimates \eqref{mu-final-0} and \eqref{mu-final} back to \eqref{after-dec-2} so that
\begin{align}
\nonumber
    & \|\sum_{T\in\ZT}\hw f_T\|_{L^{q_n}(X)}
    \\ \label{after-hairbrush-2}
    \lesssim&\min\Big\{1, 
    \bigg(  \frac{R^{\frac{n+1-3}{2}}}{\si}\!\!\bigg)^{\!\frac{1}{6(n+1)}}\! \Big\}\ell^{-\frac{1}{2(n+1)}} R^{\frac1{q_n}+\f14} \|f\|_{L^2}^{\f2{q_n}}\,\cW(f,B_R^{n+1})^{1-\f2{q_n}}.
\end{align}
Interpolate \eqref{l2-2}, \eqref{after-hairbrush-2}, and \eqref{k-br-in} to obtain \eqref{localwave-2} for $p=2+\frac{8}{3n-3}$ for odd $n$.

\bigskip

\section{Preliminaries}\label{section-preliminary}

In this section, we set up some notations and introduce basic properties for Fourier integral operators.
We denote the variables as $z=(x,t)=(x',x_n,t)\in \R^{n-1}\times \R\times \R$ and $\xi=(\xi',\xi_n)\in\R^{n-1}\times \R$ in the rest of the paper.

\subsection{Reduction to a local estimate}

Recall the Fourier integral operator defined in \eqref{FIO} and the associated assumptions (\textbf{H1}) and (\textbf{H2}).
By considering the new phase function $\wt \phi(z;\xi)=\phi(z;\xi)-\phi(0;\xi)$, we can assume that for all $\al$,
\begin{equation}\label{phi0}
    \partial^\al_\xi\phi(0;\xi)\equiv 0.
\end{equation}

Let $\la\ge 1$.
Given a Fourier integral operator $\cf$ defined in \eqref{FIO}, let 
\begin{equation}
\label{FIOlambda}
    \cFl f(x,t):=\int_{\R^n} e^{i\phi^\la(x,t;\xi)}a^\la(x,t;\xi)\wh f(\xi)\mathrm{d}\xi
\end{equation}
where $\phi^\la$ and $a^\la$ are defined as
    \[\phi^\la(x,t;\xi):=\la\phi(x/\la,t/\la;\xi),\ \  a^{\la}(x,t;\xi):= a(x/\la,t/\la;\xi). \]

By a standard Littlewood-Paley decomposition (see, for example, \cite[Section 1.3]{gao2023square}), and by a smooth partition on the unit sphere $S^{n-1}$, we can assume the amplitude function $a$ obeys that $\supp_\xi\ a\subset\A^n(1)^\circ$, the interior of $\A^n(1)$, where 
    \[\A^n(1):=\{ (\xi',\xi_n)\in\R^n: 1/2\le \xi_n\le 2, |\xi'|\le \xi_n \}. \]
Therefore, Theorem \ref{FIOthm} reduces to the following result:

\begin{proposition}\label{propLpest}
Let $\cf$ be a Fourier integral operator given by \eqref{FIO}.
Suppose the phase function $\phi$ satisfies \eqref{phi0}, and the amplitude function $a$ obeys $\supp_\xi\ a\subset\A^n(1)^\circ$.
Let $\cf^\la$ be given by \eqref{FIOlambda}. 
Then for all $\e>0$ and all $p>p(n)$, \begin{equation}\label{replace}
    \|\cFl f\|_{L^{p}(B_\la^{n+1})}\le C_{\e,p} \la^{(n-1)(\frac12-\frac1p)+\e}\|f\|_{L^p(\R^n)}.
\end{equation}
\end{proposition}

\bigskip

\subsection{Quantitative conditions}\label{subsecquant}

We need to strengthen Proposition \ref{propLpest} to obtain uniform estimates for a class of Fourier integral operators (FIOs). 
This is necessary because we aim to use induction on scales, but an individual operator $\cF$ is not invariant under rescaling. Fortunately, rescaling preserves certain properties of the FIOs, enabling us to perform induction for a broader class of operators simultaneously. 
We will adopt the framework established in \cite[Section 2.3]{beltran2020variable}.

Let $\cpar>0$ be a small fixed constant. 
For $\bA=(A_1,A_2,A_3)\in[1,\infty)^3$, consider the following conditions on the phase function:

\begin{enumerate}
    \item[$(\text{H1}_{\bA})$]  $|\partial^2_{\xi x}\phi(z;\xi)-I_n|\le \cpar A_1$ for all $(z;\xi)\in B^{n+1}_1\times \A^n(1)$.
    \item[$(\text{H2}_{\bA})$] $|\partial^2_{\xi'\xi'}\partial_t\phi(z;\xi)-\frac{1}{\xi_n}I_{n-1}|\le \cpar A_2$ for all $(z;\xi)\in B^{n+1}_1\times \A^n(1)$.
    \item[$(\textup{D1}_\bA)$]$\|\partial_\xi^\beta\partial_{x_k}\phi\|_{L^\infty(B_1^{n+1}\times \A^n(1))}\le \cpar A_1$ for all $1\le k\le n$ and $\beta\in\N_0^n$ with $2\le|\beta|\le 3$ satisfying $|\beta'|\ge 2$;
    
    \noindent $\|\partial_{\xi'}^{\beta'}\partial_{x_k}\phi\|_{L^\infty(B_1^{n+1}\times \A^n(1))}\le \frac{\cpar}{2n} A_1$ for all $\beta'\in\N_0^{n-1}$ with $|\beta'|=3$. 
    \item[$(\textup{D2}_\bA)$] For some large integer $N=N_{\e,M,p}\in\N$ depending only on the dimension $n$ and the fixed choice of $\e,M$ and $p$, one has
        \[\|\partial_\xi^\beta\partial_z^\al\phi\|_{L^\infty(B_1^{n+1}\times \A^n(1))}\le \frac{\cpar}{2n}A_3 \]
    for all $(\al,\beta)\in\N_0^{n+1}\times \N_0^n$ with $2\le |\al|\le 4N$ and $1\le |\beta|\le 4N+2$ satisfying $1\le |\beta|\le 4N$ or $|\beta'|\ge 2$. 
\end{enumerate}

\noindent
Finally, it is useful to assume a margin condition on the spatial support of the amplitude $a$:

\begin{enumerate}
    \item[$(\textup{M}_\bA)$] $\dist(\supp_z a, \R^{n+1}\setminus B_1^{n+1})\ge \frac{1}{4A_3}$.
 
\end{enumerate}

\medskip

Datum $(\phi,a)$ satisfying $(\textup{H1}_\bA),(\textup{H2}_\bA),(\textup{D1}_\bA),(\textup{D2}_\bA)$ and $(\textup{M}_\bA)$ (in addition to (\textbf{H1}) and (\textbf{H2})) is said to be of \textit{type} $\bA$. 
By cutting the amplitude function into pieces and various rescaling arguments, it is possible to reduce to the case where $\bA=\mathbf{1}:=(1,1,1)$. 
We refer to \cite[Section 2.5]{beltran2020variable} for more details.

Throughout the paper, we will assume the FIOs are of type \textbf{1}. 
Proposition \ref{propLpest} can be strengthened as

\begin{proposition}\label{1propLpest}

Let $\cf$ be a Fourier integral operator of type $\bfone$ given by \eqref{FIO}.
Suppose the phase function $\phi$ satisfies \eqref{phi0}, and the amplitude function $a$ obeys $\supp_\xi\ a\subset\A^n(1)^\circ$.
Let $\cf^\la$ be given by \eqref{FIOlambda}. 
Then for all $\e>0$ and all $p>p(n)$, there exists a constant $C_{\e,p}$ independent to $\cf$, such that
\begin{equation}
\label{lp-esti-1}
    \|\cFl f\|_{L^{p}(B_\la^{n+1})}\le C_{\e,p} \la^{(n-1)(\frac12-\frac1p)+\e}\|f\|_{L^p(\R^n)}.
\end{equation}
Here, $C_{\e,p}>0$ is uniform for all $\cF$ of type $\bfone$.
\end{proposition}

\bigskip

\subsection{Wave packet decomposition}

Fix $0<\e<1/100$.
We will perform the wave packet decomposition for a function $f$ with $\supp\wh f\subset\A^n(1)$ at a scale $R\in[1,\la^{1-\e}]$.

By condition (\textbf{H1}), we see that $\partial_\xi\phi(\cdot, t; \xi)$ is a local diffeomorphism.
Indeed, for fixed $t,\xi$, the map $x\mapsto \partial_\xi \phi(x,t;\xi)$ has non-vanishing Jacobian.
By rescaling, we may assume $\partial_\xi\phi(\cdot,t;\xi)$ is a diffeomorphism on the domain $B_1^n$ for any $(t,\xi)\in [-1,1]\times \A^n(1)$. Therefore, it has an inverse function which we defined as follows.

\begin{definition}
    For $(t,\xi)\in [-1,1]\times \A^n(1)$, let $\Phi(u,t;\xi)$ be the unique solution to
    \begin{equation}\label{defPhi}
        \partial_\xi\phi(\Phi(u,t;\xi),t;\xi)=u.
    \end{equation}
Here, we may assume the equation holds for $u$ in a small neighborhood of the origin.
\end{definition}

Since $\phi$ is $1$-homogeneous in $\xi$, we have $\partial_\xi \phi(z;\xi)=\partial_\xi \phi(z;s\xi)$. Hence, 
\begin{equation}
\label{1-homo}
    \Phi(u,t;\xi)=\Phi(u,t;s\xi)
\end{equation}

By the inverse function theorem and conditions $(\textup{D1}_\bA),(\textup{D2}_\bA)$, we obtain favorable upper bounds on the $L^\infty$ norm of $\Phi$ and its derivatives.
Taking the derivative in $u$ in \eqref{defPhi}, we get $\partial^2_{\xi x} \phi\cdot \partial_u\Phi=I_n$.
Thus,
    \[\det\partial_u\Phi(u,t;\xi)\gtrsim 1.\] 
In other words, $\partial_u\Phi(u,t;\xi)$ is a quantitatively non-singular matrix. 

We define 
\begin{equation}
\nonumber
    \ga^\la(u,t;\xi):=\la \Phi(u/\la,t/\la;\xi). 
\end{equation}
Then $\ga^\la$ satisfies
\begin{equation}
\label{transformation-id}
    \partial_\xi\phi^\la(\ga^\la(u,t;\xi),t;\xi)=u.
\end{equation}
For fixed $(u;\xi)$, $t\mapsto (\ga^\la(u,t;\xi),t)$ is a curve, with $t$ ranging over $[-R,R]$.
Our wave packets will be certain non-isotropic neighborhoods of such curves.

\medskip

Now we present the wave packet decomposition for $\cFl f$ in $B_R^{n+1}$. 
In fact, we will perform the wave packet decomposition for $f$: $f=\sum_{T}f_T$ first. 
Then, by applying the linear operator $\cFl$, we get $\cFl f=\sum_T \cFl f_T$.

We begin by decomposing in frequency space.
Fix a maximally $R^{-1/2}$-separated subset of $B^{n-1}_1\times \{1\}$.
For each $\xi_\theta$ belonging to this subset, define
\begin{equation}
\label{partiiton-of-annulus}
    \theta=\{ (\xi',\xi_n)\in \A^n(1): |\xi'/\xi_n-\xi_\theta|\le R^{-1/2} \}.
\end{equation}
Then $\theta$ is roughly a $1\times R^{-1/2}\times\dots\times R^{-1/2}$-box, and $\xi_\theta$ is the center of $\theta$. 
We call such $\theta$ an \textbf{$R^{-1/2}$-cap} and denote the collection of these $\theta$ by $\Theta_{R^{-1/2}}$.
Note that $\Theta_{R^{-1/2}}$ form a finitely overlapping covering of $\A^{n}(1)$.
Choose $\{\psi_\theta(\xi)\}_{\theta\in\Theta_{R^{-1/2}}}$ to be a smooth partition of unity adapted to $\Theta_{R^{-1/2}}$, so that $\sum_{\theta\in\Theta_{R^{-1/2}}}\psi_\theta(\xi)=1$ for $\xi\in\A^n(1)$. 
Since $\supp\wh f\subset \A^n(1)$, we have the frequency decomposition
    \[\wh f =\sum_{\theta\in\Theta_{R^{-1/2}}}\psi_\theta \wh f. \]
We denote $f_\theta:=(\psi_\theta\wh f)^\vee$, and hence
\[ f=\sum_{\theta\in\Theta_{R^{-1/2}}}f_\theta. \]
\smallskip

Before performing decomposition in the physical space, we introduce a notation.
\begin{definition}
\label{dual-box}
Let $U\subset \R^n$ be a box of dimensions $a_1\times \dots\times a_n$.
Define $U^*$ to be the {\bf dual box} of $U$, which is a box centered at the origin of dimensions $a_1^{-1}\times \dots\times a_n^{-1}$, and with edges parallel to the corresponding edges of $U$.
\end{definition}

Now for a fixed $\theta\in\Theta_{R^{-1/2}}$, we cover $B_R^n$ by $1\times R^{1/2}\times\dots\times R^{1/2}$-slabs that are parallel to $\theta^*$.
Denote this collection of slabs by  
\begin{equation}
\label{T-flat}
    \T_\theta^\flat=\{T^\flat\}.
\end{equation}
Choose a set of non-negative smooth functions $\{\eta_{T^\flat}\}$ adapted to $\T^\flat_\theta$, so that $\eta_{T^\flat}$ decays rapidly outside $T^\flat$, $\sum \eta_{T^\flat}=1$ in $B_R^n$ and $\supp\wh \eta_{T^\flat}\subset \theta$.
Thus, inside $B_R^n$, we have
    \[\wh f=\sum_{\theta\in\Theta_{R^{-1/2}}}\sum_{T^\flat\in\T^\flat_\theta}\wh\eta_{T^\flat}*(\psi_\theta \wh f)+\rap(R)\|f\|_2. \]

\begin{figure}[ht]
\centering
\includegraphics[width=10cm]{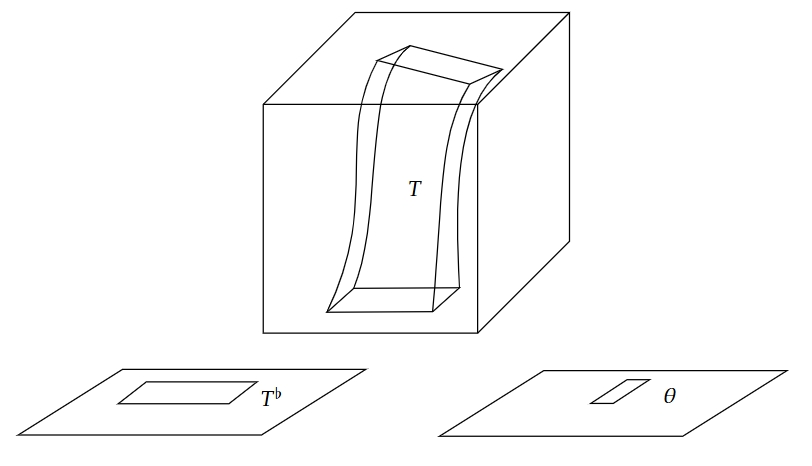}
\caption{Wave packet}
\label{wavepacket}
\end{figure}

Let $\de=\e^{1000}$, which is a small number to handle rapidly decaying tails. 
For each $T^\flat$, define
\begin{equation}
\label{single-wpt}
    T:=\{ (\ga^\la(u,t;\xi_\theta),t): u\in  R^\de T^\flat, |t|\le R \}.
\end{equation}
Here, $R^\de T^\flat$ is a non-isotropic dilate of $T^\flat$ of dimensions $R^{2\de}\times R^{1/2+\de}\times \dots\times R^{1/2+\de}$. Each $T$ is a ``curved plank" (see Figure \ref{wavepacket}).   
We denote $\T_\theta:=\{T: T^\flat\in\T^\flat_\theta\}$ and
    \[f_T:=\eta_{T^\flat}(\psi_\theta^\vee *f). \]
This gives the wave packet decomposition in $B_R^n$
    \[f=\sum_{\theta\in\Theta_{R^{-1/2}}}\sum_{T\in\T_\theta}f_T+\rap(R)\|f\|_2. \]
Therefore, in $B_R^{n+1}$, we obtain
    \[\cFl f=\sum_{\theta\in\Theta_{R^{-1/2}}}\sum_{T\in\T_\theta}\cFl f_T+\rap(R)\|f\|_2. \]

\begin{definition}[Curved plank]
\label{defgaV}
Given $\xi_0\in\A^n(1)$ and a box $V\subset B_R^n$, define
\begin{equation}\label{defgaVeq}
    \Ga_V(\xi_0;R):=\{ (\ga^\la(u,t;\xi_0),t):u\in V,|t|\le R \}.
\end{equation}
$\Ga_V(\xi_0;R)$ is called a curved plank with base $V$ and direction $\xi_0$.
\end{definition}

We primarily focus on the case where $V$ is a slab with dimensions $Rs^2 \times Rs \times \dots \times Rs$ for $s \in [R^{-1/2}, 1]$.
Next, we define the non-isotropic dilation for such slabs.

\begin{definition}
For a slab $U\subset \R^n$ of dimensions $Rs^2\times Rs\times \dots\times Rs$ with $s\in[R^{-1/2},1]$ and a constant $C>0$, define $CU$, the non-isotropic $C$-dilation of $U$, to be a slab of dimensions $C^2Rs^2\times CRs\times \dots\times CRs$ with the same center as $U$.

For two slabs $U_1, U_2\subset \R^n$, we say they are \textbf{comparable} if 
    \[\frac{1}{C}U_1\subset U_2\subset C U_1.\] 
Here, $C=C_\phi$ is a constant that may vary from line to line, but eventually only depends on $\phi$. 
%If $U_1, U_2$ are not comparable, we say they are \textbf{essentially distinct}. We may identify sets that are comparable. Therefore, we can also talk about the dual box for any convex set. 
\end{definition}

\begin{remark}
    {\rm We will also need the isotropic dilation. Namely, for a $a_1\times \dots \times a_n$-box $U$, the isotropic $C$-dilation of $U$ will be the box of dimensions $Ca_1\times \dots \times Ca_n$ with the same center as $U$. We will always specify when doing isotropic dilation; otherwise, the notation $CU$ always refer to the non-isotropic dilation.}
\end{remark}

We also define the dilation and comparability for curved planks. 
Since a curved plank is not convex, we will define these notions in terms of its base.

\begin{definition}
\label{distinct}
For a curved plank $\Ga_V(\xi_0;R)$ and a constant $C>0$, we define the (non-isotropic) $C$-dilation of $\Ga_V(\xi_0;R)$ to be
    \[C\Ga_V(\xi_0;R):=\{ (\ga^\la(u,t;\xi_0),t):u\in CV,|t|\le R  \}. \]
For two curved planks $\Ga_{V_1}(\xi_1;R), \Ga_{V_2}(\xi_2;R)$, we say they are \textbf{comparable} if 
\[\frac{1}{C}\Ga_{V_1}(\xi_1;R)\subset \Ga_{V_2}(\xi_2;R)\subset C \Ga_{V_1}(\xi_1;R)\] 
for some constant $C$ only depending on $\phi$. 
If $\Ga_{V_1}(\xi_1;R), \Ga_{V_2}(\xi_2;R)$ are not comparable, we say they are \textbf{distinct}. 
We may identify curved planks that are comparable.
\end{definition}

The next lemma concerns the geometry of curved planks.

\begin{lemma}\label{geometric}
Let $s\in[R^{-1/2},1]$. 
Let $V\subset B_R^n$ be a slab of dimensions $Rs^2\times Rs \times \dots\times Rs $. 
Then for $|t_0|\leq R$, each $t_0$-slice $\Ga_V(\xi_0;R)\cap \{t=t_0\} $ is comparable to an $Rs^2\times Rs\times \dots\times Rs$-slab.
\end{lemma}
   
\begin{proof}
Suppose $V$ is centered at $u_0$. 
We express each point in $V$ as $u_0+u$, where $u$ ranges over $V-u_0$, an $Rs^2\times Rs\times \dots\times Rs$-box centered at the origin. 

For fixed $(t_0,\xi_0)$, by applying Taylor's expansion in the $u$ variable, we have
\begin{equation}\label{taylorexpansion}
    \ga^\la(u_0+u,t_0;\xi_0)=\ga^\la(u_0,t_0;\xi_0)+\partial_u\ga^\la(u_0,t_0;\xi_0)\cdot u+O(\|\partial^2_{uu}\ga^\la(\cdot,t_0;\xi_0)\|_\infty)|u|^2. 
\end{equation} 
Since $\partial^2_{uu}\ga^\la(u_0,t_0;\xi_0)=\la^{-1}\partial^2_{uu}\Phi(u_0/\la,t_0/\la;\xi_0)$, whose $L^\infty$ norm is  $O(\la^{-1})$, 
    \[O(\|\partial^2_{uu}\ga^\la(\cdot,t_0;\xi_0)\|_\infty)|u|^2\lesssim \la^{-1}(Rs)^2\le Rs^2. \]  
Therefore, the sets $\Ga_V(\xi_0;R)\cap\{t=t_0\}$ and
    \[\ga^\la(u_0,t_0;\xi_0)+\{\partial_u\ga^\la(u_0,t_0;\xi_0)\cdot x:x\in V-u_0\} \]
are within a distance $\le Rs^2$ from each other.
Since $\partial_u\ga^\la(u_0,t_0;\xi_0)$ is quantitative non-singular, the latter is comparable to $V-u_0$ which is a $Rs^2 \times Rs\times \dots\times Rs$-slab.
Therefore, $\Ga_V(\xi_0;R)\cap \{t=t_0\}$ is also comparable to such a slab.
\qedhere

\end{proof}

\begin{remark}
    {\rm  From the above lemma, we see that in Figure \ref{wavepacket}, each horizontal slice of $T$ is a slab of dimensions $R^{2\de}\times R^{1/2+\de}\times \dots\times R^{1/2+\de}$. Though, they may not be parallel.}
\end{remark}

The next lemma concerns the perturbation of a curved plank resulting from choosing different directions.
It was essentially proved in \cite{gao2023square}.  

\begin{lemma}\label{lemuncertainty}
Let $s\in[R^{-1/2}, 1]$, and let $\tau\in\Theta_{s}$. 
Define $V_{\tau,R}:=(Rs^2)\tau^*$ to be an isotropic dilate of $\tau^*$, which is an $Rs^2\times Rs\times \dots\times Rs$-box.
Let $V\subset B_R^n$ be a translated copy of $V_{\tau,R}$. 
Then for any $\xi_1,\xi_2\in\tau$, the two curved planks $\Ga_V(\xi_1;R)$ and $\Ga_V(\xi_2;R)$ are comparable. 
In other words, $\Ga_V(\xi_1;R)\subset \Ga_{CV}(\xi_2;R)$ for some constant $C>0$ and vice versa.
\end{lemma}

\begin{proof}
It suffices to show that there exists a large constant $C>0$ such that
\begin{equation}\label{toshowinclusion}
     \{\ga^\la(u,t;\xi_2):u\in V\}\subset \{\ga^\la(u,t;\xi_1):u\in CV\}. 
\end{equation}

Let $u^\la(z;\xi):=\partial_\xi\phl(z;\xi)$, so $u^\la(\gal (u,t;\xi),t;\xi)=u$ (recall \eqref{transformation-id}). 
To obtain \eqref{toshowinclusion}, we only need to show that for any $u\in V$,
    \[u^\la(\gal(u,t;\xi_2),t;\xi_1)\subset CV. \]
Since $u^\la(\gal(u,t;\xi_2),t;\xi_2)=u\in V$, it suffices to show
\begin{equation}\label{uncertainty}
    u^\la(z;\xi_1)-u^\la(z;\xi_2)\subset 4V_{\tau,R}.
\end{equation} 
for all $z\in B_R^{n+1}$ and all $\xi_1,\xi_2\in \tau$. 
This is true by Lemma 4.3 in \cite{gao2023square}.
\qedhere

\end{proof}

Using the notation of curved plank, each wave packet $T\in \T_\theta$ can be written as
    \[T= \Ga_{R^\de T^\flat}(\xi_\theta;R), \]
a curved plank with base $R^\de T^\flat$ and direction $\xi_\theta$. We call each such $T$  an \textbf{$R$-plank}.
Note that each $t$-slice $T\cap \{t=t_0\}$ of $T$ is morally an $R^{2\de} \times R^{1/2+\de}\times \dots\times R^{1/2+\de}$-slab. 
Moreover, $\Ga_{R^\de T^\flat}(\xi_\theta;R)$ and $\Ga_{R^{\de}T^\flat}(\xi;R)$ are comparable for all $\xi\in\theta$.

\medskip

\begin{lemma}\label{lemsuppT}
    $(\cFl f_T)\Id_{B_R^{n+1}}$ is essentially supported in $ T$. In other words, \[(\cFl f_T)\Id_{B_R^{n+1}\setminus  T}=\rap(R)\|f\|_2.\]
\end{lemma}

\begin{proof}
Note that
    \[\cFl f_T(x,t)=\int_{\R^n} e^{i\phi^\la(x,t;\xi)}a^\la(x,t;\xi) \wh\eta_{T^\flat}*(\psi_\theta\wh f)\mathrm{d}\xi. \]
Since $a^\la$ is a smoothing symbol and $\wh f_T$ is supported in $2\theta$, we have

    \[\cFl f_T(x,t)= \int_{\R^n} \Big(e^{i\phi^\la(x,t;\xi)} a^\la\wt\psi_\theta\Big) \wh\eta_{T^\flat}*(\psi_\theta\wh f)\mathrm{d}\xi+\rap(R)\|f\|_2,  \]
where $\wt \psi_\theta$ is a bump function on $2\theta$. 
Let $G_{x,t}(\xi)$ be so that $\overline{G_{x,t}(\xi)}=e^{i\phi^\la(x,t;\xi)} a^\la\wt\psi_\theta$.
By Plancherel,
\begin{equation}\label{planch}
    \cFl f_T(x,t)= \int_{\R^n} \overline{\check G_{x,t} (y)} \Big(\eta_{T^\flat}(\psi_\theta^\vee *f)\Big)(y) \mathrm{d}y+\rap(R)\|f\|_2. 
\end{equation} 

Let us compute 
\[\overline{\check G_{x,t}(y)}=\int_{\R^n} e^{i(\langle y,\xi\rangle-\phi^\la(x,t;\xi))} \overline{a^\la\wt\psi_\theta} \mathrm{d}\xi.\]
Via the change of variable $\xi=\xi_\theta+\eta$, we have
\begin{equation}
\nonumber
    \overline{\check G_{x,t}(y)}=\int_{\R^n} e^{i(\langle y,\xi_\theta+\eta\rangle-\phi^\la(x,t;\xi_\theta+\eta))} \overline{a^\la(z;\xi_\theta+\eta)\wt\psi_\theta(\xi_\theta+\eta)} \mathrm{d}\xi. 
\end{equation} 
By the method of non-stationary phase (see  \cite[Lemma 5.4]{guth2019sharp} for a similar argument),  $|\check G_{x,t}(y)|=\rap(R)$ if 
    \[\partial_\xi\phl(x,t;\xi_\theta+\eta)-y\notin \theta^*. \]
Recall \eqref{uncertainty}, which implies that $\partial_\xi\phl(x,t;\xi_\theta+\eta)-\partial_\xi\phl(x,t;\xi_\theta)\in4\theta^*$.
Therefore,  $|\check G_{x,t}(y)|=\rap(R)$ if 
\begin{equation}\label{Gy}
    \partial_\xi\phl(x,t;\xi_\theta)-y\notin 5\theta^*. 
\end{equation} 

\smallskip

Now we are ready to prove the lemma. If $(x_0,t_0)\in B_R^{n+1}\setminus  T$, then 
\[(x_0,t_0)\notin (T\cap \{t=t_0\})=\{\ga^\la(u,t_0;\xi_\theta): u\in R^{\de} T^\flat\}.\]
By applying $\partial_\xi \phi^\la(\cdot,\cdot;\xi_\theta)$ to both sides  (recall \eqref{transformation-id}), we see that the above is a consequence of
\begin{equation}\label{notinT}
    \partial_\xi \phi^\la(x_0,t_0;\xi_\theta)\notin R^{\de}T^\flat. 
\end{equation} 
We use \eqref{planch} to show $\cF^\la f_T(x_0,t_0)=\rap(R)\|f\|_2$:
Since $\eta_{T^\flat}$ decreases rapidly outside $R^{\de/2} T^\flat$, it suffices to show $|\check G_{x_0,t_0}(y)|=\rap(R)$ for $y\in R^{\de/2} T^\flat$.
This boils down to verifying \eqref{Gy}, or equivalently,  $\partial_\xi\phi^\la(x_0,t_0;\xi_\theta)\notin y+5\theta^*\subset 5  R^{\de/2}T^\flat$, since $T^\flat$ is a translated copy of $\theta^*$ and since $y\in T^\flat$. 
However, this is given by \eqref{notinT}.
\end{proof}

\bigskip

\subsection{Comparing wave packets at different scales}

Next, we discuss the wave packet decomposition for $\cFl f$ inside a ball not necessarily centered at the origin. 
Fix $z_0=(x_0,t_0)\in B^{n+1}_\la$ and consider the ball $B_R^{n+1}(z_0)$. 
Define $\wt f$ so that
    \[(\wt f)^\wedge (\xi):=e^{i\phi^\la(z_0;\xi)}\wh f(\xi). \]
As a result,
    \[\cFl f(z)=\wt \cF^\la \wt f(\wt z)\ \textup{~for~}\wt z=z-z_0, \]
where $\wt \cF^\la$ is the Fourier integral operator whose phase $\wt\phi^\la$ and amplitude $\wt a^\la$ are given by
    \[\wt \phi(z;\xi):=\phi(z+\frac{z_0}{\la};\xi)-\phi(\frac{z_0}{\la};\xi)\textup{~and~}\wt a(z;\xi):=a(z+\frac{z_0}{\la};\xi). \]
If $z\in B^{n+1}_R(z_0)$, then $\wt z\in B^{n+1}_R$, and we can therefore apply the wave packet decomposition for $\wt \cF^\la \wt f$ inside $B^{n+1}_R$:
    \[\wt \cF^\la \wt f=\sum_{\wt T} \wt \cF^\la \wt f_{\wt T}+\rap(R)\|f\|_2. \]
Each $\wt T$ is given by the collection of $\wt z\in B^{n+1}_R$ satisfying
    \[\partial_\xi\phi(\frac{\wt z}{\la}+\frac{z_0}{\la};\xi_\theta)=\frac{v}{\la}+\partial_\xi\phi(\frac{z_0}{\la};\xi_\theta), \textup{~for~some~}v\in R^\de T^\flat. \]
Recall \eqref{T-flat} that $\{T^\flat\}$ are translated copies of $\theta^*$ that cover $B_R^n$.

\smallskip

For a given $z_0$, define
\begin{equation}\label{defT}
    T(z_0):=\Big\{\Big(\ga^\la(v+\partial_\xi\phi^\la(z_0;\xi_\theta),t;\xi_\theta),t\Big): v\in R^\de T^\flat, |t-t_0|\le R \Big\}. 
\end{equation}

\begin{remark}\label{replacedirection}
{\rm
If we replace $\xi_\theta$ by any $\xi\in \theta$ in the definition of $T(z_0)$ in \eqref{defT}, then the resulting curved plank is comparable to the original curved plank defined using $\xi_\theta$.
This is a consequence of Lemma \ref{lemuncertainty}.
We will use this fact later when comparing wave packets in different scales.

}
\end{remark}

Thus, the function
    \[\wt \cF^\la \wt f_{\wt T}(z-z_0) \]
is essentially supported in $ T(z_0)$ when restricted to $B^{n+1}_R(z_0)$.
We denote the collection of planks $\{T(z_0): T^\flat\in\T^\flat_\theta\}$ as 
\begin{equation}
\label{T-theta-z0}
    \T_\theta (z_0):=\{T(z_0): T^\flat\in\T^\flat_\theta\}.
\end{equation}

For each $T\in\bigcup_\theta\T_\theta(z_0)$, we also define
\begin{equation}
\label{f-T-z0}
    \wh f_T(\xi):=e^{-i\phi^\la(z_0;\xi)}(\wt f_{\wt T})^\wedge(\xi), 
\end{equation}
where $T^\flat=\wt T^\flat$.
Under this notation,  
    \[\cFl f=\sum_{T\in\bigcup_\theta \T_\theta(z_0)}\cFl f_T+\rap(R)\|f\|_2 \]
in $B^{n+1}_R(z_0)$. 
For each $T\in\T_\theta(z_0)$, we remark that $(\cFl f_T) \Id_{B_R^{n+1}(z_0)}$ is essentially supported in $R^\de T$, $f_T$ is essentially supported in $T^\flat$, and $\wh f_T$ is supported in $2\theta$. 

\smallskip

We have the following lemma regarding the comparison of wave packets at different scales. Its counterpart can be found in \cite[Lemma 9.1]{guth2019sharp}.
\begin{lemma}\label{lemcomparewp}
Let $r\in[1,R)$, and let $\theta\in\Theta_{R^{-1/2}}$, $\wt \theta\in \Theta_{r^{-1/2}}$. 
Suppose $B_{r}^{n+1}(z_0)\subset B_R^{n+1}$, and suppose $T\in\T_\theta$, $\wt T\in\T_{\wt\theta}(z_0)$. 
If either  $\theta \not\subset 4\wt\theta$, or $100T\cap 100\wt T=\emptyset$, then
    \[\Big(f_T\Big)_{\wt T}=\rap(R) \|f\|_2. \]
\end{lemma}

\begin{proof}
For simplicity, we denote $e^{i\phl(z_0;D)}g:= (e^{i\phl(z_0;\cdot)}\wh g)^\vee$.
Recall \eqref{f-T-z0} for the definition of the wave packet decomposition inside $B_r^{n+1}(z_0)$, we have
    \[ e^{i\phl(z_0;D)} \Big(f_{T}\Big)_{\wt T}=\eta_{\wt T^\flat}\Big(\psi_{\wt\theta}^\vee* (e^{i\phl(z_0;D)}f_{T})\Big).\]
Note that $\supp(e^{i\phl(z_0;D)}f_{T})^\wedge=\supp \wh f_{T}$, which is essentially contained in $ 2 \theta$.
Since $\supp\psi_{\wt\theta}\subset \wt\theta$, if $\theta\not\subset 4\wt\theta$, we have $ e^{i\phl(z_0;D)} \Big(f_{T}\Big)_{\wt T}=\rap(R)\|f\|_2$, which implies $\Big(f_{T}\Big)_{\wt T}=\rap(R)\|f\|_2$.

\smallskip

Next, we assume $\theta\subset 4\wt \theta$ and $10T\cap \wt T=\emptyset$. 
Compute
\begin{align*}
    e^{i\phl(z_0;D) }f_T(x)&=\int_{\R^n} e^{i(\phl(z_0;\xi)+\langle x,\xi\rangle)} \wh \eta_{T^\flat}*(\psi_\theta \wh f)\mathrm{d}\xi\\
    &=\int_{\R^n} \Big(e^{i(\phl(z_0;\xi)+\langle x,\xi\rangle)}\wt \psi_\theta\Big) \wh \eta_{T^\flat}*(\psi_\theta \wh f)\mathrm{d}\xi+\rap(R)\|f\|_2,
\end{align*}
where $\wt \psi_\theta$ is a bump function on $2\theta$. 
Introduce $G_{z_0,x}(\xi)$ such that $\overline{G_{z_0,x}(\xi)}=e^{i(\phl(z_0;\xi)+\langle x,\xi\rangle)}\wt \psi_\theta$.
By Plancherel,
\begin{equation}\label{planch2}
    e^{i\phl(z_0;D) }f_T(x)=\int_{\R^n}\overline{\check G_{z_0,x}(y)} \Big( \eta_{T^\flat}(\psi_\theta^\vee*f) \Big)(y)\mathrm{d}y+\rap(R)\|f\|_2.
\end{equation}
Let us compute
    \[\check G_{z_0,x}(y)=\int_{\R^n}e^{i(-\phl(z_0;\xi)+\langle y-x,\xi\rangle )}\wt\psi_\theta\mathrm{d}\xi. \]
By the same argument as in the proof of Lemma \ref{lemsuppT}, $|\check G_{z_0,x}(y)|=\rap(R)$ if
\begin{equation}\label{Hy}
    \partial_\xi\phl(z_0;\xi_\theta)+x-y\notin 5\theta^*.
\end{equation}
To show $\eta_{\wt T^\flat}\Big(\psi_{\wt\theta}^\vee* (e^{i\phl(z_0;D)}f_{T})\Big)=\rap(R)\|f\|_2$, we just need to show $\eqref{planch2}=\rap(R)\|f\|_2$ for $x\in r^{\de/2}\wt T^\flat$.
This boils down to show $|\check G_{z_0,x}(y)|=\rap(R)$ for $x\in r^{\de/2}\wt T^\flat, y\in R^{\de/2}T^\flat$.
By \eqref{Hy}, it suffices to show 
\begin{equation}
\label{lem-wpt-2-scales-eq1}
    \partial_\xi\phl(z_0;\xi_\theta)+x-y\notin \theta^*\ \textup{~for~all~}x \in r^{\de/2}\wt T^\flat, y\in R^{\de/2}T^\flat.
\end{equation}

Since $T^\flat$ is a translated copy of $\theta^*$, \eqref{lem-wpt-2-scales-eq1} follows from  
    \[2R^{\de/2} T^\flat\cap (\partial_\xi\phl(z_0;\xi_\theta)+r^{\de/2}\wt T^\flat)=\emptyset. \]
By applying the transformation $\partial_\xi(\cdot,t;\xi_\theta  )$  (recall \eqref{transformation-id}), the above is a consequence of the fact that
$4T$ and 
    \[\{ \Big(\partial_{\xi}\phl(v+\partial_\xi\phl(z_0;\xi_\theta),t;\xi_\theta),t\Big): v\in r^{\de/2}\wt  T^\flat \}\] are disjoint.
However, by Remark \ref{replacedirection},  the latter one is contained in $100\wt T$.
Thus, if $4T\cap 100\wt T=\emptyset$, then the aforementioned two planks are disjoint, yielding \eqref{lem-wpt-2-scales-eq1}.
\qedhere

\end{proof}

\bigskip

\subsection{Broad norm }\label{sectionbroadnorm}
We introduce the broad norm in the variable coefficient setting. 
A similar setup can be found in \cite[Section 1.5]{guth2019sharp}. 

Let $K=R^{\e^{50}}$.
Recall \eqref{partiiton-of-annulus} that  $\A^{n}(1)$ can be partitioned into caps $\Theta_{K^{-1}}=\{\tau\}$, where each $\tau$ has dimensions $1\times K^{-1}\times\dots\times K^{-1}$. 
In view of the rescaling $\phi^\la$ of the phase function, define the rescaled generalized Gauss map
\begin{equation}
\label{generalized-gauss-map}
    G^\la(z;\xi):=G(z/\la;\xi)\ \textup{~for~}(z;\xi)\in\supp\  a^\la. 
\end{equation}
For each $z\in B_\la^{n+1}$, there is a range of normal directions associated with the cap $\tau$, as given by
    \[G^\la(z;\tau):=\{ G^\la(z;\xi):\xi\in\tau, (z;\xi)\in\supp\  a^\la \}. \]
For any subspace $V\subset \R^{n+1}$, define $\ang (G^\la(z;\tau),V)$ to be the smallest angle between non-zero vectors $v\in V$ and $v'\in G^\la(z;\tau)$.

Fix $A\ge 1$ and $1\le k\le n+1$.
For a $K^2$-ball $B_{K^2}^{n+1}\subset B_\la^{n+1}$ centered at $z_0$, define
    \[\mu_{\cFl f}(B_{K^2}^{n+1}):=\min_{V_1,\dots,V_A}\max_{\tau\notin V_i}\int_{B_{K^2}^{n+1}}|\cFl f_\tau|^p. \]
Here, $V_1,\dots,V_A$ are taken over all $(k-1)$-dimensional subspaces, and $\tau\notin V_i$ refers to those $\tau\in\Theta_{K^{-1}}$ such that $\ang(G^\la(z_0;\tau),V_i)>K^{-2}$ for all $1\le i\le A$.
If $U\subset B_\la^{n+1}$ is a disjoint union of $K^{2}$-balls, then we define the broad norm as
\begin{equation}
\nonumber
    \|\cFl f\|_{\BL^p_{k,A}(U)}^p:=\sum_{B_{K^2}^{n+1}\subset U}\mu_{\cFl f}(B_{K^2}^{n+1}).
\end{equation}

\smallskip

The broad norm enjoys the following properties.
See \cite{guth2019sharp} for their proofs.

\begin{lemma}\label{lemtriangle}
Suppose $1\le p<\infty$, $A_1,A_2\ge 1$ and $A=A_1+A_2$, then
    \[\|\cFl(f_1+f_2)\|_{\BL^p_{k,A}(U)}\lesssim \|\cFl f_1\|_{\BL^p_{k,A_1}(U)}+\|\cFl f_2\|_{\BL^p_{k,A_2}(U)}. \]
\end{lemma}

\begin{lemma}\label{lemholder}
Suppose $1\le p,p_1,p_2<\infty$, $0\le \al_1,\al_2\le 1$ satisfy $\al_1+\al_2=1$ and 
    \[\frac1p=\frac{\al_1}{p_1}+\frac{\al_2}{p_2}. \]
Suppose also $A_1,A_2\ge 1$ and $A=A_1+A_2$.
Then,
\begin{equation}
\nonumber
    \|\cFl f\|_{\BL^p_{k,A}(U)}\lesssim \|\cFl f\|_{\BL^{p_1}_{k,A_1}(U)}^{\al_1}\|\cFl f\|_{\BL^{p_2}_{k,A_2}(U)}^{\al_2}.
\end{equation}
\end{lemma}

\bigskip

\subsection{Some \texorpdfstring{$L^2$}{} results}

The first one is a local $L^2$ estimate for a sum of wave packets whose directions are transverse to the integration domain.

\begin{lemma}
\label{local-L2-lem-1}
Let $1\leq K^{1000}\leq R$.
Let $S$ be a $1\times R^{1/2}\times \cdots\times R^{1/2}$-slab in $\ZR^{n+1}$ parallel to the $n$-dimensional subspace $V$. 
Let $\tau\in\Theta_{K^{-1}}$. 
Suppose there exists $z_0\in S$ such that 
    \[\ang (G^\la(z_0;\tau),V)>K^{-2}. \] 
Then for any set of distinct $R$-planks $\bT$ that $\theta(T)\subset \tau$ for all $T\in\bT$, we have
\begin{equation}
\nonumber
    \| \sum_{T\in\bT}\cFl f_T \|_{L^2(S)}^2\lesssim K^{O(1)} \sum_{T\in\bT} \|f_T\|_2^2.
\end{equation}
\end{lemma}
\begin{proof}
We refer the reader to Figure \ref{slab}, where the planks are curved rather than straight.
Since we are allowed to lose a $K^{O(1)}$ factor, by the triangle inequality, we assume the direction caps of all the $T\in\bT'$ lie in a $K^{-10n}$-cap.
Let $\bT=\cup_\theta\bT_\theta$, where the directional cap for the planks in $\bT_\theta$ is $\theta$.
Let $f_\theta=\sum_{T\in\bT_\theta}f_T$,
so that 
\[\sum_{T\in\bT}\cFl f_T=\sum_\theta \cFl f_\theta.\]
We will show 
\begin{equation}
\label{l2-ortho}
    \| \sum_\theta \cFl f_\theta \|_{L^2(S)}^2\lesssim K^{O(1)} \sum_\theta\|f_\theta\|_2^2.
\end{equation}
Then the lemma follows from the $L^2$-orthogonality of $\sum_{T\in\bT_\theta}f_T$.

\smallskip

Let $\Id_S^\ast$ be a bump function that equals 1 on $S$ and is supported on $2S$.
By expanding the $L^2$ norm, the left-hand side of \eqref{l2-ortho} is bounded above by
\begin{equation}
\label{expanding}
    \sum_{\theta_1,\theta_2}\int \Id_S^\ast e^{i(\phl(z;\xi_1)-\phl(z;\xi_2)}\mathrm{d}z\int a^\la(z;\xi_1)\overline{a^\la(z;\xi_2)}\wh f_{\theta_1}(\xi_1)\overline{\wh f_{\theta_2}(\xi_2)}\mathrm{d}\xi_1\mathrm{d}\xi_2.
\end{equation}
Note that \eqref{l2-ortho} follows from the estimate
\begin{equation}
\label{stationary-phase}
    \big|\int \Id_S^\ast e^{i(\phl(z;\xi_1)-\phl(z;\xi_2))}\mathrm{d}z\big|=R^{n/2}(1+R^{1/2}K^{-O(1)}|\xi_1-\xi_2|)^{-10n}.
\end{equation}
In fact, for each $\theta$, partition $\theta$ into $R^{-1/2}$-balls $\om$, and by let $\wh f_{\om}=\wh f_\theta\Id_\om$.
Then $\wh f_{\om}$ is supported in an $R^{1/2}$-ball, and we have
\begin{equation}
\nonumber
    \eqref{expanding}=\sum_{\om_1,\om_2}\int \Id_S^\ast e^{i(\phl(z;\xi_1)-\phl(z;\xi_2)}\mathrm{d}z\int a^\la(z;\xi_1)\overline{a^\la(z;\xi_2)}\wh f_{\om_1}(\xi_1)\overline{\wh f_{\om_2}(\xi_2)}\mathrm{d}\xi_1\mathrm{d}\xi_2,
\end{equation}
which, via \eqref{stationary-phase} and the estimate $\|\wh f_\om\|_1\lesssim R^{-n/4}\|\wh f_\om\|_2$, is bounded above by
\begin{align*}
    &\sum_{\om_1,\om_2}\int|f_{\om_1}(\xi_1)f_{\om_2}(\xi_2)| R^{n/2}(1+R^{1/2}K^{-O(1)}|\xi_1-\xi_2|)^{-10n}\mathrm{d}\xi_1\mathrm{d}\xi_2\\
    \lesssim &\sum_{\om_1,\om_2}(1+R^{1/2}K^{-O(1)}|c(\om_1)-c(\om_2)|)^{-10n}\|\wh f_{\om_1}\|_2\|\wh f_{\om_2}\|_2\\
    \lesssim &\, K^{O(1)}\sum_\om\|\wh f_\om\|_2^2= K^{O(1)}\sum_\theta\|\wh f_\theta\|_2^2.
\end{align*}
Here $c(\om)$ is the center of $\om$.
This proves \eqref{l2-ortho}.

\smallskip

By the method of stationary phase, to prove \eqref{stationary-phase}, it suffices to show that there exists $\vec v \parallel V$ such that for all $z\in 2S$, 
\begin{equation}\label{notindual}
    |\partial_z (\phl(z;\xi_1)-\phl(z;\xi_2))\cdot \vec v|\gtrsim K^{-O(1)}|\xi_1-\xi_2|.
\end{equation}
Via Taylor's theorem, $\partial_z\phl(z;\xi_1)-\partial_z\phl(z;\xi_2)=\partial^2_{z\xi}\phl(z;\xi_2)\cdot(\xi_1-\xi_2)+O(|\xi_1-\xi_2|^2)$.
Since $|\partial^2_{z\xi}\phl(z;\xi_2)\cdot(\xi_1-\xi_2)|\gtrsim |\xi_1-\xi_2|$ and since $|\xi_1-\xi_2|\lesssim K^{-100}$, the following is true:
\begin{enumerate}
    \item $|\partial_z (\phl(z;\xi_1)-\phl(z;\xi_2))|\sim|\xi_1-\xi_2|$.
    \item $\ang(\partial_z (\phl(z;\xi_1)-\phl(z;\xi_2)), \partial^2_{z\xi}\phl(z;\xi_2)\cdot(\xi_1-\xi_2))\lesssim K^{-100}$.
\end{enumerate}
Thus, \eqref{notindual} boils down to showing that for all $z\in 2S$,
\begin{equation}
\label{angle2}
    \ang(\partial^2_{z\xi}\phl(z;\xi_2)\cdot(\xi_1-\xi_2), V^\perp)\gtrsim K^{-50}.
\end{equation}

Note that when $|z-z_0|\lesssim R^{1/2}$, 
    \[|\partial^2_{z\xi}\phl(z;\xi_2)-\partial^2_{z\xi}\phl(z_0;\xi_2)|\lesssim |z-z_0|\cdot \|\partial^3_{zz\xi}\phl\|_\infty\lesssim R^{1/2}\la^{-1}<K^{-100}. \]
Hence, to show \eqref{angle2}, we just need to show
\begin{equation}
\label{angle3} 
    \ang(\partial^2_{z\xi}\phl(z_0;\xi_2)\cdot(\xi_1-\xi_2),V^\perp)>K^{-30}. 
\end{equation}
However, by the assumption in the lemma, we have
    \[\ang(G^\la(z_0;\xi_2),V)>K^{-2},\]
which, since $G^\la(z_0;\xi_2)$ is orthogonal to $\partial^2_{z\xi}\phl(z_0;\xi_2)\cdot(\xi_1-\xi_2)$, implies \eqref{angle3}.
\qedhere

\end{proof}

\medskip

As a direct corollary, we have

\begin{lemma}
\label{local-L2-lem-2}
Let $Q$ be an $R^{1/2}$-ball and $\bT$ is a set of wave packets. 
Then 
\begin{equation}
\nonumber
    \| \sum_{T\in\bT}\cFl f_T \|_{L^2(Q)}^2\lesssim R^{1/2} \sum_{T\in\bT} \| f_T\|_2^2.
\end{equation}
\end{lemma}

\bigskip

\section{Wave packet density }
\label{section-wpd}

In this section, we introduce the wave packet density $\cW(f, B^{n+1}_R(z_0))$.
Given a $s$-cap $\tau\in \Theta_{s}$ with $R^{-1/2}\le s\le 1$, recall that  
    \[V_{\tau,R}=(Rs^2)\tau^*\]
is an isotropic dilate of $\tau^*$, a box centered at the origin of dimensions $R s^2\times R s\times \dots\times  R s$.

\begin{definition}[Wave packet density]
\label{wpd-def}
    Suppose $g$ is a function defined in $\R^{n}$ such that $\supp\wh g\subset \A^n(1)$, and suppose $B^{n+1}_R(z_0)=B^{n+1}_R(x_0,t_0)\subset B_\la^{n+1}$. Define
    \begin{equation}\label{eqdensity}
    \begin{split}
        &\cW(g,B_R^{n+1}(z_0)):=\\
        &\sup_{R^{-1/2}\le s\le 1}\sup_{\tau\in\Theta_s}\sup_{
             V\parallel V_{\tau,R}
       }\bigg(\frac{1}{|V|}\int \sum_{\begin{subarray}{c}
            \theta\subset \tau\\
            \theta\in\Theta_{R^{-1/2}}
        \end{subarray}}\sum_{\begin{subarray}{c}
            T\in\T_\theta(z_0)\\
            T^\flat\subset V
        \end{subarray}}|g_T|^2\bigg)^{1/2}.
        \end{split}
    \end{equation}
    Here $V\parallel V_{\tau,R}$ means $V$ is a translated copy of $V_{\tau,R}$, and recall \eqref{T-theta-z0} for $\T_\theta(z_0)$.
\end{definition}

\begin{remark}\label{remarkwpd}
    {\rm

We give some intuitive explanations to the set of planks  $\{T\in\bigcup_{\theta\subset \tau}\T_\theta(z_0):T^\flat\subset V\}$ that appeared in the above summation. 
For each $V\parallel V_{\tau,R}$, mimicking Definition \ref{defgaV}, we define
    \begin{equation}\label{labeldeftau}
        \Ga_V(z_0;\xi_\tau;R):=\{(\ga^\la(v+\partial_{\xi}\phi^\la(z_0;\xi_\tau),t;\xi_\tau): v\in V, |t-t_0|\le R\}. 
    \end{equation} 
When $z_0=z$, it recovers Definition \ref{defgaV}. From \eqref{defT}, one sees that for $T\in \T(z_0)$, $T=\Ga_{R^\de T^\flat}(z_0;\xi_\theta;R)$.
We call $\Ga_V(z_0;\xi_\tau;R)$ a curved plank in a ball centered at $z_0$.
Note that Lemma \ref{geometric} and Lemma \ref{lemsuppT} still hold in this setting when we move the center from the origin to $z_0$. 
By Lemma \ref{lemsuppT}, $\Ga_V(z_0;\xi_\tau;R)$ is comparable to $\Ga_V(z_0;\xi_\theta;R)$ for $\theta\subset \tau$. 
Hence, we may identify $\Ga_V(z_0;\xi_\tau;R)$ and $\Ga_V(z_0;\xi_\theta;R)$,
and the condition $T^\flat \subset V$ implies $T\subset \Ga_V(z_0;\xi_\tau;R)$.
Therefore, the wave packets summed in \eqref{eqdensity} are morally those $T$ that are contained in the fat curved plank $\Ga_V(z_0;\xi_\tau;R)$ with directional cap contained in $\tau$. 
In this regard, $\cW(g,B_R^{n+1}(z_0))$ can be viewed as the density of wave packets among all the fat curved planks in $B_R^{n+1}(z_0)$.  
}
\end{remark}

\begin{lemma}
\label{wpd-less-lem}
Suppose $g$ is a function defined in $\R^n$ with $\supp\wh g\subset \A^n(1)$.
Let $g_{\T_1}=\sum_{T\in\T_1}g_T$ and $g_{\T_2}=\sum_{T\in\T_2}g_T$ be the sums of wave packets at scale $R$ in $B_R^{n+1}(z_0)$. 
If $\T_1\subset \T_2$, then
\begin{equation}
\nonumber
     \cW(g_{\T_1},B_R^{n+1}(z_0))^2\le  \cW(g_{\T_2},B_R^{n+1}(z_0))^2. 
\end{equation}
\end{lemma}
\begin{proof}
The proof is straightforward from the definition.
\end{proof}

\begin{lemma}
\label{wpd-lem}
Suppose a function $g$ defined in $\ZR^n$ is a sum of wave packets $g=\sum_{T\in\ZT}g_T$.
Suppose $\|g_T\|_2^2\gtrsim R^{\frac{n-1}{2}}$ for all $T\in\ZT$.
Then 
\begin{equation}
\nonumber
    \cW(g,B_R^{n+1})^2\gtrsim \max\{ 1, R^{-\frac{n+1}{2}}(\#\T)\}\gtrsim R^{-\frac{n+1}{4}}(\#\ZT)^{1/2}.
\end{equation}
\end{lemma}
\begin{proof}
In the first inequality, the lower bound $1$ is obtained by testing with $V=T^\flat$ in \eqref{eqdensity}, and the lower bound $ R^{-\frac{n+1}{2}}(\#\T)$ is obtained by testing with $V=B_R^n$.
\qedhere

\end{proof}

\medskip

The next two lemmas are the main reasons why the wave packet density is appropriate for induction on scales.
The first lemma suggests the wave packet density is an appropriate substitute for the $L^\infty$ norm.

\begin{lemma}\label{Linfty} Suppose $g$ is a function defined in $\R^n$ with $\supp\wh g\subset \A^n(1)$.
Then
    \begin{equation}
    \nonumber
    \cW(g, B^{n+1}_R)\lesssim \|g\|_{\infty}.
    \end{equation}
\end{lemma}

\begin{proof}
For $\tau,V\parallel V_{\tau,R}$, denote $\T=\{T\in\bigcup_{\theta\subset \tau}\T_\theta(z_0):
            T^\flat\subset V\}$.
We just need to prove
    \[\frac{1}{|V|}\int \sum_{T\in\T}|g_T|^2\lesssim \|g\|_\infty^2. \]
Recall that both $\{\psi_\theta\}$ and $\{\eta_{T^\flat}\}$ are families of non-negative functions, and that $\{\text{supp}(\psi_\theta)\}_{\theta}$ are finitely overlapping.
By pigeonholing, we can choose $\T'\subset \T$ so that the wave packets $\{g_T:T\in\ZT'\}$ are essentially orthogonal, and we have 
\begin{equation}\label{pflinfty}
    \frac{1}{|V|}\int \sum_{T\in\T}|g_T|^2\lesssim \frac{1}{|V|}\int \sum_{T\in\T'}|g_T|^2\lesssim \frac{1}{|V|}\int |\sum_{T\in\T'}g_T|^2. 
\end{equation} 

Decompose
    \[\sum_{T\in\T'}g_T=\sum_{T\in\T'}(g\Id_{100V})_T+\sum_{T\in\T'}(g\Id_{(100V)^c})_T. \]
We claim that $(g\Id_{(100V)^c})_T=\rap(R)\|g\|_{\infty}$.
Note that,
    \[(g\Id_{(100V)^c})_T=\eta_{T^\flat}\big(\psi_\theta^\vee*(g\Id_{(100V)^c})\big). \]
Since $\psi^\vee_\theta$ is essentially supported in $\theta^*\subset V_{\tau,R}$, $\psi_\theta^\vee*(g\Id_{(100V)^c})$ is essentially supported in $(100V)^c+V_{\tau,R}\subset (90V)^c$.
The claim follows from the fact that $\eta_{T^\flat}$ is essentially supported in $T^\flat\subset V$.
Therefore,
\begin{align*}
     \eqref{pflinfty}&\lesssim \frac{1}{|V|}\int_V |\sum_{T\in\T'}(g\Id_{100V})_T|^2+\rap(R)\|g\|_\infty^2\\
     &\lesssim  \frac{1}{|V|}\int  |g\Id_{100V} |^2+\rap(R)\|g\|_\infty^2\\[1ex]
     &\lesssim \| g\|_\infty^2+\rap(R)\|g\|_\infty^2. \qedhere
\end{align*}

\end{proof}

\medskip

The second lemma establishes a connection between wave packet densities at different scales.
\begin{lemma}\label{leminduct}
Suppose $1\le r\le R^{1-10\de}$ and $B^{n+1}_r(z_0)\subset B^{n+1}_R$. Then
    \[\cW(f, B^{n+1}_r(z_0))\lesssim R^{O(\de)} (\frac Rr)^{\frac{n-1}{2}} \cW(f, B^{n+1}_R)+\rap(R)\|f\|_2.\]
   
\end{lemma}

\begin{proof}
By definition, we can assume $\cW( f, B^{n+1}_r(z_0))$ is attained at the pair $(\wt V,\wt \tau)$. 
In other words, there are $\wt s\in [r^{-1/2},1]$, $ \wt \tau\in\Theta_{\wt s}$, $ \wt V\parallel V_{\wt\tau,r}$ such that
\begin{equation}
\label{reverse-ortho-2}
    \cW(  f, B^{n+1}_r(x_0))^2 \lesssim  \frac{1}{| \wt V|}\int \sum_{\wt T\in\wt\bT}
    |f_{\wt T}|^2\lesssim\frac{1}{|\wt V|}\int \bigg|\sum_{\wt T\in\wt\bT'} f_{ \wt T}\bigg|^2,
\end{equation} 
where $\wt\bT'\subset\wt\bT$ is an appropriate subset, and $\wt\bT$ is defined as 
    \[\wt \bT:=\{ \wt T \in \bigcup_{\wt \theta\in\Theta_{r^{-1/2}},\wt\theta\subset\wt \tau}\T_{\wt\theta}: \wt T^\flat\subset \wt V \}. \]

Next, we analyze how the wave packets at scale $r$ interact with the wave packets at scale $R$. 
Recall the wave packet decomposition at scale $R$ in $B^n_R$:
    \[f=\sum_{\theta\in\Theta_{R^{-1/2}}}\sum_{T\in\T_\theta} f_T+\rap(R)\|f\|_2. \]
We define a relation $\prec$ between wave packets of scales $r$ and $R$. 
For two planks $\wt T\in\T_{\wt\theta}(z_0), T\in\T_\theta$, if $\theta\subset 4\wt \theta$ and $100 T\cap 100 \wt T\neq \emptyset$, we write
    \[\wt T\prec T. \]
By Lemma \ref{lemcomparewp}, $(f_T)_{\wt T}=\rap(R)\|f\|_2$ unless $\wt T\prec T$. 

Let $s=\wt s(\frac{r}{R})^{1/2} R^{5\de}$, so $s\in [R^{-1/2+5\de},1]$.
Let $\Theta_s(4\wt\tau)\subset\Theta_s$ be a subset with $\#\Theta_s(4\wt\tau)\sim (\frac Rr)^{\frac{n-1}{2}}$ such that
\[4\wt\tau\subset \bigcup_{\tau\in\Theta_s(4\wt\tau)} \tau.\] 
For each $\tau\in\Theta_s(4\wt\tau)$, let $V_\tau \parallel V_{\tau,R}$ be such that $V_\tau \supset \wt V+\partial_\xi\phl(z_0;\xi_\tau)$, and define
    \[\bT_\tau:=\{ T\in \bigcup_{\theta\in\Theta_{R^{-1/2}},\theta\subset\tau}\T_\theta: T^\flat \subset V_\tau \}. \]
By remark \ref{remarkwpd}, $\wt \bT$ is morally the set of $r$-planks contained in $\Ga_{\wt V}(z_0;\xi_{\wt \tau};r)$ with directional cap contained in $\wt \tau$; $\bT_\tau$ is morally the set of $R$-planks contained in $\Ga_{V_\tau}(\xi_\tau;R)$ with directional cap contained in $\tau$. See Figure \ref{comparewp} for the geometry of $\Ga_{\wt V}(z_0;\xi_{\wt \tau};r)$ and $\{\Ga_{V_\tau}(\xi_\tau;R): \tau\in \Theta_s(4\wt \tau)\}$. 
Let
    \[\bT:=\bigcup_{\tau\in\Theta_s(\wt\tau)}\bT_\tau. \]

\begin{figure}[ht]
\centering
\includegraphics[width=8cm]{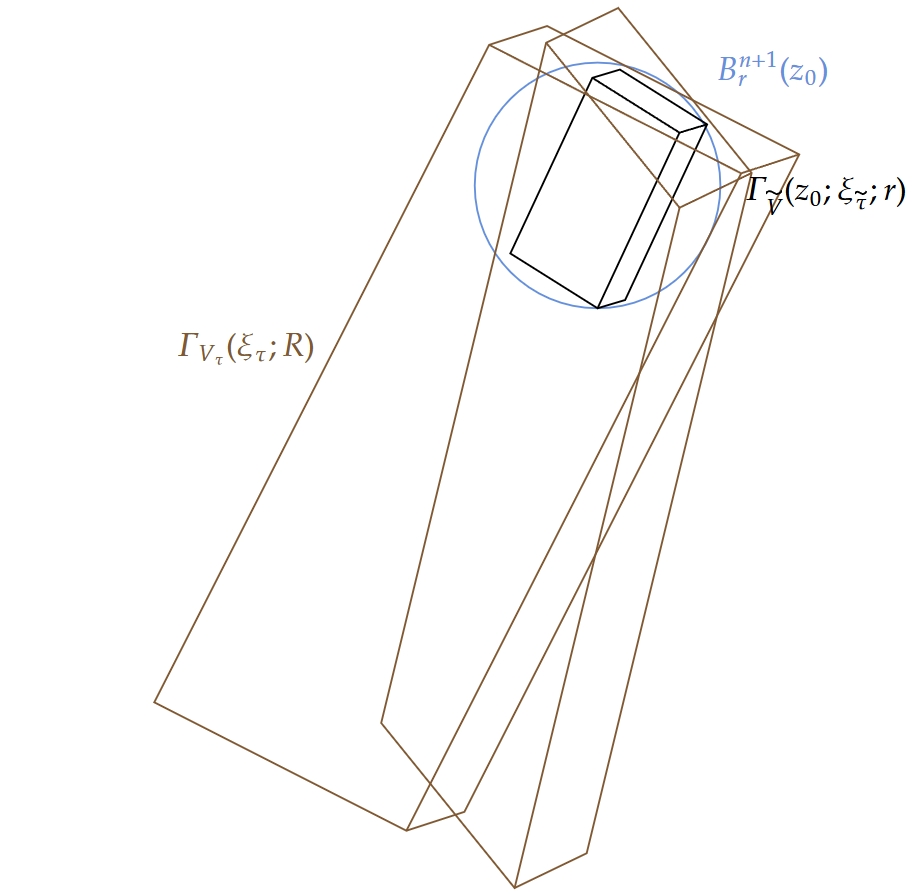}
\caption{}
\label{comparewp}
\end{figure}

We claim that if $\wt T\prec T$ for some $\wt T\in \wt \bT$, then $T\in \bT$. 
Indeed, since $\wt T\prec T$, we have $T\in\T_\theta$ for some $\theta\subset 4\wt \tau$. 
Thus, we can find $\tau\in\Theta_s(4\wt\tau)$ such that $\theta\subset \tau$.
Consider the curved plank $\Ga_{\wt V}(z_0;\xi_{\wt \tau};r)$ as in \eqref{labeldeftau}. By Remark \ref{remarkwpd}, planks in $\wt\bT$ are all contained in $\Ga_{C\wt V}(z_0;\xi_{\wt \tau};r)$ for some large $C>0$.
By Lemma \ref{lemuncertainty}, $\Ga_{C\wt V}(z_0;\xi_{\wt\tau};r)\subset \Ga_{C^2\wt V}(z_0;\xi_{\tau};r)$.
 
Since $T$ intersects the dilate of some $\wt T$ in $\wt\bT$, we know that $T\cap \Ga_{C^3 \wt V}(z_0;\xi_{\tau};r)\neq\emptyset$, which implies $R^\de T^\flat\cap \Big(C^3\wt V+\partial_\xi\phl(z_0;\xi_\tau)\Big)\neq\emptyset$ via \eqref{labeldeftau}.
This would imply
\[T^\flat \subset V_\tau.\]
This shows $T\in \bT_\tau\subset\bT$ and hence the claim.

\smallskip

Consequently, from \eqref{reverse-ortho-2}, we have
\begin{equation}
\nonumber
\begin{split}
    \cW( f,B_R^{n+1}(z_0))^2&\lesssim \frac{1}{|\wt V|}\int \bigg|\sum_{\wt T\in\wt \bT'} \bigg( \sum_{T\in\bT} f_T\bigg)_{ \wt T}\bigg|^2+\rap(R)\|f\|_2^2\\
    &\lesssim \frac{1}{|\wt V|}\int \bigg| \sum_{T\in\bT} f_T\bigg|^2+\rap(R)\|f\|_2^2\\
    &\lesssim \frac{1}{|\wt V|}\int  \sum_{T\in\bT} |f_T|^2+\rap(R)\|f\|_2^2\\
    &=\sum_{\tau\in\Theta_s(4\wt\tau)} \frac{1}{|\wt V|}\int \sum_{T\in\bT_\tau} |f_T|^2+\rap(R)\|f\|_2^2.
\end{split}
\end{equation}

Now we finish the proof of the lemma.
By the definition of $\bT_\tau$, we have
\begin{equation}
\nonumber
\cW( f,B_r^{n+1}(z_0))^2\lesssim \!\!\!\sum_{\tau\in\Theta_s(4\wt\tau)} \!\frac{|V_\tau|}{|\wt V|}\bigg(\frac{1}{|V_\tau|}\int \sum_{\begin{subarray}{c}
            \theta\subset \tau\\
            \theta\in\Theta_{R^{-1/2}}
        \end{subarray}}\sum_{\begin{subarray}{c}
           T\in\T_\theta \\
           T^\flat\subset  V_\tau
\end{subarray}}\!|f_T|^2\bigg)+\rap(R)\|f\|_2^2.
\end{equation}
Note that each $V_\tau$ is of dimensions $R^{10\de}r \wt s^2\times R^{5\de} ( Rr)^{1/2}\wt s\times \dots\times R^{5\de}(Rr)^{1/2}\wt s$.
We have
\[\frac{|V_\tau|}{|\wt V|}\lesssim R^{O(\de)}(\frac Rr)^{\frac{n-1}{2}}.\] 
Also note 
\[\#\Theta_s(4\wt\tau)\lesssim R^{O(\de)} (\frac Rr)^{\frac{n-1}{2}}.\] 
We obtain
\begin{equation*}
    \cW( f,B_r^{n+1}(z_0))^2\lesssim R^{O(\de)} (\frac Rr)^{n-1}\cW (f, B^{n+1}_R)^2+\rap(R)\|f\|_2^2. \qedhere
\end{equation*}

\end{proof}

\begin{remark}
    { \rm
    The $R^{O(\de)}$ factor appears here since the wave packet $T$ has size $R^{2\de}\times R^{1/2+\de}\times \dots\times R^{1/2+\de}$, rather than $1\times R^{1/2}\times \dots\times R^{1/2}$. Though, the lemma still works for the induction step since $\de$ is very small compared to $\e$.
    
    }
\end{remark}

\medskip

\begin{remark}
\label{wpd-cwa-rmk-1}
\rm

The wave packet density introduced in Definition \ref{wpd-def} naturally imposes a non-concentration condition on the set of planks $\ZT$, where $g=\sum_{T\in\ZT}g_T$ is the wave packet decomposition of $g$.
Indeed, if $\|g_T\|_2$ are about the same for all $T\in\ZT$, then the assumption $\cW(g,B_R^{n+1}(z_0))\lesssim1$ is the necessary condition for the following Kakeya estimate to hold: $|\cup_\ZT|\gtrapprox\sum_{T\in\ZT}|T|$.
However, it is likely that this is not a sufficient condition.
Under the current definition of wave packet density, the assumption $\cW(g, B_R^{n+1}(z_0)) \lesssim 1$ is analogous to the classical ``$(n-1)$-ball condition" for a fractal set (see \cite[(1.1)]{KWZ} for an example with $n = 3$ and $\ZT$ replaced by a family of tubes).

\end{remark}

\begin{remark}[Continuation of Remark \ref{wpd-cwa-rmk-1}]
\label{wpd-cwa-rmk-2}
\rm

Nevertheless, wave packet density is a versatile concept.
One may modify its definition to impose a stronger non-concentration condition on the planks $\ZT$.
For instance, the next modification of Definition \ref{wpd-def} will enable us to impose a non-concentration condition on the planks $\ZT$, similar to the \textit{Convex Wolff Axiom} (see \cite[Definition 0.12]{wang2024restriction}) for tubes:

\smallskip

For dyadic numbers $s_1, s_2\dots,s_{n-1}\in[R^{-1/2}, 1]$, and $\tau_\circ$ being a convex subset of $[-1,1]^{n-1}\times \{1\}$ of dimensions $s_1\times \dots\times s_{n-1}$, define $\tau:=\{ \xi\in\A^n(1): \xi/|\xi|\in\tau_\circ \}$. 
We say $\tau$ is an $(s_1,\dots,s_{n-1})$-cap. Then $\tau$ is a convex set and can be roughly viewed as a box of dimensions $1\times s_1\times \dots\times s_{n-1}$. 
For example, a cap $\tau\in\Theta_{s}$ is an $(s,\dots,s)$-cap.
Given an $(s_1,\dots,s_{n-1})$-cap $\tau$, we define its dual box $V_{\tau,R}$ in the physical space $\R^n$ as follows:
Let $s_{\text{max}}:=\max_{1\le i\le n-1} s_i$. 
Define $V_{\tau,R}$ to be the box centered at the origin, of dimensions $R \smax^2\times R s_1\times \dots\times R s_{n-1}$, whose edges are parallel to the corresponding edges of the $1\times s_1\times \dots\times s_{n-1}$-box $\tau$.

Now, we give a generalization of Definition \ref{wpd-def}. 
One can check that the lemmas corresponding to Lemmas \ref{wpd-less-lem}, \ref{Linfty}, and \ref{leminduct} are also true.

\medskip

\noindent {\emph{Generalization of Definition \ref{wpd-def}}: } Suppose $g$ is a function in $\R^n$ with $\supp\wh g\subset \A(1)$, and $B^{n}_R(x_0)$ is a ball. Define
    \begin{equation}
    \nonumber
        \wt\cW(g,B_R^{n+1}(z_0)):=
        \sup_{\tau}\sup_{\begin{subarray}{c}
             V\parallel V_{\tau,R} \\
             V\subset B^{n}_R
        \end{subarray}}\bigg(\frac{1}{|V|}\int \sum_{\begin{subarray}{c}
            \theta\subset \tau\\
            \theta\in\Theta_R
        \end{subarray}}\sum_{\begin{subarray}{c}
            T^\flat\subset V\\
            T\in\T_\theta(z_0)
        \end{subarray}}|g_T|^2\bigg)^{1/2}.
    \end{equation}
Here, $\sup_\tau$ is taken over all caps $(s_1,\dots,s_{n-1})$-caps $\tau$ for all dyadic numbers $s_1, s_2, \ldots,s_{n-1}\in[R^{-1/2}, 1]$.

\end{remark}

\bigskip

Finally, we reduce Proposition \ref{1propLpest} to the following theorem, which is formulated in terms of a mixed norm defined via the wave packet density.
\begin{theorem}
\label{mixed-norm-thm}
Let $\cf$ be a type $\bfone$ Fourier integral operator given by \eqref{FIO}.
Suppose the phase function $\phi$ satisfies \eqref{phi0}, and the amplitude function $a$ obeys $\supp_\xi\ a\subset\A^n(1)^\circ$.
Let $\cf^\la$ be given by \eqref{FIOlambda}. 
Then for all $\e>0$ and when $p=p(n)$, there exists a constant $C_{\e}$ that is independent to $\cf$, such that 
\begin{equation}
\label{mixed-norm-esti-1}
    \|\cFl f\|_{L^p(B_R^{n+1})}\le C_\e R^{(n-1)(\frac12-\frac1p)+\e}\|f\|_{2}^{\frac2p}\cW( f,B_R^{n+1})^{1-\frac2p}
\end{equation}
for all $f$ with $\supp\wh f\subset\A^n(1)$ and all $R\in[1, \la^{1-\e}]$.
\end{theorem}

\smallskip

\begin{proof}[Proof that Theorem \ref{mixed-norm-thm} implies Proposition \ref{1propLpest}]

Since $\supp_\xi\ a\subset\A^n(1)^\circ$, it follows that $\cFl f=\cFl (\Id_{\A^n(1)}^\ast \wh f)^\vee$, where $\Id_{\A^n(1)}^\ast$ is a bump that equals to 1 on $\supp_\xi\ a$ and is supported in $\A^n(1)$.
By summing up all balls $B_R^{n+1}$ inside $B_\la^{n+1}$, Theorem \ref{mixed-norm-thm} and Lemma \ref{Linfty} imply that for all $\e>0$ and when $p=p(n)$,
\begin{equation}
\label{mixed-norm-esti-2}
    \|\cFl f\|_{L^p(B_\la^{n+1})}\le C_\e \la^{(n-1)(\frac12-\frac1p)+\e}\|f\|_{2}^{\frac2p}\|f\|_\infty^{1-\frac2p}
\end{equation}
We see that \eqref{mixed-norm-esti-2} implies \eqref{lp-esti-1} when $f$ is a characteristic function and when $p=p(n)$.
Therefore, the range $p>p(n)$ for \eqref{lp-esti-1} follows from  a real interpolation between the restrict-type estimate when $p=p(n)$ and the trivial bound $\|\cf^\la f\|_\infty \lesssim \la^{\frac{n-1}{2}}\|f\|_\infty$.
This proves Proposition \ref{1propLpest}.
\qedhere

\qedhere

\end{proof}

\bigskip

\section{Some geometric results}

In this section, we establish several geometric results that will be used to prove Theorem \ref{mixed-norm-thm} in subsequent sections.

\subsection{Lorentz rescaling}
\label{Lorentz-rescaling-section}

Suppose $\tau\in\Theta_{K^{-1}}$ and $\supp \wh f\subset \tau$. 
We want to find an upper bound for $\|\cFl f\|_{L^p(B_R^{n+1})}$ by induction. 
The idea is to perform a rescaling so that $\cF^\la f$ becomes $\wt \cF^{\la/K^2} g$, where $\wt \cF$ is a new Fourier integral operator, and $g$ is a function with $\supp\wh g\subset \A^{n}(1)$. 
At the same time, the scale of the integration domain drops from $R$ to $R/K^2$. 
Moreover, when $\cf$ is of type $\bfone$, the new operator $\wt \cF$ is of type $(1,1,C)$ for some constant $C$ (see \cite[Section 2.5]{beltran2020variable}).
Therefore, by partitioning the new operator $\wt \cF$ into $O(1)$ parts, we can use an induction hypothesis at the scale $(\la/K^2,R/K^2)$ to bound $\|\cFl f\|_{L^p(B_R^{n+1})}$.

Next, we describe the rescaling and the relevant geometry.
Given a $\tau\in\Theta_{K^{-1}}$, let $V_{\tau,R}\subset B_R^n$ be a dual $RK^{-2}\times RK^{-1}\times\cdots\times RK^{-1}$-box, as defined in Lemma \ref{lemuncertainty}. 
Tile $B_R^n$ by translated copies of $V_{\tau,R}$ and denote them by $\B_\tau^\flat=\{\Box^\flat\}$. 
For each $\Box^\flat\in\B^\flat_\tau$, define a curved plank
\begin{equation}
\label{Box}
    \Box:=\Ga_{\Box^\flat}(\xi_\tau;R)=\{(\ga^\la(v,t;\xi_\tau),t;\xi_\tau):v\in \Box^\flat,|t|\le R \}. 
\end{equation}
Denote $\B_\tau=\{\Box:\Box^\flat\in\B^\flat_\tau\}$.
Observe that $\B_\tau$ form a covering of $B_R^{n+1}$, and each $T\in \bigcup_{\theta\subset \tau}\T_\theta$ belongs to $O(1)$ many $\Box\in\B_\tau$.
We assign each $T\in\bigcup_{\theta\subset \tau}\T_\theta$ to a single $\Box$ to which $T$ belongs, and denote the collection of $T$ assigned to $\Box$ by $\T_\Box$.
This gives a partition 
\begin{equation}
\label{ZT-Box}
    \bigcup_{\theta\subset \tau}\T_\theta=\bigcup_{\Box}\T_\Box.
\end{equation}

Given a $\tau\in\Theta_{K^{-1}}$, write $\xi_\tau=(\om,1)$, and define $\Upsilon_{\om}(y,t):=(\ga(y,t;\om,1),t)$ and $\Upsilon^\la_{\om}(y,t):=\la\Upsilon_\om(y/\la,t/\la)$.
Introduce two non-isotropic dilations $D_K(y',y_n,t):=(Ky',y_n,K^2t)$ and $D'_{K^{-1}}(y',y_n):=(K^{-1}u',K^{-2}u_n)$. 
Then, as shown in  \cite[Proof of Lemma 2.3]{beltran2020variable}, we have
\begin{equation}\label{rescaling}
    \cFl g\circ \Upsilon_{\om}^\la\circ D_K= \wt\cF^{\la/K^2}\wt g, 
\end{equation} 
where
    \[\wt g(\eta):=K^{-(n-1)}g(\eta_n\om+K^{-1}\eta',\eta_n), \]
\begin{equation}
\label{cf-tilde}
    \wt\cF^{\la/K^2}\wt g(y,t):=\int_{\R^n}e^{i\wt\phi^{\la/K^2}(y,t;\eta)}\wt a^\la(y,t;\eta)\wt g(\eta)\mathrm{d}\eta.
\end{equation}
Here the phase $\wt \phi(y,t;\eta)$ is given by
    \[\langle y,\eta\rangle+ \int_0^1 (1-r)\langle \partial^2_{\xi'\xi'}\phi(\Upsilon_\om(D'_{K^{-1}}y,t);\eta_n\om+rK^{-1}\eta',\eta_n)\eta',\eta'\rangle\mathrm{d} r, \]
and the amplitude function $\wt a(y,t;\eta)$ is given by 
\[\wt a(y,t;\eta):= a(\Upsilon_\om(D'_{K^{-1}}y;t);\eta_n\om+K^{-1}\eta',\eta_n).\]
We remark that for each $\Box\in\B_\tau$, $(\Upsilon^\la_\om\circ D_K)^{-1}(\Box)$ is morally a $R/K^2$-ball.

\medskip
Now we state and prove the following lemma.

\begin{lemma}\label{lemrescaling}
Let $1\le K\le R^\e$.
Suppose there is a constant $C>0$ such that for all function $g$ defined in $\R^n$ with $\supp\wh g\subset \A^{n}(1)$ and all type $\bfone$ FIO $\wt\cF$, 
\begin{equation}\label{resc}
    \|\wt\cF^{\la/K^2} g\|_{L^p(B_{R/K^2}^{n+1})}\le C \|g\|_2^{\frac2p}\cW(g,B_{R/K^2}^{n+1})^{1-\frac2p}.
    \end{equation}
Let $\tau\in \Theta_{K^{-1}}$. 
Then for all function $f$ defined in $\R^n$ with $\supp\wh f\subset \tau$ and all type $\bfone$ FIO $\wt\cF$, we have
\begin{equation}
\nonumber
    \|\cFl f\|_{L^p(B_R^{n+1})}\lesssim C K^{\frac{2}{p}} \|f\|_2^{\frac2p}\cW(f,B_R^{n+1})^{1-\frac2p}.
\end{equation}
    
\end{lemma}

\begin{proof}
Write $\tau=\bigcup_{\theta\subset \tau}\theta$, where $\theta\in\Theta_{R^{-1/2}}$. 
Recall the definitions of $\B_\tau=\{\Box\}$ in \eqref{Box} and $\T_\Box$ in \eqref{ZT-Box}.
Perform the wave packet decomposition for $\cFl f$ inside $B_R^{n+1}$ to have
    \[\cFl f=\sum_{\theta\subset \tau}\sum_{T\in\T_\theta}\cFl f_T=\sum_\Box\sum_{T\in\T_\Box} \cFl f_T. \]
Recall the rescaling $\Upsilon^\la_\om\circ D_K$ in \eqref{rescaling}.
For each $\Box\in\B_\tau$, let $B_{R/K^2,\Box}^{n+1}$ be the $R/K^2$-ball containing $(\Upsilon^\la_\om\circ D_K)^{-1}(\Box)$. 
Note that, under this rescaling, each $T\in \T_\Box$ becomes a $R/K^2$-plank $\wt T$ in $B_{R/K^2,\Box}^{n+1}$.
We denote $\wt \T_\Box=\{\wt T: T\in\T_\Box\}$. 
Thus, under the same rescaling, the function $\cFl f$ becomes
    \[\sum_{\Box}\sum_{\wt T\in \wt \T_\Box} \wt\cF^{\la/K^2}\wt g_{\wt T}, \]
where $\wt \cF$ is an FIO of type $(1,1,C)$ for some constant $C$.

By the essential disjointness of $\Box$, we have
\begin{equation}\label{rescplugin}
    \|\cFl f\|_{L^p(B_R^{n+1})}^p\lesssim \sum_{\Box} \|\sum_{T\in \T_\Box }\cFl f_T\|_{L^p(\Box)}^p. 
\end{equation} 
For each $\Box$, since $\Upsilon^\la_\om$ is a diffeomorphism with determinant $\sim1$,
\[\|\sum_{T\in \T_\Box }\cFl f_T\|_{L^p(\Box)}\sim K^{\frac{n+1}{p}}\|\sum_{\wt T\in \wt\T_\Box }\wt\cF^{\la/K^2} g_{\wt T}\|_{L^p(B^{n+1}_{R/K^2,\Box})}, \]
which, by partitioning $\wt \cF$ into $O(1)$ parts and by \eqref{resc}, is bounded by
    \[\lesssim K^{\frac{n+1}{p}} C \|\sum_{\wt T \in \wt \T_\Box} g_{\wt T}\|_{2}^{\frac2p}\cW(\sum_{\wt T \in \wt \T_\Box} g_{\wt T}, B^{n+1}_{R/K^2,\Box})^{1-\frac2p}. \]
Again, since $\Upsilon^\la_\om$ is a diffeomorphism with determinant $\sim1$, we have
    \[\|\sum_{\wt T \in \wt \T_\Box} g_{\wt T}\|_{2}\sim K^{-\frac{n-1}{2}}\|\sum_{T \in \T_\Box}f_{T}\|_{2}. \]
Finally, notice that
    \[\cW(\sum_{\wt T \in \wt \T_\Box}  g_{\wt T}, B^{n+1}_{R/K^2,\Box})\sim \cW(\sum_{T \in \T_\Box}  f_{T}, B^{n+1}_{R})\lesssim \cW(  f, B^{n+1}_R). \]

Plugging the aforementioned information into \eqref{rescplugin}, we obtain
\begin{align*}
    \|\cFl f\|_{L^p(B_R^{n+1})}^p & \lesssim C^p K^{2} \sum_\Box \| 
    \sum_{T\in\T_\Box}  f_T \|_2^2\, 
    \cW(  f, B_R^{n+1})^{p-2}\\
    &\lesssim C^p K^{2} \|f\|_2^2\,\cW(  f,B_R^{n+1})^{p-2}. \qedhere
\end{align*}

\end{proof}

\subsection{Hairbrush}

We call $\Tau\subset B_R^{n+1}$ an $R$-tube if $\Tau$ has the form
    \[\cT=\Ga_{B}(\xi;R) \]
for some $R^{1/2}$-ball $B=B_{R^{1/2}}^n(v)$ and $\xi\in\A^n(1)$ (recall Definition \ref{defgaV}).
Since $\Ga_B(\xi;R)=\Ga_B(s\xi;R)$ (recall \eqref{1-homo}), we may assume $\xi=(\xi',1)$ by rescaling.
We call $\xi$ the direction of the tube and $v$ the position of the tube. 
Notice that the core curve of the tube is given by
\begin{equation}
\label{core-curve}
    \{(\ga^\la(v,t;\xi),t): |t|\le R\}.
\end{equation}

Similar to Definition \ref{distinct}, we make the following definition.
\begin{definition}
Two $R$-tubes $\Ga_{B_1}(\xi_1;R)$ and $\Ga_{B_2}(\xi_2;R)$ are \textbf{distinct} if 
    \[\Ga_{B_j}(\xi_j;R)\not\subset \Ga_{CB_k}(\xi_k;R)\] 
for $j,k\in\{1,2\}, j\not=k$, and some absolute big constant $C$.
\end{definition}

Also, similar to Lemma \ref{lemuncertainty}, we have
\begin{lemma}
\label{lem-distinct-tube}
Two $R$-tubes $\Ga_{B_1}(\xi_1;R)$ and $\Ga_{B_2}(\xi_2;R)$ are distinct if one of the following is true:
\begin{enumerate}
    \item $\dist(\xi_1',\xi_2')\gtrsim R^{-1/2}$.
    \item $\dist(B_1,B_2)\gtrsim R^{1/2}$.
\end{enumerate}
\end{lemma}

Recall $K=R^{\e^{50}}$ and let $\Kc=R^{\e^{100}}$.

\begin{lemma}\label{lemhairbrush}
Let $\boldsymbol{\Tau}=\{\Tau\}$ be a family of distinct $R$-tubes and let $\cT_0$ be an $R$-tube in $B_R^{n+1}$. 
Let $\om_1=\R^{n}\times [i_1R\Kc^{-1},(i_1+1)R\Kc^{-1}]$ and $\om_2=\R^{n}\times [i_2R\Kc^{-1},(i_2+1)R\Kc^{-1}]$ be two horizontal regions with $i_1,i_2\in\ZZ, |i_1|,|i_2|\le \Kc$ and $|i_1-i_2|\ge 2$. 
Suppose for any $\Tau\in\boldsymbol{\Tau}$, the direction of $\Tau_0$ and the direction of $\Tau$ are $K^{-1}$-separated,  $\Tau_0\cap \Tau\neq\emptyset$, and $\Tau_0\cap \Tau\subset \om_1$. 
Then for $z\in\om_2$,
\begin{equation}\label{hairbrushineq2}
    \sum_{\Tau\in\boldsymbol{\Tau}}\Id_\Tau(z)\lesssim (K\Kc)^{O(1)}.
\end{equation} 

\end{lemma}

\begin{figure}[ht]
\centering
\includegraphics[width=10cm]{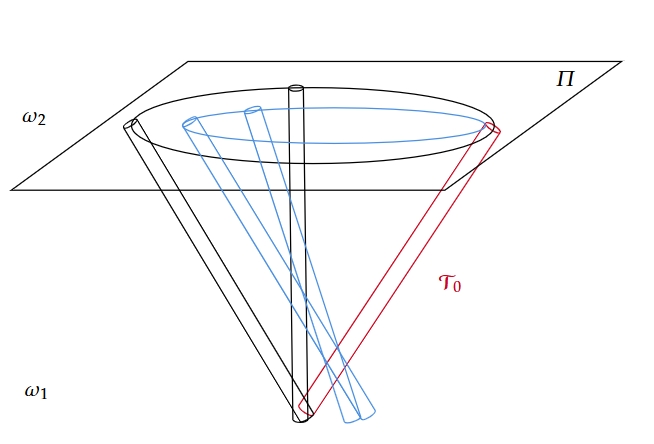}
\caption{Hairbrush for light rays}
\label{hairbrush}
\end{figure}

\smallskip

Before proving the lemma, we would like to discuss the proof idea when $\cf$ is a constant coefficient FIO (note that $\boldsymbol{\Tau}$ depends on the operator $\cf$). 
The variable coefficient case follows similarly.

We refer to Figure \ref{hairbrush} for the key geometric observation. 
When $\cf$ is a constant coefficient FIO, each $R$-tube $\Tau$ is a straight tube pointing to a light ray direction.
Fix $t\in [i_2R\Kc^{-1},(i_2+1)R\Kc^{-1}]$, and let $\Pi=\R^n\times \{t\}$ be a horizontal hyperplane in $\om_2$. 
It suffices to prove \eqref{hairbrushineq2} for $z\in\Pi$. 
Let $\{Q\}$ be a finite-overlapping cover of $\Tau_0\cap \om_1$ by $R^{1/2}$-balls.
Since $\Tau\cap \Tau_0\subset \om_1$, we assign each $\Tau\in\boldsymbol{\Tau}$ to a $Q$ such that $\Tau_0\cap\Tau\cap Q\neq\emptyset$. 
Denote the set of tubes assigned to $Q$ by $\boldsymbol{\Tau}(Q)$ so that
    \[\boldsymbol{\Tau}=\bigsqcup_{Q}\boldsymbol{\Tau}(Q). \]
Let $C(Q):=\bigcup_{\Tau\in\boldsymbol{\Tau}(Q)}\Tau\cap \Pi$, which is contained in the $R^{1/2}$-neighborhood of a $(n-1)$-sphere in $\Pi$. 
Since $\dist(\om_1,\om_2)\ge R\Kc^{-1}$, the spheres $C(Q)$ for different $Q$ have $R^{1/2}\Kc^{-1}$-separated radius, and all of them are tangent at $\Tau_0\cap \Pi$. 
Also, since each $\Tau$ and $\Tau_0$ are $K^{-1}$-separated in direction, $C(Q)\cap N_{100(K\Kc)^{-1}R}(\Tau_0\cap \Pi)=\emptyset$. 
Therefore, the sets $\{C(Q)\}_Q$ are $\lesssim (K\Kc)^{O(1)}$-overlapping, and hence $\{\Tau\cap\Pi\}_{\Tau\in\boldsymbol{\Tau}}$ are $(K\Kc)^{O(1)}$-overlapping. 
We remark that a $O(1)$ loss in exponent is acceptable, as $(K\Kc)^{O(1)}\leq R^{10\e}$.

The variable coefficient case will be clear once we understand the geometry of each corresponding $C(Q)$.
As we will see, each $C(Q)$ is contained in the $R^{1/2}$-neighborhood of a certain surface with a positive definite second fundamental form.

\medskip

\begin{proof}[Proof of Lemma \ref{lemhairbrush}]
After affine transformation, we assume $\Tau_0$ has position $v=0$ and direction $(\xi',\xi_n)=(0,\dots,0,1)$. 
Also, we want to simplify the phase function. 
We claim that, by applying a certain change of variables, we may assume 
\begin{equation}
\nonumber
    \phi(x,t;\xi)=\langle x,\xi\rangle+t \frac{h(\xi')}{\xi_n}+\cE(x,t;\xi'/\xi_n)\xi_n,
\end{equation}
where $h$ is quadratic and $|\cE(x,t;\xi')|=O( |t| |\xi'|^3+|z|^2|\xi'|^2  )$ is the higher order term.
The reader may compare it with  \cite[(1.11)]{bourgain1991lp}. 
We remark that the changes of variables we will use always have Jacobians whose determinants have absolute value comparable to 1.

\smallskip

Now, we prove the claim. 
Recall the notations  $z=(x,t)$, $x=(x',x_n),\ \xi=(\xi',\xi_n)$. 
For $i\le n$, note that $z_i=x_i$. 
Since $\phi$ is $1$-homogeneous in $\xi$, we focus on $\phi(x,t;\xi',1)$. 
Write
    \[\phi(z;\xi',1)=\phi(z;0,1)+g(z;\xi'), \]
with $g(z;0)\equiv 0$.
By the {\bf (H1)} condition of $\phi$, we have
$\partial_x \phi(z;0,1)\neq 0$. 
We may assume $\partial_{z_n}\phi(z;0,1)\neq 0$, so that we can perform the change of variables $\wt z_n=\phi(z;0,1)$ and $\wt z_i=z_i$ ($i\neq n$). 
In the new coordinate (we still use $z$ to denote $\wt z$), we have
    \[\phi(z;\xi',1)=z_n+g(z;\xi'). \]

Apply Taylor's expansion on $g$ to have
    \[\phi(z;\xi',1)=z_n+\sum_{i=1}^{n+1} \sum_{j=1}^{n-1}c_{ij}z_i\xi_j+g_1(z;\xi'), \]
where $\partial^2_{z\xi'}g_1(0,0)=0$. 
There is no $\xi$-term in the above expansion since $\partial^\al_\xi\phi(0;0)=0$ by \eqref{phi0}. 
Since $\phi$ is $1$-homogeneous in $\xi$, we have
    \[\phi(z;\xi',\xi_n)=z_n\xi_n+\sum_{i=1}^{n+1} \sum_{j=1}^{n-1}c_{ij}z_i\xi_j+g_1(z;\xi'/\xi_n)\xi_n. \]
By condition {\bf (H1)} that $\partial^2_{x\xi}\phi(0,0)$ has rank $n$, we see that
\[\partial^2_{x\xi}(z_n\xi_n+\sum_{i=1}^{n+1} \sum_{j=1}^{n-1}c_{ij}z_i\xi_j)\] has rank $n$. 
Hence, a linear change of variables in $z$ gives
    \[\phi(z;\xi',\xi_n)=\sum_{i=1}^nz_i\xi_i+g_2(z;\xi'/\xi_n)\xi_n. \]
Thus, in appropriate coordinates,
    \[\phi(z;\xi',1)=z_n+\sum_{i=1}^{n-1}z_i\xi_i+ g_2(z;\xi'). \]

By Taylor's expansion, we can further write
    \[\phi(z;\xi',1)=z_n+\sum_{i=1}^{n-1}z_i\xi_i+\sum_{i=1}^{n-1}a_i(z)\xi_i+\sum_{i=1}^{n+1}z_ib_i(\xi') +g_3(z;\xi'), \]
where $a_i, b_i$ are quadratic and $ |g_3(z;\xi')|=O(|z|^3|\xi'|+|x||\xi'|^3))$. 
Via the changes of variables $\wt z_i=z_i+a_i(z), \wt \xi_i=\xi_i+b_i(\xi')$ for $i\le n-1$, we have 
    \[\phi(z;\xi',1)=z_n+\sum_{i=1}^{n-1}z_i\xi_i+z_n b_n(\xi')+z_{n+1}b_{n+1}(\xi') +g_4(z;\xi'), \]
where $g_4$ satisfies the same condition as $g_3$.
Again, since $\phi$ is $1$-homogeneous in $\xi$,
    \[\phi(z;\xi',\xi_n)=z_n (\xi_n+b_n(\xi'/\xi_n)\xi_n)+\sum_{i=1}^{n-1}z_i\xi_i+z_{n+1}b_{n+1}(\xi'/\xi_n)\xi_n +g_4(z;\xi'/\xi_n)\xi_n.\]
Perform the $1$-homogeneous change of variables $\wt\xi_n=\xi_n+b_n(\xi'/\xi_n)\xi_n$. 
Noting that $\xi_n\sim 1$ and since $b_n$ is quadratic, when $|\xi'|$ is sufficiently small, $\wt \xi_n\sim 1$.
Thus, we reduce to the form
    \[\phi(z;\xi',\xi_n)=\sum_{i=1}^{n}z_i\xi_i+z_{n+1}\frac{h(\xi')}{\xi_n} +g_5(z;\xi'/\xi_n)\xi_n, \]
where $h$ is quadratic and $g_5(z;\xi')=O(|z|^3|\xi'|+|z||\xi'|^3)$.
Finally, as before, we perform a change of variables $\wt z_i=z_i+O(|z|^3)$ ($i\le n$) to eliminate the $|z|^3|\xi'|$ term in $g_5$; then perform $\wt\xi_i=\xi_i+O(|\xi|^3)$ ($i\le n-1$) to eliminate the $|x||\xi'|^3$ term in $g_5$. 
Eventually, we can reduce $\phi$ to the following form (note that $t=z_{n+1}$)
\begin{equation}
\nonumber
    \phi(x,t;\xi)=\langle x,\xi\rangle+t \frac{h(\xi')}{\xi_n}+\cE(x,t;\xi'/\xi_n)\xi_n,
\end{equation}
where $\cE(x,t;\xi')=O(|z|^2|\xi'|^2+|t||\xi'|^3)$. 
This finishes the proof of the claim.
Moreover, by condition {\bf (H2)}, we know that  $\textup{Hess}(h)$ is positive definite. 
Thus, under appropriate coordinates of the physical variable $z=(x,t)$, by another linear change of variable on $\xi'$, we can further reduce $\phi$ to
\begin{equation}
\label{standard-form}
    \phi(x,t;\xi)=\langle x,\xi\rangle+t \frac{|\xi'|^2}{\xi_n}+\cE(x,t;\xi'/\xi_n)\xi_n.
\end{equation}
By partitioning the amplitude function in the operator $\cf$ a priori, we assume that $|\partial_{t,\xi'}\cE(x,t;\xi'/\xi_n)|\leq (10n)^{-10}$ for $\xi\in\text{supp}_\xi\ a$, the $\xi$-support of the the amplitude function $a(\cdot\,;\cdot)$.
Denote $\wt \cE(x,t;\xi',\xi_n)=\cE(x,t;\xi'/\xi_n)\xi_n$, and $\wt\cE^\la(x,t;\xi',\xi_n)=\la \wt\cE(x/\la,t/\la;\xi',\xi_n)$.

\medskip

We return to the proof of Lemma \ref{lemhairbrush}. 
One can compute
    \[\partial_\xi\phi(x,t;\xi)=(x'+2t\frac{\xi'}{\xi_n}+\partial_{\xi'}\wt\cE,x_n-t\frac{ |\xi'|^2}{\xi_n^2}+\partial_{\xi_n}\wt\cE). \]
Hence,
    \[\partial_\xi\phl(x,t;\xi)=(x'+2t\frac{\xi'}{\xi_n}+\partial_{\xi'}\wt\cE^\la,x_n-t\frac{ |\xi'|^2}{\xi_n^2}+\partial_{\xi_n}\wt\cE^\la). \]
Recall \eqref{transformation-id}.
To find out the core curve of $\Tau_0$, let $(\xi',\xi_n)=(0,1)$ and solve the equation $\partial_\xi\phl(x,t;0,1)=0$. 
Since $h$ is quadratic and since $\cE(x,t;\xi')=O(|z|^2|\xi'|^2+|t||\xi'|^3)$, the solution is $x=0$.
Hence, the core curve of $\Tau_0$ is $\{ (0,t):|t|\le R \}$.

Suppose $\Tau$ is a tube whose core curve intersects $\Tau_0$ at $z_0=(0,t_0)$. 
Then there exists $\xi'$ so that the points $(x,t)$ on the core curve of $\Tau$ satisfy 
    \[\partial_\xi \phl(x,t;\xi',1)=\partial_\xi \phl(0,t_0;\xi',1). \]
Solve this equation and obtain
\begin{equation}\label{geodesic}
    \begin{split}
    &x'=x'(t,t_0,\xi')=-2(t-t_0) \xi'+ \partial_{\xi'}\Big( 
    \wt\cE^\la(x,t;\xi',1)-\wt\cE^\la(x,t_0;\xi',1) \Big),\\
    &x_n=x_n(t,t_0,\xi')=(t-t_0) |\xi'|^2+\partial_{\xi_n}\Big( \wt\cE^\la(x,t;\xi',1)-\wt\cE^\la(x,t_0;\xi',1) \Big). 
    \end{split}
\end{equation} 

Let $\cq$ be a finite-overlapping cover of $\cT\cap \om_1$ by $R^{1/2}$-balls. We may assume the center of each $Q\in\cQ$ is of form $(0,\dots,0,t_Q)$.
Since $\Tau\cap \Tau_0\subset \om_1$, we assign each $\Tau\in\boldsymbol{\Tau}$ to a $Q\in\cq$ such that $\Tau_0\cap\Tau\cap Q\neq\emptyset$. 
Denote the set of tubes assigned to $Q$ by $\boldsymbol{\Tau}(Q)$ so that
    \[\boldsymbol{\Tau}=\bigsqcup_{Q}\boldsymbol{\Tau}(Q). \]
For each tube $\Ga_{B}(\xi;R)\in\boldsymbol{\Tau}(Q)$, we may assume that the core curve of $\Ga_{B}(\xi;R)$ intersects the core curve of $\cT_0$ at the center of $Q$.
Fix an arbitrary $t_1$ such that $\ZR^n\times\{t_1\}\subset\om_2$.
For each $\cT\in\boldsymbol{\cT}$, let $x(\cT;t_1)$ be the intersection of the core curve of $\cT$ with the horizontal hyperplane $\ZR^n\times\{t_1\}$. By \eqref{geodesic}, 
\begin{equation}\label{xT}
    x(\cT;t_1)=(x'(t_1,t_Q,\xi'),x_n(t_1,t_Q,\xi')).
\end{equation} 
Let $B(\cT;t_1)=B^n_{R^{1/2}}(x(\cT;t_1))$.
It suffices to show that the balls $\{B(\cT;t_1)\}_{\cT\in\boldsymbol{\cT
}}$ are $(K\Kc)^{O(1)}$-overlapping.

\smallskip

Note that $|\xi|\gtrsim K^{-1}$, since the directions of the tubes in $\boldsymbol{\Tau}$ and the direction of $\cT_0$, which is $0$, are $K^{-1}$-separated.
Since tubes in $\boldsymbol{\cT}$ are distinct, by Lemma \ref{lem-distinct-tube}, the points $(R\xi',t_Q)$ corresponding to $x(\cT;t_1)$ (as in \eqref{xT}) are $R^{1/2}$-separated for all $\cT\in\boldsymbol{\cT}$.
Now consider $(x',x_n)$ as a function of $(\xi',t_0)$ as in \eqref{geodesic}.
Calculate
\begin{equation}
\label{jacobian}
    \partial_{\xi',t_0}(x',x_n)=\Big(\begin{array}{cc}
    -2(t_1-t_0)I_{n-1} & 2(t_1-t_0)(\xi')^T \\
    2\xi' &  -|\xi'|^2
    \end{array}\Big)+E,
\end{equation}
where the error term $E$ is a matrix such that 
\begin{equation}
\nonumber
    E=(10n)^{-10}\Big(\begin{array}{cc}
    O(|t_1-t_0|)\cdot I_{n-1} & O(|t_1-t_0|\cdot|\xi|)\cdot (\vec{1}_{n-1})^T \\
    O(|\xi'|^2)\vec{1}_{n-1} &  O(|\xi'|^2)
    \end{array}\Big).
\end{equation}
Here $\vec{1}_{n-1}$ is the vector $(1,\ldots,1)\in\ZR^{n-1}$.

Therefore, since $|t_1-t_Q|\gtrsim R\Kc^{-1}$ and since $|\xi'|\gtrsim K^{-1}$, the Jacobian \eqref{jacobian} shows that the points $\{x(\cT;t_1)\}_{\cT\in\boldsymbol{\cT}}$ are $R^{1/2}(\Kc K)^{-O(1)}$-separated, which implies that $\{B(\cT;t_1)\}_{\cT\in\boldsymbol{\cT
}}$ are $(K\Kc)^{O(1)}$-overlapping.
\qedhere

\end{proof}

\medskip

As a corollary, we have the result for planks.
\begin{lemma}\label{hairbrush1}
Suppose $\Tau_0$ is the $R^{1/2}$-neighborhood of a $R$-plank $T_0$ and $\bT=\{T\}$ is a set of distinct $R$-planks contained in $B_R^{n+1}$. 
Let 
    \[\om_1=\R^n\times [i_1R\Kc^{-1},(i_1+1)R\Kc^{-1}],\ \om_2=\R^n\times [i_2R\Kc^{-1},(i_2+1)R\Kc^{-1}]\] be two horizontal regions with $i_1,i_2\in\ZZ, |i_1|,|i_2|\le \Kc$ and $|i_1-i_2|\ge 2$. 
Suppose for all $T\in\bT$, the direction of $\Tau_0$ and the direction of $T$ are $K^{-1}$-separated, and $\Tau_0\cap T\neq\emptyset$ with $\Tau_0\cap T\subset \om_1$.
Then for all $x\in\om_2$,
\begin{equation}
\label{hairbrushineq}
    \sum_{T\in\bT}\Id_T(x)\lesssim (K\Kc)^{O(1)}.
\end{equation} 

\end{lemma}
\begin{proof}
Note that we can cover $\cup_\ZT$ by a family of distinct $R$-tubes such that 
\begin{enumerate}
    \item Each $R$-plank $T$ is contained in $O(1)$ many $R^{\de}$-dilate of $R$-tubes.
    \item For each $R$-tube, the $R$-planks inside the $R^{\de}$-dilate of it are finitely overlapping.
\end{enumerate}
By partitioning $\cT_0$ into $R^{O(\de)}$ many $R$-tubes, we then apply Lemma \ref{lemhairbrush} to conclude the proof. 
\end{proof}

\bigskip

\subsection{Transversality of normal directions, and a trilinear estimate of slabs}

For each $K^2$ ball $B$, let $\ZT(B)$ be the set of planks intersecting $B$.
Pick a point $z\in B$.
Recall \eqref{single-wpt} that for each plank $T\in\ZT(B)$, $(\frac{d}{dt}\gal(z;\xi_\theta),1)$ is its direction. 
We will show that, if the directions of $\ZT(B)$ are broad, then so are the normal directions, which are essentially given by $\{\partial_z\phi(z;\xi_\theta)\}$.

Let us first prove a standard lemma, which provides an upper bound on $\#\{\tau\in\Theta_{K^{-1}}:\tau\text{ lies near a linear subspace}\}$.
It appears in various references, for example, in the display below (7.6) in \cite{gao2023improved}. 
Nevertheless, we include the proof here for completeness.

\begin{lemma}
\label{numberofcaps-lem}
For any $(k-1)$-dimensional subspace $V\subset \R^{n+1}$, we have
\begin{equation}\label{numberofcaps}
    \#\{\tau\in\Theta_{K^{-1}}:\ang(G^\la(z_B,\tau),V)\le K^{-2}\}\lesssim \max\{1,K^{k-3}\}.
\end{equation}
\end{lemma} 

\begin{proof}
Without loss of generality, we assume $z_B=0$.
Via change of variables, we also assume $\phi$ is of reduced form as in the proof of Lemma \ref{lemhairbrush}, equation \eqref{standard-form}:
    \[\phi(x,t;\xi)=\langle x,\xi\rangle +t\frac{|\xi'|^2}{\xi_n}+\wt\cE(x,t;\xi), \]
where $|\wt\cE(x,t;\xi)|=O(|t||\xi'|^3+|z|^2|\xi'|^2)$.
Applying $\frac{d}{dt}$ to $\partial_\xi\phl(\gal(u,t;\xi),t;\xi)=u$, one obtains
    \[ \partial^2_{z\xi}\phl\cdot(\frac{d}{dt}\gal,1)=0. \]
Hence, $(\frac{d}{dt}\gal(z;\xi),1)$ is parallel to $G^\la(z;\xi)$. 
One can compute as in \eqref{geodesic} that
    \[\frac{d}{dt}\gal(z_B;\xi',1)=\frac{d}{dt}\gal(0;\xi',1)=(-2\xi',|\xi'|^2). \]
Since $(-2\xi',|\xi'|^2, 1)$ is parallel to $G^\la(z_B;\xi)$, the left-hand side of \eqref{numberofcaps} is bounded above by
    \[\#\{\tau\in\Theta_{K^{-1}}: (-2\xi_\tau',|\xi_\tau'|^2,1)\in N_{CK^{-2}}(V) \}. \]
Let $\mathcal S^{n-1}:=\{(-2\xi',|\xi'|^2,1):|\xi'|\le 100^{-1}\}$.
Since $\mathcal S^{n-1}$ is away from $0$,
    \[N_{CK^{-2}}(V)\cap \mathcal S^{n-1}\subset N_{C^2K^{-1}}(V\cap \mathcal S^{n-1}). \]
Since $V\cap \mathcal S^{n-1}$ is at most $k-3$ dimensional and since $\{(-2\xi'_\tau,|\xi'_\tau|^2,1)\}_{\tau\in\Theta_{K^{-1}}}$ are $K^{-1}$-separated,
we have
    \[\#\{\tau\in\Theta_{K^{-1}}: (\xi_\tau',|\xi_\tau'|^2,1)\in N_{CK^{-2}}(V) \}\lesssim \max\{1,K^{k-3}\}. \qedhere \]
\end{proof}

\medskip

\begin{lemma}
\label{radial-normal-lem}
There exists a constant $A$ depending on $\phi$ and $n$ such that the following is true:
Let $z\in\ZR^{n+1}$, and let $\cT\subset\Theta_{K^{-1}}$ be a family of $K^{-1}$-caps such that for any two distinct $\tau_i,\tau_j\in\cT$,
\begin{equation}
\label{enormalt}
    |G^\lambda(z,\tau_i)\wedge G^\lambda(z,\tau_j)|
    \geq K^{-2}.
\end{equation}
Here $G^\la$ is the rescaled Gauss map defined in \eqref{generalized-gauss-map}.
Define $v_\xi:=\partial_z\phi(z;\xi)$.

Suppose $\#\cT\geq A$. 
Then there exists 
$\tau_1,\tau_2,\tau_3\in\cT$ such that for all $\xi_j\in \tau_j$,
\[
|v_{\xi_1}\wedge v_{\xi_2}\wedge v_{\xi_3}|\gtrsim K^{-O(1)}.
\]
\end{lemma}

\begin{proof}

Since $\phi$ is homogeneous in $\si$ and obeys the quantitative $(\text{H2}_{\bA})$, without loss of generality, we may assume $z=0$ and write $\partial_z \phi(0, 0; \xi', 1)$ as the graph of a function $h: \xi' \mapsto h(\xi')$ satisfying  
\begin{equation}
\label{hid}
    \Vert \partial_{\xi'\xi'}^2 h - I_{n-1} \Vert_{\text{op}} < \epsilon,
\end{equation}
where $\epsilon > 0$ is a sufficiently small fixed constant, and $\phi$ obeys  
\begin{equation}
\label{qua}
    h(\xi') = \frac{1}{2} |\xi'|^2 + O(|\xi'|^3).
\end{equation}
By homogeneity, $\partial_z \phi(0, 0; \xi) $ can be parameterized as  
\[
(\xi', \xi_n, h(\xi'/\xi_n) \xi_n).
\]
Thus, $G^\la(0,\xi)$ is parallel to
\[
    G_0^\lambda(\xi) := \left( \nabla h(\xi'/\xi_n), - \nabla h\left( \frac{\xi'}{\xi_n} \right)\cdot \frac{\xi'}{\xi_n}+h(\f{\xi'}{\xi_n}), -1 \right).
\]
Denote by 
\[
v(\xi) = (\xi', \xi_n, h(\xi'/\xi_n) \xi_n).
\]

By homogeneity, we may take $\xi=(\xi',1)$ and obtain
\[
    G_0^\lambda(\xi) = (\nabla h(\xi'), -\nabla h(\xi') \cdot \xi'+h(\xi'), -1),
\]
and
\begin{equation}\label{evs}
   v(\xi) = (\xi', 1, h(\xi')). 
\end{equation}
By \eqref{hid}, we know that
    \[|G_0^\lambda(\xi_i)\wedge G_0^\lambda(\xi_j)|\lesssim |\xi_i'-\xi_j'|.\]
Therefore, by the assumption \eqref{enormalt}, for any two distinct $\tau_i,\tau_j\in\cT$ and all $(\xi_i',1)\in\tau_i,\,(\xi_j',1)\in\tau_j$, we have 
\begin{equation}
\nonumber
    |\xi_i'-\xi_j'|\gtrsim K^{-2},
\end{equation}
which further implies that for all distinct  $\tau_i,\tau_j\in\cT$ and all $\xi_i\in\tau_i,\,\xi_j\in\tau_j$, the angle between $(\xi_i',1)$ and $(\xi_j',1)$ is $\gtrsim K^{-2}$. 
Moreover, since $ |h(\xi')|=O(1)$, we have
\begin{equation}
\label{evKs}
    \angle\big(v(\xi_i),v(\xi_j)\big)\gtrsim K^{-2}
\end{equation}
We refer to \cite[Lemma 4.6]{guth2019sharp} for a similar argument. 

\smallskip

Let $v(\tau)=\{v(\xi):\xi\in\tau\}$ and let $\xi_\tau$ be the center of $\tau$.
Since $|v(\xi)|\sim 1$, for any 2-dimensional subspace $V$, we have 
\begin{align*}
        &\#\{\tau\in\Theta_{K^{-1}}:\ang(v(\tau),V)\le K^{-2}\}\\
        \leq C_1&\#\{\tau\in\Theta_{K^{-1}}: ((\xi_{\tau_i})',1, h((\xi_{\tau_i})'))\in N_{C_2K^{-2}}(V) \}
\end{align*}
for some constants $C_1, C_2$.
By assumption \eqref{qua}, we can easily modify the argument in Lemma~\ref{numberofcaps-lem} to obtain from \eqref{evKs} that 
    \[\#\{\tau\in\Theta_{K^{-1}}: ((\xi_{\tau_i})',1, h((\xi_{\tau_i})'))\in N_{C_2K^{-2}}(V) \}\leq C
\]
for some constant $C$ depending on $\phi$ and $n$. 
Thus, if $A$ is chosen to be larger than $10CC_1$, we can find three caps $\tau_1,\tau_2,\tau_3$ such that for all $\xi_j\in \tau_j$, we have $
|v_{\xi_1}\wedge v_{\xi_2}\wedge v_{\xi_3}|\gtrsim K^{-O(1)}$ as desired. 
\qedhere

\end{proof}

\medskip

Recall \eqref{single-wpt}.
Next, we discuss the intersection of three $R$-planks
\[T_i=\{(\ga^\lambda(u,t;\xi_i),t):\ u\in R^\de T_i^{\flat}, |t|\le R\},\hspace{.3cm}1\leq i\leq 3.\]

\begin{lemma}
\label{mka}
Let $\vec v_i=\partial_z\phi^\lambda(z;\xi_i)$. Suppose
%For curved planks 
%$T_i=\{(\ga^\lambda(u,t;\xi_i),t):\ u\in R^\de T_i^{\flat}, |t|\le R\}$, $1\le i\le 3$, satisfying 
\begin{equation}\label{evt}
|\vec v_1\wedge\vec v_2\wedge \vec v_3|\ge K^{-1}.
\end{equation}
Then
%    Let $2\le k\le n+1$. Suppose $T_i,\ i=1,\dots, k$ are $K^{-1}$-transverse. Then
    \begin{equation}\label{ekk}
        |T_1\cap T_2\cap  T_3|\lesssim  K^{O(1)} R^{\frac{n-2}2}.
    \end{equation}
    
\end{lemma}

\begin{proof}%[Proof of Proposition~\ref{mka}]

Without loss of generality, we may assume that for all $i=1,2,3$, $T_i$ is centered at $(u,t)=(0,0)$.
That is, the core curve of each $T_j$ intersects $(0,0)$.

Next, we show that $T\cap B_{R^{1/2}}$ is comparable to a slab of dimensions $R^{2\de}\times  R^{1/2}\times \cdots \times R^{1/2}$.
The proof is a refinement of Lemma~\ref{geometric}.
Express $\gamma^\lambda (u,t;\xi)$ as 
    \[\ga^\lambda(u,t;\xi)-\ga^\lambda(0,t;\xi)+\ga^\lambda(0,t;\xi)-\ga^\lambda(0,0;\xi)+\gamma^\lambda(0,0;\xi).\]
As we have seen in the proof of Lemma~\ref{geometric}, since $u\in T^\flat$, and $R\le \lambda^{1-\epsilon}$, we have $\ga^\lambda(u,t;\xi)-\ga^\lambda(0,t;\xi)=\partial_u\ga^\la(0,t;\xi)\cdot u+O(1)$. 
Notice that
    \[\partial_u\ga^\lambda(0,t;\xi)-\partial_u\ga^\lambda(0,0;\xi)=O(\lambda^{-1}t)=O(R^{-1/2})\]
whenever $|t|\le R^{1/2}$.
Consequently, we obtain
    \[\ga^\lambda(u,t;\xi)-\ga^\lambda(0,t;\xi)=\partial_u\gamma^\lambda(0,0;\xi)\cdot u+O(1).\]
Similarly, as $\|(\partial_t)^2\gamma^\lambda\|_\nf\lesssim \lambda^{-1}$, we have
    \[\ga^\lambda(0,t;\xi)-\ga^\lambda(0,0;\xi)=\partial_t \gamma^\lambda (0,0;\xi)t +O(1).\]
In summary, we know that $T\cap B^{n+1}_{R^{1/2}}$ is comparable to the set
    \[\{\Big(\partial_t\gamma^\lambda(0,0;\xi)t+\partial_u\gamma^{\lambda}(0,0;\xi)\cdot x,t\Big):\ \ x\in T^\flat,\ |t|\le \sqrt R\},\]
which is a slab in $\mathbb R^{n+1}$ of dimensions $R^{2\de}\times R^{1/2}\times\cdots \times R^{1/2}$.

\smallskip 

To describe a normal vector of the slab $T\cap B^{n+1}_{R^{1/2}}$, first observe that in the $\mathbb R^n$-space, the vector $\xi$ is normal to $T^\flat$, which implies that a normal vector $v$ of $T\cap B^{n+1}_{R^{1/2}}$ is of the form $([\partial_u\gamma^\lambda (0,0;\xi)]^{-tran}\cdot\xi,0)+ae_{n+1}$. 
Here $[\partial_u\gamma^\lambda (0,0;\xi)]^{-tran}$ is the transpose of $[\partial_u\gamma^\lambda (0,0;\xi)]^{-1}$, and $e_{n+1}=(0,\dots,0,1)\in\mathbb R^{n+1}$. 
Since $v$ is orthogonal to $\partial_t\gamma^\lambda (0,0;\xi)$, we obtain $a=-\langle [\partial_u\gamma^\lambda (0,0;\xi)]^{-tran}\cdot\xi,\partial_t\gamma^\lambda (0,0;\xi)\rangle$. 
Taking partial derivatives of \eqref{transformation-id} in $u$ and $t$ separately, we get
    \[\partial_u\gamma^\lambda=(\partial_x\partial_\xi \phi^\lambda)^{-1},\hspace{.3cm}\partial_t\gamma^\lambda =-(\partial_x\partial_\xi \phi^\lambda)^{-1}\cdot \partial_t\partial_\xi\phi^\lambda.\]
Therefore, the vector
    \[v=((\partial_x\partial_\xi \phi^\lambda)^{tran}(0,0;\xi)\cdot\xi,0)+\langle \xi, \partial_t\partial_\xi\phi^\lambda (0,0;\xi)\rangle e_{n+1}\]
is normal to the slab $T\cap B^{n+1}_{R^{1/2}}$. 
As $\phi^\lambda$ is $1$-homogeneous in $\xi$, we have 
%Using that \todo{Why this is true?}
$$
\partial_z\phi^\lambda=\lim_{s\to 1^-}\frac{\partial_z\phi^\lambda(z;\xi) -\partial_z\phi^\lambda(z;s\xi) }{1-s}=\partial_\xi\partial_z\phi^\lambda\cdot\xi,
$$
which implies further that
$$
\vec v=\partial_z\phi^\lambda(0,0;\xi).
$$

By above argument, we see that $\vec v_i$ is normal to $T_i\cap B^{n+1}_{\sqrt R}$, therefore, using the condition \eqref{evt} and that $|\vec v_j|\lesssim 1$, we obtain 
\eqref{ekk}.
\qedhere

\end{proof}

\medskip

Finally, we can state the trilinear estimate we need.
\begin{lemma}
\label{trilinear-lem}
Let $Q$ be an $R^{1/2}$-ball, and let $\ZT(Q)$ be a family of $R$-planks intersecting $Q$.
Let $X\subset Q$ be a union of $K^2$-balls.
Let $\mu\geq1$.
Suppose for each $K^2$-ball $B\subset X$, there exists three distinct caps $\tau_1,\tau_2,\tau_3\in\Theta_{K^{-1}}$ such that
\begin{enumerate}
    \item $\#\ZT_{\tau_j}(B)\geq\mu$ for all $j=1,2,3$, where $\ZT_{\tau_j}(B):=\{T\in\ZT(Q):T\cap B\not=\varnothing, \,T\in\ZT_\theta\text{ for some }\theta\subset \tau_j\}$.
    \item For all $\xi_j\in \tau_j$, $|v_{\xi_1}\wedge v_{\xi_2}\wedge v_{\xi_3}|\geq C$.
\end{enumerate}
Then we have
\begin{equation}
\nonumber
    |X|\lesssim (CK)^{O(1)}  R^{\frac{n-2}2}(\#\ZT(Q))^3\mu^{-3}.
\end{equation}
\end{lemma}
\begin{proof}
Let $\ZT_{\tau_j}(Q):=\{T\in\ZT(Q):T\in\ZT_\theta\text{ for some }\theta\subset \tau_j\}$
We say three caps $\tau_1,\tau_2,\tau_3$ are transverse if $|v_{\xi_1}\wedge v_{\xi_2}\wedge v_{\xi_3}|\geq C$ for all $\xi_j\in \tau_j$.
Notice that
\begin{equation}
\nonumber
    \mu^3|X|\lesssim K^{O(1)}\int_X\sum_{T_j\in\ZT_{\tau_j}(B)}|T_1\cap T_2\cap T_3|\lesssim K^{O(1)}\sum_{\substack{\tau_1,\tau_2,\tau_3\text{ are}\\\text{ transverse}}}\sum_{T_j\in\ZT_{\tau_j}(Q)}|T_1\cap T_2\cap T_3|.
\end{equation}
By Lemma \ref{mka}, we have
\begin{equation}
\nonumber
    \mu^3|X|\lesssim (CK)^{O(1)}R^{\frac{n-2}2}\sum_{\tau_1,\tau_2,\tau_3}\prod_{j=1}^3\#\ZT_{\tau_j}(Q)\lesssim  (CK)^{O(1)}  R^{\frac{n-2}2}(\#\ZT(Q))^3,
\end{equation}
which implies the lemma. \qedhere
\end{proof}

\bigskip

\subsection{A covering lemma}

In this subsection, we focus on the geometry inside each $R^{1/2}$-ball.
Although the main result is stated for all dimensions, we will only apply it in the context of 3+1 dimensions.

Suppose $T$ is an $R$-plank and $Q$ is an $R^{1/2}$-ball. 
By Taylor's expansion of the core curve of $T$ inside $Q$, one sees that $T\cap Q$ is contained in a $R^{2\de}\times R^{1/2}\times \dots\times R^{1/2}$-slab (see the proof of Lemma \ref{mka}).
We will prove a covering lemma for $K^2\times R^{1/2}\times \dots\times R^{1/2}$-slabs inside an $R^{1/2}$-ball. 
As before, $K=R^{\e^{50}}$.

\begin{definition}
\label{rho-regular}
Let $R\geq1$ and let $\rho$ be such that $1\le\rho\le R^{O(1)}$. 
Let $\bB$ be a set of disjoint $K^2$-balls inside an $R^{1/2}$-ball $Q\subset\ZR^{n+1}$.
We say $\bB$ is {\bf{$\rho$-regular}} with respect to $K^2\times R^{1/2}\times \dots\times R^{1/2}$-slabs in $Q$ if
\begin{enumerate}
    \item For any $K^2\times R^{1/2}\times \dots\times R^{1/2}$-slab $S$ in $Q$,
    \begin{equation}
    \nonumber
        \#\{ B\in\bB:B\cap S\neq\emptyset  \}\le 2 \rho.
    \end{equation}
    \item There exists a collection $\ZS $ of $K^2\times R^{1/2}\times \dots\times R^{1/2}$-slabs  with $\#\ZS\lesssim \#\bB/\rho$ so that the balls in $\bB$ all intersect with $ \bigcup_{S\in\S}S$, and hence are all contained in $ \bigcup_{S\in\S}5S$.
\end{enumerate}
When $Q$ is clear and $K,R$ are fixed, we simply call $\bB$ $\rho$-regular.
\end{definition}

\smallskip

Definition \ref{rho-regular} is similar to Definition 1.9 in \cite{Li-Wu}.
The next Lemma \ref{regular-lem} is similar to Lemma 2.1 in the same paper.

\begin{lemma}
\label{regular-lem}
Suppose $\bB$ is a disjoint union of $K^2$-balls inside an $R^{1/2}$-ball $Q\subset\ZR^4$.
Then there exists a number $\rho$ and a subset $\bB'\subset \bB$ with $\#\bB'\gtrsim (\log R)^{-1}\#\bB$, such that $\bB'$ is $\rho$-regular.
\end{lemma}

\begin{proof}
The proof depends on a greedy algorithm.
The algorithm is designed to construct a partition of $\bB$ into $O(\log R)$ subsets $\{\bB_{\rho}\}_\rho$ where $\rho$ are dyadic numbers in $[1,R^2]$, and $\bB_\rho$ is $\rho$-regular.

Let $\bB_1=\bB$. 
We will construct a sequence $S_1,S_2,\dots$, along with $\rho_1,\rho_2,\dots$ in the following way. 
Let $\bB_j=\bB\setminus\Big( \bigcup_{i\le j-1} \{B\in\bB: B\cap S_i\neq\emptyset\} \Big)$. 
One sees that $\bB_{j}=\bB_{j+1}\sqcup \{B\in \bB_{j}: B\cap S_j\neq\emptyset\}$.
Define 
\begin{equation}
\nonumber
    \rho_{j}=\max_{ S\text{ is a } K^2\times R^{1/2}\times \dots\times R^{1/2}\text{-slab}}\#\{B\in \bB_j: B\cap S\neq\emptyset\}.
\end{equation}
We use $S_j$ to denote the $K^2\times R^{1/2}\times \dots\times R^{1/2}$-slab that attains the maximum. 

We keep doing until $\bB_{m+1}=\emptyset$ at certain step. By the definition, we have the monotonicity $\rho_i\ge \rho_{i+1}$ for any $i$. 
For each dyadic $\rho\in[1,R^{O(1)}]$, we let $I_\rho\subset [1,m]\cap \mathbb N$ be the indices such that $\rho_j\in[\rho,2\rho)$ for any $j\in I_\rho$. 
We obtain a partition of indices
    \[[1,m]\cap \mathbb N=\bigsqcup_{\rho} I_\rho. \]
For each dyadic $\rho$, we inductively define
    \[\bB_\rho=\{B\in \bB: B\cap \bigcup_{i\in I_\rho}S_i\neq\emptyset \}\setminus \Big( \bigcup_{\rho'>\rho}\bB_{\rho'} \Big). \]
Here, $\bigcup_{\rho'>\rho}$ ranges over those dyadic numbers $\rho'$ bigger that $\rho$. 
Thus, 
    \[\bB=\bigsqcup_{\rho}\bB_\rho. \]

By pigeonholing, there exists $\rho$ such that $\#\bB_\rho\gtrsim (\log R)^{-1}\#\bB$. 
We claim that $\bB_\rho$ is $\rho$-regular.
By the definition of $\rho_j$, we see that for any $K^2\times R^{1/2}\times\dots\times R^{1/2}$-slab $S$, $\#\{B\in\bB_\rho: B\cap S\neq\emptyset\}\le \sup_{i\in I_\rho} \rho_i\le 2\rho.$ 
Note that
    \[\bB_\rho= \bigcup_{j\in I_\rho}\bB_j\setminus \bB_{j+1} =\bigsqcup_{j\in I_\rho} \{B\in\bB_j: B\cap S_j\neq \emptyset\}. \]
We have
    \[\sum_{i\in I_\rho}\rho_i= \sum_{j\in I_\rho}\{B\in\bB_j: B\cap S_j\neq \emptyset\} = \#\bB_\rho. \]
Therefore, $\# I_\rho\lesssim \#\bB_\rho/\rho$, and all the balls in $\bB_\rho$ intersect with $\bigcup_{i\in I_\rho}S_i$.
\qedhere

\end{proof}

\begin{remark}

\label{covering-lem-rmk}
\rm

Note that the proof of Lemma \ref{regular-lem} does not rely on any geometric properties of the $K^2\times R^{1/2}\times \dots\times R^{1/2}$-slabs.
That is to say, a similar covering lemma is expected if the family of $K^2\times R^{1/2}\times \dots\times R^{1/2}$-boxes is replaced by another family of geometric objects (for example, a family of tubes, which was used in \cite{Li-Wu}).

\end{remark}

\bigskip

\section{An algorithm}\label{section-algorithm}

In this section, we present (a single step of) an algorithm designed to establish the broad \& two-ends structure on the sum of wave packets $\sum_T\cFl f_T$ mentioned in the introduction. 
Fix $R\ge 1$ and recall our choice of parameters $K=R^{\e^{50}},\Kc=R^{\e^{100}}$. We may assume $R^{1/2},K,\Kc\in\N$. 
Let $\bT$ be a family of $R$-planks and $\bB$ be a family of $K^2$-balls in $B_R^{n+1}$. By slightly modify the parameters (up to a constant multiple), we may assume $K^2\vert R^{1/2}\vert R\Kc^{-1}$.

Let $\Om=\{\om\}$ be a collection of horizontal regions where each $\om$ is of the form \[\R^n\times [iR\Kc^{-1},(i+1)R\Kc^{-1})\] where $i\in\ZZ\cap[-\Kc,\Kc-1]$), so $\Om$ forms a disjoint cover of $B_R^{n+1}$. In the algorithm, we will also work with $\{Q\}$, a set of $R^{1/2}$-balls. We assume each $Q$ is of form \[\prod_{j=1}^{n+1}[i_jR^{1/2},(i_j+1)R^{1/2}],\] where $i_j\in\ZZ\cap [-R^{1/2},R^{1/2}-1]$.
We may assume each $B\in \bB$ is of form \[\prod_{j=1}^{n+1}[i_jK^2,(i_j+1)K^2],\] where $i_j\in \ZZ\cap [-RK^{-2},RK^{-2}-1]$. By our assumption, each $B$ is contained in one $Q$, and each $Q$ is contained in one $\om$. 
Given $T\in\bT$, we have
    \[\{B\in\bB: B\cap T\neq\emptyset\}=\bigsqcup_{\om\in\Om}\{B\in\bB:B\cap T\neq\emptyset, B\subset \om\}. \]
Since the two generic regions in $\Om$ are $R\Kc^{-1}$-separated, the following two statements are conceptually equivalent:
\begin{enumerate}
    \item The $K^2$-balls in $\bB$ intersecting $T$ are from many different regions in $\Om$.
    \item The shading $Y(T)=\cup_{\bB}\cap T$ satisfies a two-ends condition.
\end{enumerate}

\smallskip

To formulate the two-ends property, we introduce the following notion.

\begin{definition}[Shaded incidence triple]
Let $R\ge 1$, let $K=R^{\e^{50}}$, and let $\Om$ be the set of horizontal regions introduced above. 
A \textbf{shaded incidence triple} (or simply triple) $(\bB,\bT;\cG)$ is the following:
\begin{enumerate}
    \item $\bB=\{B\}$ is a set of $K^2$-balls in $B_R^{n+1}$;
    \item $\bT=\{T\}$ is a set of $R$-planks in $B_R^{n+1}$;
    \item $\cG$ is the \textbf{shading map} that $\cG:\bT\to P(\Om)$. Here $P(X)$ is the collection of subsets of $X$, which is also called the power set of $X$. 
\end{enumerate}
\end{definition}

We remark that for a given plank $T$, $\cG(T)$ is a subset of $\Om$.
For a triple, we will be interested in the following incidence:
    \[\cI(\bB,\bT;\cG):=\#\{ (B,T)\in\bB\times \bT: B\cap T\neq\emptyset, B\subset \cup_{\cG(T)} \}. \]
We also define for any $B\in\bB$,
    \[\bT(B;\cG):=\{ T\in\bT: B\cap T\neq\emptyset, B\subset \cup_{\cG(T)} \}, \]
and for any $T\in \bT$,
    \[\bB(T;\cG):=\{B\in\bB: B\cap T\neq\emptyset, B\subset \cup_{\cG(T)}\}. \]
By a double-counting argument, one has
\begin{equation}
\label{double-counting}
    \cI(\bB,\bT;\cG)=\sum_{B\in\bB}\#\bT(B;\cG)=\sum_{T\in\bT}\#\bB(T;\cG).
\end{equation}

To formulate the broad property, we introduce the following notions.
\begin{definition}
\label{T[S]}
For an $R$-plank $T$, we use $\theta(T)\in \Theta_{R^{-1/2}}$ to denote its directional cap.
Let $S$ be a union of $\theta\in\Theta_{R^{-1/2}}$ and let $\bT$ be a set of $R$-planks. We define 
    \[\bT[S]:=\{T\in\bT: \theta(T)\subset S\}. \]
\end{definition}
\noindent Here are two examples:
\begin{enumerate}
    \item $ \bT[\theta]=\{ T\in\bT: \theta(T)=\theta \}$.
    \item For $\tau\in\Theta_{K^{-1}}$,  $\bT[(3\tau)^c]=\{T\in\bT: \theta(T)\not\subset 3\tau\} $.
\end{enumerate}

\begin{definition}[Broad incidence]\label{defbbr}
Given a triple $(\bB,\bT;\cG)$ and an integer $A\ge 1$, define 
    \[\cI_{\bbr,A}(\bB,\bT;\cG):=\inf_{\cT\subset \Theta_{K^{-1}},\#\cT=A}\cI(\bB,\bT[(\cup_{\tau\in\cT}3\tau)^c];\cG). \]
For a set of $R$-planks $\bT$, we define the broad cardinality as
    \[\#_{\bbr,A} \bT:=\inf_{\cT\subset \Theta_{K^{-1}},\#\cT=A}\#\bT[(\cup_{\tau\in\cT}3\tau)^c]. \]
When $A=1$, we write $\cI_{\bbr}(\bB,\bT;\cG)=\cI_{\bbr,A}(\bB,\bT;\cG)$ and $\#_{\bbr} \bT=\#_{\bbr,A} \bT$ in short.
\end{definition}

\medskip

{\bf Algorithm:}

\bigskip

Now we can present the algorithm. 
It allows us to refine a given triple $(\bB,\bT,\cG)$, where $\bT$ corresponds to the set of wave packets and $\bB$ corresponds to the integration domain.
At the end of the algorithm, we will obtain a ``refinement" of $(\bB,\bT,\cG)$ along with several associated parameters.
This will be summarized in Proposition \ref{propalgorithm}. We will use the notation: for two finite sets $E\subset F$, we say $E$ is a \textbf{refinement} of $F$ if $\# E\gtrapprox \#F$.

\medskip

\noindent\textit{Initial stage:} Given an incidence triple $(\bB,\bT;\cG)$ and a $\nu>0$ such that for $B\in\bB$,
\begin{equation}
\label{to-pigeonhole}
    \big\|\sum_{T\in\bT(B;\Si,\cG)}\cFl f_T\big\|_{L^p(B)}\sim \nu.
\end{equation}

\medskip

\noindent
\textit{Step 1: Broad multiplicity in each \texorpdfstring{$K^2$}{}-ball $B\in\bB$.}

Let $A\geq1$ be an integer.
By dyadic pigeonholing on $\{\#_\bbr \bT(B;\cG):B\in\bB\}$,
we can find a refinement $\bB_1$ of $\bB$ and a $\mu\in  [1,R^{O(1)}]\cup\{0\}$ such that 
\begin{equation}
\label{B1}
\begin{split}
    &\#_{\bbr,A} \bT(B;\cG)\in[\mu,2\mu) \textup{~when~}\mu\ge1,\\
    &\#_{\bbr,A} \bT(B;\cG)=0 \textup{~when~}\mu=0,
\end{split}
\end{equation}
for all $B\in\bB_1$. 
Note that
\begin{equation}
\label{item3}
    \cI(\bB_1,\bT;\cG)=\sum_{B\in\bB_1}\#\bT(B;\cG)\ge \#\bB_1 \mu.
\end{equation}

\medskip
\noindent
\textit{Step 2: Two-ends reduction.}

From Step 1, we obtain a triple $(\bB_1,\bT;\cG)$.
For a plank $T\in\bT$ and a dyadic number $1\le \lambda \le R^{O(1)}$, we define 
\begin{equation}\label{Gla}
    \cG_\lambda(T):=\{ \om\in\cG(T):  \#\{B\in\bB_1(T;\cG):B\subset \om\}\in[\lambda,2\lambda) \}.
\end{equation}
Note that $\cG_\lambda$ is a new shading map.
Since
$\cG(T)=\bigsqcup_\lambda \cG_\lambda(T)$, we have for each $B$,
    \[\bT(B;\cG)=\bigsqcup_{\la}\bT(B;\cG_\la), \]
and hence
\begin{equation}
\nonumber
\begin{split}
    \sum_{T\in\bT(B;\cG)}\cFl f_T=\sum_{\lambda}\sum_{T\in\bT(B;\cG_\lambda)}\cFl f_T.
\end{split}
\end{equation}
By pigeonholing, there is a $\lambda(B)$ such that
    \[\nu\sim\big\|\sum_{T\in\bT(B;\cG)}\cFl f_T\big\|_{L^p(B)}\lessapprox \big\|\sum_{T\in\bT(B;\cG_{\la(B)})}\cFl f_T \big\|_{L^p(B)}. \] 
By dyadic pigeonholing on the tuples
    \[\bigg\{ \Big(\lambda(B), \big\|\sum_{T\in\bT(B;\cG_{\la(B)})}\cFl f_T \big\|_{L^p(B)}\Big) \bigg\}_{B\in\bB_1}, \]
there exist $\bB_2$, $\lambda$, and $\nu_1$, such that the following is true:
\begin{enumerate}
    \item $\bB_2$ is a refinement of $\bB_1$.
    \item $\lambda(B)=\lambda$ for all $B\in\bB_2$.
    \item For all $B\in\bB_2$, $\|\sum_{T\in\bT(B;\cG_{\la})}\cFl f_T \|_{L^p(B)}\sim \nu_1\gtrapprox\nu$ .
\end{enumerate}

\smallskip
Next, we will prune the wave packets $\bT$. 
Consider the partition $\bT=\bigsqcup_\be\bT_\be$, where $\be\in[1,2\Kc]$ ranges over dyadic numbers and 
\begin{equation}
\label{beta}
    \bT_\beta=\{ T\in\bT: \beta\le \#\cG_\lambda(T)<2\beta \}.
\end{equation}
For each $B\in\bB_2$, we have 
    \[\bT(B;\cG_\la)=\bigsqcup_\beta \bT_{\beta}(B;\cG_\la). \]
As a result, for each $B\in\bB_2$,
\begin{equation}
\nonumber
    \sum_{T\in\bT(B;\cG_\la)}\cFl f_T=\sum_\be\sum_{T\in\bT_\beta(B;\cG_\la)}\cFl f_T.    
\end{equation}
By pigeonholing, there is a $\be(B)$ such that 
\begin{equation}
\nonumber
   \big\|\sum_{T\in\bT(B;\cG_\la)}\cFl f_T\big\|_{L^p(B)}\lessapprox \big\|\sum_{T\in\bT_{\beta(B)}(B;\cG_\la)}\cFl f_T \big\|_{L^p(B)}.
\end{equation}
Since $\|\sum_{T\in\bT(B;\cG_\la)}\cFl f_T\|_{L^p(B)}\sim \nu_1$ for all $B\in\bB_2$, by a similar dyadic pigeonholing argument as before, there exist $\bB_3$, $\be$, and $\nu_2$, such that:
\begin{enumerate}
    \item $\bB_3$ is a refinement of $\bB_2$.
    \item $\be(B)=\be$ for all $B\in\bB_3$.
    \item For all $B\in\bB_3$, $\|\sum_{T\in\bT_{\beta}(B;\cG_\la)}\cFl f_T \|_{L^p(B)}\sim \nu_2\gtrapprox\nu_1$.
\end{enumerate}

Note that for each $T\in\bT_\beta$, one has 
\begin{equation}
\nonumber
\begin{split}
     \#\bB_1(T;\cG_\la)=&\,\#\{B\in\bB_1: B\cap T\neq\emptyset,  B\subset\cup_{\cG_\la(T)}\}\\[1ex]
     =&\sum_{\om\in\cG_\la(T)}\#\{B\in\bB_1: B\cap T\neq\emptyset, B\subset \om\}\\
     \le& \sum_{\om\in\cG_\la(T)}\#\{B\in\bB_1(T;\cG): B\subset \om\}.
\end{split}
\end{equation}
By \eqref{Gla} and \eqref{beta}, $\#\{B\in\bB_1(T;\cG): B\subset \om\}<2\la$ when $\om\in\cG_\la(T)$; $\#\cG_\la(T)<2\beta$ when $T\in\bT_\beta$. Therefore,  $\#\bB_1(T;\cG_\la)\lesssim \beta\la$ for $T\in\bT_\beta$, which gives
\begin{equation}\label{item4}
\cI(\bB_1,\bT_\beta;\cG_\lambda)=\sum_{T\in\bT_\beta}\#\bB_1(T;\cG_\la)\lesssim \#\bT_\beta\cdot\beta\la.
\end{equation}

\bigskip

\noindent
\textit{Step 3: Uniformization on each $R^{1/2}$-ball.}

Let $\{Q\}$ be a set of $R^{1/2}$-balls that form a partition of $B_R^{n+1}$.
For each $R^{1/2}$-ball $Q$, consider the set $\bB_3|_Q=\{B\in\bB_3: B\subset Q\}$. By Lemma \ref{regular-lem}, there is a refinement $\bB_3'|_Q$ of $\bB_3|_Q$ and a number $\rho(Q)\ge 1$ such that $\bB_3'|_Q$ is $\rho$-regular. 
Let $\bB_3'=\bigcup_{Q}\bB_3'|_Q$, so $\bB_3'$ is a refinement of $\bB_3$.
For each $R^{1/2}$-ball $Q$, we consider the following quantities:

\begin{enumerate}
    \item $\rho(Q)$.
    \item $\#\bB_3'|_Q$.
    \item $\#_{\bbr}\bigcup_{B\in\bB_4|_Q}\bT_\be(B;\cG_\la)$and $\#_{\bbr} \bigcup_{B\in\bB_4|_Q}\bT(B;\cG_\la)$.
\end{enumerate}
By pigeonholing, we can find a set of $R^{1/2}$-balls $\cQ=\{Q\}$ and numbers $\rho,\si,m, \bar m, \iota$, such that for all $Q\in \cQ$ the following is true:
\begin{enumerate}
    \item $\rho(Q)\sim\rho$.
    \item $\#\bB_3'|_Q\sim \si$.
    \item $\#_{\bbr} \bigcup_{B\in\bB_3'|_Q}\bT_\be(B;\cG_\la)\sim m$ and $\#_\bbr \bigcup_{B\in\bB_3'|_Q}\bT(B;\cG_\la)\sim \bar m$, so $\bar m\gtrsim m\gtrsim\mu$.
    \item $\cI_{\bbr, A}(\bB_3'|_Q,\bT;\cG)\sim \iota$.
    \item $\#\bigcup_{Q\in\cQ}\bB_3'|_Q\gtrapprox \#\bB_3'$.
\end{enumerate}
Let $\bB_4$ be a refinement of $\bB_3'$ defined as
\[\bB_4:=\bigcup_{Q\in\cQ}\bB_3'|_Q.\]
We also define an important quantity:
\begin{equation}\label{defl}
     l:=\sup_{T\in\bT_\beta}\#\{Q\in \cQ: Q\cap T\neq\emptyset \}. 
\end{equation}

\smallskip

What follows is some numerology regarding the parameters we have introduced.
Recall \eqref{B1}.
Hence, for any $B\in\bB_1$ and any $\cT\subset\Theta_{K^{-1}}$ with $\#\cT=A$, 
    \[\#\bT[(\cup_{\tau\in\cT}3\tau)^c](B;\cG)\ge \#_{\bbr, A} \bT(B;\cG) \ge \mu. \]
Note that $\bB_4\subset\bB_1$ and $\#\bB_4|_Q\sim \si$ for all $Q\in\cQ$ (here $\bB_4|_Q=\{B\in\bB_4:B\subset Q\}$).
Therefore, for all $Q\in\cQ$,
\begin{equation}\label{item5}
    \begin{split}
    \cI_{\bbr, A}(\bB_4|_Q,\bT;\cG)=\inf_{\cT\subset \Theta_{K^{-1}},\#\cT=A}\cI(\bB_4|_Q,\bT[(\cup_{\tau\in\cT}3\tau)^c];\cG)\\
    =\inf_{\cT\subset \Theta_{K^{-1}},\#\cT=A}\sum_{B\in\bB_4|_Q}\#\bT[(\cup_{\tau\in\cT}3\tau)^c](B;\cG) \gtrsim \si\mu. 
    \end{split}
\end{equation}

\eqref{item5} is our first estimate.
To establish our third estimate, note that by \eqref{double-counting},
\begin{equation}
\nonumber
    \cI_{\bbr,A}(\bB_4|_Q,\bT_\beta;\cG_\la)\leq \cI_{\bbr}(\bB_4|_Q,\bT_\beta;\cG_\la)=\inf_{\tau\in \Theta_{K^{-1}}}\sum_{T\in \bT_\beta[(3\tau)^c]}\#(\bB_4|_Q)(T;\cG_\la).
\end{equation}
Thus, for any $\tau\in\Theta_{K^{-1}}$, we have
\begin{equation}
\nonumber
    \begin{split}
    \cI_{\bbr, A}(\bB_4|_Q,\bT_\beta;\cG_\la)&\le \sum_{T\in \bT_\beta[(3\tau)^c]}\#(\bB_4|_Q)(T;\cG_\la)\\
    &=\sum_{T\in \bigcup_{B\in\bB_4|_Q}\bT_\beta[(3\tau)^c](B;\cG_\la)}\#(\bB_4|_Q)(T;\cG_\la),
    \end{split}
\end{equation}
as $(\bB_4|_Q)(T;\cG_\la)=\emptyset$ when $T\notin \bigcup_{B\in\bB_4|_Q}\bT_\beta[(3\tau)^c](B;\cG_\la)$.
Choose $\tau$ so that
\begin{align}
    \#\bigcup_{B\in\bB_4|_Q}\bT_\beta[(3\tau)^c](B; \cG_\la)=\#_{\bbr}\bigcup_{B\in\bB_4|_Q}\bT_\beta(B; \cG_\la)\lesssim m.
\end{align}
Note that $\#(\bB_4|_Q)(T; \cG_\la)\le \#\{B\in \bB_4|_Q:B\cap T\neq\emptyset \}$.
Since $\bB_4|_Q$ is $\rho$-regular,  
    \[\#(\bB_4|_Q)(T; \cG_\la)\lesssim \rho. \]
Therefore, we obtain our second estimate
\begin{equation}
\label{item6}
    \cI_{\bbr,A}(\bB_4|_Q,\bT_\beta; \cG_\la)\lesssim \rho m.
\end{equation}

We put the third estimate into a lemma.
\begin{lemma}
\label{si-uppper-bound}
Suppose $A\geq \log K$.
Then we have
\begin{equation}
\label{item-5-half}
    \si\lesssim K^{O(1)}R^{\frac{n-2}2}\bar m^3\mu^{-3}.
\end{equation}
\end{lemma}

\begin{proof}
Since $\#_\bbr \bigcup_{B\in\bB_4|_Q}\bT(B;\cG_\la)\sim \bar m$, there exists a cap $\tau_\ast\in\Theta_{K^{-1}}$ such that, by denoting $\Theta_{K^{-1}}^\ast=\{\tau\in\Theta_{K^{-1}}:\tau\cap 3\tau_\ast=\varnothing\}$, we have 
\begin{equation}
\label{bar-m-upper}
    \#\bigcup_{B\in\bB_4|_Q}\bT[\cup_{\Theta_{K^{-1}}^\ast}](B;\cG)\lesssim \bar m.
\end{equation}
Since $\#_{\bbr,A} \bT(B;\cG) \ge \mu$ for all $B\in \bB_1$, there exists $\cT(B)\subset\Theta_{K^{-1}}^\ast$ with $\#\cT(B)\geq \log K-3n$ so that for all $\tau\in\cT(B)$, 
\[
\#\bT[\tau](B;\cG)\gtrsim K^{-(n-1)}\mu.
\]
As $\log K-3n$ is bigger than any constant depending only on $\phi$ and $n$ for large enough $K$, by Lemma \ref{radial-normal-lem}, there are three caps $\tau_1,\tau_2,\tau_3\in\cT(B)$ such that for all $\xi_j\in \tau_j$,
\[
|v_{\xi_1}\wedge v_{\xi_2}\wedge v_{\xi_3}|\gtrsim K^{-O(1)}.
\]
Since $\bB_4\subset\bB_1$ and since $\#\bB_4|_Q\sim \si$ for all $Q\in\cQ$, applying Lemma \ref{trilinear-lem} with $\ZT(Q)=\bigcup_{B\in\bB_4|_Q}\bT[\cup_{\Theta_{K^{-1}}^\ast}](B;\cG)$ and using \eqref{bar-m-upper}, we prove \eqref{item-5-half}.
\qedhere

\end{proof}

\bigskip

\textbf{At this point, the algorithm stops.}

\bigskip

Before concluding the algorithm as a theorem, we present an incidence result based on the parameters introduced in the algorithm.

\begin{lemma}\label{lem-incidency}
    Suppose $\beta>100$. Then
    \begin{equation}\label{item7}
        \#\bB_1\gtrsim (K\Kc)^{-O(1)}lm\la.
    \end{equation}
\end{lemma}
\begin{proof}
By the definition of $l$ in \eqref{defl}, there exists $T_0\in\bT_\beta$ such that
    \[l=\#\{Q\in \cQ: Q\cap T_0\neq\emptyset \}.\]
By pigeonholing, there exists $\om_1\in\Om$ such that
    \[\#\{Q\in\cQ: Q\cap T\neq\emptyset, Q\subset \om_1\}\gtrsim l\Kc^{-1}.\]

Denote $\cQ'=\{Q\in\cQ: Q\cap T\neq\emptyset, Q\subset \om_1\}$.
Let $\Tau_0$ be the $R^{1/2}$-neighborhood of the $R$-plank $T_0$.
Recall that $\#_\bbr\bigcup_{B\in\bB_4|_Q}\bT_\beta(B;\cG_\la)\sim m$ for each $Q\in\cQ$. 
Let $\tau\in\Theta_{K^{-1}}$ be the cap containing the direction of $\Tau_0$, so
    \[\#\bigcup_{B\in\bB_4|_Q}\bT_\beta[(3\tau)^c](B;\cG_\la)\gtrsim m. \]
Now we define 
    \[\bT'= \bigcup_{Q\in\cQ'}\bigcup_{B\in\bB_4|_Q}\bT_\beta[(3\tau)^c](B;\cG_\la), \]
which can be viewed as a hairbrush with stem $\Tau_0$.
Since the directions of planks $T\in\bT'$ are $K^{-1}$-separated with the direction of $\Tau_0$, each $T\in\bT'$ belongs to $\lesssim K^{O(1)}$ different $\{\bigcup_{B\in\bB_4|_Q}\bT_\beta[(3\tau)^c](B;\cG_\la):Q\in\cq'\}$. Therefore,
    \[\#\bT'\gtrsim K^{-O(1)}\sum_{Q\in\cQ'}\#\bigcup_{B\in\bB_4|_Q}\bT_\beta[(3\tau)^c](B;\cG_\la)\gtrsim (K\Kc)^{-1}lm. \]

Next, we consider the set
    \[\bB'=\bigcup_{T\in\bT'}\bB_1(T; \cG_\la). \]
Note that $\bT'\subset \bT_\beta$ and for each $T\in\bT_\beta$, $\#\cG_\la(T)\ge\beta>100$.
Recall \eqref{Gla}. 
By pigeonholing on the horizontal regions in $\Om$, we can find a uniform $\om_2\in\Om$ with $\dist(\om_1,\om_2)>R\Kc^{-1}$, so that
    \[\#\{B\in\bB_1(T;\cG_\la):B\subset \om_2\}\ge \la \]
holds for $\ge (2\Kc)^{-1}\#\bT'$ many $T\in\bT'$. 
By Lemma \ref{hairbrush1}, $\{T\}_{T\in\bT'}$ are $(K\Kc)^{O(1)}$-overlapping in $\om_2$. As a result, we have
    \[\#\bB'\gtrsim (2\Kc)^{-1}\#\bT' (K\Kc)^{-O(1)}\la\gtrsim (K\Kc)^{-O(1)}lm\la. \]
Therefore, as $\bB'\subset \bB_1$, we have
    \[\#\bB_1\gtrsim (K\Kc)^{-O(1)}lm\la. \qedhere \]

\end{proof}

\bigskip

We summarize the algorithm as a Proposition.

\begin{proposition}\label{propalgorithm}
Let $\bT$ be a set of $R$-planks, $\bB$ be a set of $K^2$-balls in $B_R^{n+1}$, and $\cG$ be a shading map on $\bT$.
Suppose $\{f_T\}_{T\in \bT}$ is a set of wave packets and there exist $\nu>0$ such that for all $B\in\bB$,
    \[\big\|\sum_{T\in\bT(\bB;\cG)}\cFl f_T\big\|_{L^p(B)}\sim \nu. \]
Let $A\geq1$ be an integer.
If $A\geq \log K$, then there exists 
\begin{enumerate}
    \item parameters $\mu, \lambda,\beta, \rho, \si, m, \bar m, \iota, l, \nu'$,
    \item a sequence of refinements  $\bB_4\subset \bB_3\subset \bB_2\subset \bB_1\subset \bB$,
    \item a subset of planks $\bT_\beta\subset \bT$, a shading map $\cG_\lambda$, and a set of $R^{1/2}$-balls $\cQ$,
\end{enumerate}
such that  $\cG_\la(T)\subset \cG(T)$ for all $T\in\bT$, and the following is true:
\begin{enumerate}[(a)]
    \item $\|\sum_{T\in\bT_{\beta}(B;\cG_\la)}\cFl f_T \|_{L^p(B)}\sim\nu'\gtrapprox\nu$ for all $B\in\bB_4$.
    \item $\#\bB_4|_Q\sim \si$, and $\bB_4|_Q$ is $\rho$-regular for all $Q\in\cq$. Here, $\bB_4|_Q$ denotes $\{B\in\bB_4:B\subset Q\}$.
    \item $l=\sup_{T\in\bT_\beta}\#\{Q\in \cQ: Q\cap T\neq\emptyset \}$.
    \item $\bar m\gtrsim m\gtrsim\mu$.
\end{enumerate}
Moreover, we have the following estimates of these parameters:
\begin{enumerate}[(i)]

    \item $\cI(\bB_1,\bT;\cG) \ge \#\bB_1 \mu$.

    \item $\cI(\bB_1,\bT_\beta;\cG_\lambda)\lesssim \#\bT_\beta\beta\la$.

    \item $\iota\sim\cI_{\bbr,A}(\bB_4|_Q,\bT;\cG)\gtrsim \si\mu$ for all $Q\in\cQ$.

    \item $\cI_{\bbr,A}(\bB_4|_Q,\bT_\beta;\cG_\la)\lesssim \rho m $ for each $Q\in\cQ$.

    \item $\be\leq \Kc$, and if $\be>100$, then $\#\bB_1\gtrsim (K\Kc)^{-O(1)}lm\la$.

    \item For all $B\in\bB_1$, $\#_{\bbr,A} \bT(B;\cG)\lesssim\mu$.

    \item For all $T\in\bT_\be$, $\#\cg_\la(T)\lesssim\be$.

    \item $\si\lesssim K^{O(1)}R^{\frac{n-2}2}\bar m^3\mu^{-3}$.
\end{enumerate}

\end{proposition}

\begin{proof}
Item (i) follows from \eqref{item3}. Item (ii) follows from  \eqref{item4}. Item (iii) follows from  \eqref{item5}. Item (iv) follows from  \eqref{item6}. Item (v) follows from  \eqref{item7}.
Item (vi) follows from  \eqref{B1}. Item (vii) follows from \eqref{beta}.
Item (viii) follows from Lemma \ref{si-uppper-bound}.
\end{proof}

\medskip

\begin{remark}
\label{algorithm-remark}
\rm

The incidence estimates in Proposition \ref{propalgorithm} are only useful if we can establish the two bounds
\begin{enumerate}
    \item $\cI(\bB_1,\bT_\beta;\cG_\lambda)\gtrapprox\cI(\bB_1,\bT;\cG)$.
    \item $\cI_{\bbr,A}(\bB_4|_Q,\bT_\beta;\cG_\la)\gtrapprox \cI_{\bbr,A}(\bB_4|_Q,\bT;\cG)$ for all $Q\in\cq$.
\end{enumerate}
However, this is generally not true, mostly because the sum in \eqref{to-pigeonhole} involves oscillatory functions rather than positive functions.
To achieve these bounds, we will repeatedly run the algorithm. As shown in the next section, after a finite number of iterations, we can essentially realize these two bounds.

\end{remark}

\bigskip

\section{Finish the proof: iteration of Proposition \ref{propalgorithm}}
\label{section-iteration}

We will prove Theorem \ref{mixed-norm-thm} in this and the next section. 
For the reader's convenience, we restate it below.

\begin{theorem}
\label{thmmixednorm}
Let $\cf$ be a type $\bfone$ Fourier integral operator given by \eqref{FIO}.
Suppose the phase function $\phi$ satisfies \eqref{phi0}, and the amplitude function $a$ obeys $\supp_\xi\ a\subset\A^n(1)^\circ$.
Let $\cf^\la$ be given by \eqref{FIOlambda}. 
Then for all $\e>0$ and when $p=p(n)$, there exists a constant $C_{\e}$ that is independent to $\cf$, such that 
\begin{equation}
\label{mixnormineq}
    \|\cFl f\|_{L^p(B_R^{n+1})}\le C_\e R^{(n-1)(\frac12-\frac1p)+\e}\|f\|_2^{\frac2p}\cW(f,B_R^{n+1})^{1-\frac2p}.
\end{equation}
for all $f$ with $\supp\wh f\subset\A^n(1)$ and all $R\in[1, \la^{1-\e}]$.

\end{theorem}

We will perform a two-parameter inductions on $(\la, R)$ with $1\le R\le \la^{1-\e}$. 
Our base case is $R\le 100$, which is not hard to verify.
Suppose Theorem \ref{thmmixednorm} has been proved for $(\la',R')$ satisfying $\la'\le \la, R'\le R/100$.
We now prove for the scale $(\la,R)$.

\smallskip

Consider the wave packet decomposition for $\cFl f$ in $B_R^{n+1}$ at scale $R$:
    \[\cFl f=\sum_{\theta\in\Theta_{R^{-1/2}}}\sum_{T\in\T_\theta}\cFl f_T. \]
Let $w_{B_R^{n+1}}$ be a weight that is $\sim1$ on $B_R^{n+1}$ and decreases rapidly outside $B_R^{n+1}$. 
By a standard dyadic pigeonholing argument (see, for instance, \cite[Section 5]{wang2024restriction}) and homogeneity, we assume that there exists $\T\subset \bigcup_{\theta\in\Theta_{R^{-1/2}}} \T_\theta$ such that 
\begin{equation}
\label{goal-1}
    \|\cFl f\|_{L^p(B_R^{n+1})}\lessapprox \|\sum_{T\in\T}\cFl f_T\|_{L^p(B_R^{n+1})},
\end{equation}
and for all $T\in\T$ and $q_n=\frac{2(n+1)}{n-1}$,
\begin{equation}
\label{normalization}
    \|\cFl f_T\|_{L^{q_n}(w_{B^{n+1}_R})}\sim  R^\frac{n-1}{4}.
\end{equation}
The normalization \eqref{normalization} is set because $|T|^{1/q_n}\approx R^{\frac{n+1}{2}\frac{1}{q_n}}=R^{\frac{n-1}{4}}$.

By dyadic pigeonholing, there is a family of disjoint $K^2$-balls $\cB=\{B\}$ such that $\|\sum_{T\in\T}\cFl f_T\|_{L^p(B)}$ are about the same for all $B\in\cB$ and
\begin{equation}
\nonumber
    \|\sum_{T\in\T}\cFl f_T\|_{L^p(B_R^{n+1})}^p\lessapprox \sum_{B\in\cB}\|\sum_{T\in\T}\cFl f_T\|_{L^p(B)}^p.
\end{equation}
Since $\cFl f_T$ is essentially supported in $T$, for each $B\in\cb$, we have
   \[\|\sum_{T\in\T}\cFl f_T\|_{L^p(B)}\sim \|\sum_{T\in\T, T\cap B\neq\emptyset}\cFl f_T\|_{L^p(B)}.\]
Thus,  $\|\sum_{T\in\T, T\cap B\neq\emptyset}\cFl f_T\|_{L^p(B)}$ are about the same for all $B\in\cB$, and
\begin{equation}
\label{goal-2}
    \|\sum_{T\in\T}\cFl f_T\|_{L^p(B_R^{n+1})}^p\lessapprox\sum_{B\in\cB}\|\sum_{T\in\T, T\cap B\neq\emptyset}\cFl f_T\|_{L^p(B)}^p. 
\end{equation}
For simplicity, we introduce the following notation.
\begin{definition}
Given a triple $(\bB,\bT;\cG)$, we denote
\begin{equation}
\nonumber
    L^p(\bB,\bT;\cG;f):=\sum_{B\in\bB}\|\sum_{T\in \bT(B;\cG)}\cFl f_T\|_{L^p(B)}^p.
\end{equation}
\end{definition}

Let $\cG^{(0)}$ be the trivial shading map: $\cG^{(0)}(T)=\Om$. Therefore, $\T(B;\cG^{(0})=\{T\in\T: T\cap B\neq\emptyset\}$. 
Denote $\T^{(0)}=\T, \cB^{(0)}=\cB$, so 
\begin{equation}
\label{goal-3}
    \text{R.H.S. of }\eqref{goal-2}= L^p(\cB^{(0)},\T^{(0)};\cG^{(0)};f). 
\end{equation}
We also introduce the following notion.

\begin{definition}
For two shading maps $\cG,\cG'$, we write $\cG\subset \cG'$ if the domain of $\cG$ is a subset of the domain of $\cG'$, and $\cG(T)\subset \cG'(T)$ for any $T$ in the domain of $\cG$.
For two triples, we write 
    \[(\bB,\bT;\cG)\subset (\bB',\bT';\cG'),\] 
    if $\bB\subset \bB', \bT\subset \bT',$ and $\cG\subset \cG'$.
\end{definition}

The triple $(\cB^{(0)},\T^{(0)};\cG^{(0)})$ will be the input of the algorithm given in Section \ref{section-algorithm}.
Next, we will repeatedly use Proposition \ref{propalgorithm} to realize the scenario described in Remark \ref{algorithm-remark}.

\bigskip

{\bf Iteration of Proposition \ref{propalgorithm}.}

\bigskip

At the initial stage (step 0), we are given a $\nu^{(0)}>0$  and a triple $(\cB^{(0)},\T^{(0)};\cG^{(0)})$ such that for all $B\in\cB^{(0)}$,
    \[\| \sum_{T\in\T^{(0)}(B;\cG^{(0)})}\cFl f_T \|_{L^p(B)}\sim \nu^{(0)}.\] 
At step $i$, we apply Proposition \ref{propalgorithm} with an integer $A\geq\log K$ to the triple $(\cB^{(i)},\T^{(i)};\cG^{(i)})$ and the number $\nu^{(i)}$, which gives the following restatement of Proposition \ref{propalgorithm} in the new notation:

\begin{List}\label{list1}
\rm

There exists
\begin{enumerate}
    \item parameters $\mu^{(i)}, \lambda^{(i)},\beta^{(i)}, \rho^{(i)}, \si^{(i)}, m^{(i)}, \bar m^{(i)},  \iota^{(i)}, l^{(i)}, \nu^{(i+1)}$,
    
    \item a sequence of refinements $ \cB_4^{(i)}\subset \cB_3^{(i)}\subset \cB_2^{(i)}\subset \cB_1^{(i)}\subset \cB^{(i)}$,
    
    \item  a set $\T_{\beta^{(i)}}^{(i)}\subset \T^{(i)}$, a shading map $\cG_{\lambda^{(i)}}^{(i)}\subset \cG^{(i)}$, and a set of $R^{1/2}$-balls $\cQ^{(i)}$,
\end{enumerate}
such that, by defining 
    \[\cB^{(i+1)}:=\cB_4^{(i)},\,\,\, \T^{(i+1)}:=\T_{\beta^{(i)}}^{(i)},\,\,\, \cG^{(i+1)}:=\cG_{\lambda^{(i)}}^{(i)},\]
which will serve as the input for step $i+1$, the following is true:
\begin{enumerate}[(a)]
    \item $\|\sum_{T\in\T^{(i+1)}(B;\cG^{(i+1)})}\cFl f_T \|_{L^p(B)}\sim\nu^{(i+1)}\gtrapprox\nu^{(i)}$ for all $B\in\cb^{(i+1)}$.
    
    \item $\#\cB^{(i+1)}|_Q\sim \si^{(i)}$, and $\cB^{(i+1)}|_Q$ is $\rho^{(i)}$-regular for all $Q\in\cq^{(i)}$. Here $\cB^{(i+1)}|_Q=\{B\in\cB^{(i+1)}:B\subset Q\}$.
    
    \item $l^{(i)}=\sup_{T\in\T^{(i+1)}}\#\{Q\in \cq^{(i)}: Q\cap T\neq\emptyset \}$.
    
    \item $\bar m^{(i)}\gtrsim m^{(i)}\gtrsim\mu^{(i)}$.
\end{enumerate}
Moreover, we have the following estimates of these parameters:
\begin{enumerate}[(i)]
    \item $\cI(\cB_1^{(i)},\T^{(i)};\cG^{(i)}) \ge \#\cB_1^{(i)}\mu^{(i)}$.

    \item $\cI(\cB_1^{(i)},\T^{(i+1)};\cG^{(i+1)})\lesssim \#\T^{(i+1)}\beta^{(i)}\la^{(i)}$.

    \item $\iota^{(i)}\sim\cI_{\bbr,A}(\cB^{(i+1)}|_Q,\T^{(i)};\cG^{(i)})\gtrsim \si^{(i)}\mu^{(i)}$ for all $Q\in\cQ^{(i)}$. 

    \item $\cI_{\bbr,A}(\cB^{(i+1)}|_Q,\T^{(i+1)};\cG^{(i+1)})\lesssim \rho^{(i)} m^{(i)} $ for all $Q\in\cQ^{(i)}$.

    \item  $\be^{(i)}\leq \Kc$, and if $\beta^{(i)}>100$, then $\#\cB_1^{(i)}\gtrsim (K\Kc)^{-O(1)}l^{(i)}m^{(i)}\la^{(i)}$.

    \item  For all $B\in\bB_1^{(i)}$, $\#_{\bbr,A} \ZT^{(i)}(B;\cG^{(i)})\lesssim\mu^{(i)}$.
    \item For all $T\in\ZT^{(i+1)}$, $\#\cg^{(i+1)}(T)\lesssim\be^{(i)}$.

    \item $\si^{(i)}\lesssim K^{O(1)}R^{\frac{n-2}2}(\bar m^{(i)})^3(\mu^{(i)})^{-3}$.
\end{enumerate}

\end{List}

\medskip

The sets and parameters obtained in the above iteration obey certain monotonicity properties.
We state some of them that will be used later.

\begin{List}\label{list2}

\rm

We have the following monotonicity on parameters:

\begin{enumerate}[(i)]
    \item $\mu^{(i)}\ge\mu^{(i+1)}, \lambda^{(i)}\ge\lambda^{(i+1)},\beta^{(i)}\ge\beta^{(i+1)}, \rho^{(i)}\ge\rho^{(i+1)}, \si^{(i)}\ge\si^{(i+1)}, m^{(i)}\ge m^{(i+1)}, \bar m^{(i)}\ge \bar m^{(i+1)},  \iota^{(i)}\ge \iota^{(i+1)}, l^{(i)}\ge l^{(i+1)}$.
    
    \item $\nu^{(i+1)}\gtrapprox\nu^{(i)}$.
\end{enumerate}

Also, we have the following monotonicity for sets and incidence:

\rm
\begin{enumerate}
    \item $(\cB^{(i)},\T^{(i)};\cG^{(i)})\supset (\cB^{(i+1)},\T^{(i+1)};\cG^{(i+1)})$.

    \item $\cq^{(i)}\supset \cq^{(i+1)}$.

    \item $\cI(\cB^{(i)},\T^{(i)};\cG^{(i)})\ge \cI(\cB^{(i+1)},\T^{(i+1)};\cG^{(i+1)})$.
    
    \item $\cI_{\bbr,A}(\cB^{(i)},\T^{(i)};\cG^{(i)})\ge \cI_{\bbr,A}(\cB^{(i+1)},\T^{(i+1)};\cG^{(i+1)})$.

    \item $\#\cB^{(i)}\lessapprox \#\cB^{(i+1)}\le \#\cB^{(i)}$.

    \item $\bar m^{(i+1)}\lesssim m^{(i)}\lesssim \bar m^{(i)}$.

\end{enumerate}

As a consequence of (ii) and (5), we have
\begin{enumerate}[(a)]
    \item $\|\sum_{T\in \T^{(i)}(B;\cG^{(i)})}\cFl f_T\|_{L^p(B)}\lessapprox \|\sum_{T\in \T^{(i+1)}(B;\cG^{(i+1)})}\cFl f_T\|_{L^p(B)}$ for all $B\in\cB^{(i+1)}$.

    \item $L^p(\cB^{(i)},\T^{(i)};\cG^{(i)};f)\lessapprox L^p(\cB^{(i+1)},\T^{(i+1)};\cG^{(i+1)};f)$.
\end{enumerate}

\end{List}

\bigskip

{\bf At this point, we have completed the description of the iteration.}

\bigskip

Now we are going to a uniformization for the parameters obtained in the iteration.
For this, we introduce a new parameter $\ka=R^{\e^{500}}$.
One may compare it with other parameters: $R^\e\gg K\gg \Kc\gg \kappa\gg1$.

Note that the 12 factors $\mu^{(i)}, \lambda^{(i)}, \beta^{(i)}, \rho^{(i)}, \si^{(i)}, m^{(i)}, \bar m^{(i)},  \iota^{(i)}, l^{(i)}, \#\cb^{(i)}$, $\ZT^{(i)}$, $\cI(\cB_1^{(i)},\T^{(i)};\cG^{(i)})$ are all natural numbers and all $\leq R^{10n}$.
Partition $[1,R^{10n}]^{12}$ into $O_\e(1)$ balls of radius $\ka$.
Since these factors are monotonically decreasing, by pigeonholing, we can find an integer $N\le\e^{-O(1)}\lesssim_\e 1$ so that the following is true:

\begin{List}\label{list3}
\rm
For $i=N-2,N-1,N,N+1,N+2$, we have the reverse control on the parameters:
\begin{enumerate}
    \item $\mu^{(i)}\le\ka\mu^{(i+1)},  \lambda^{(i)}\le\ka\lambda^{(i+1)}, \beta^{(i)}\le\ka \beta^{(i+1)}, \rho^{(i)}\le\ka\rho^{(i+1)}, \si^{(i)}\le\ka\si^{(i+1)}$, $\iota^{(i)}\le\ka \iota^{(i+1)}$, $m^{(i)}\le\ka m^{(i+1)}$, $\bar m^{(i)}\le\ka \bar m^{(i+1)}$, $l^{(i)}\le\ka l^{(i+1)}$, $\#\cb^{(i)}\le \ka \#\cb^{(i+1)}$, $\#\T^{(i)}\le \ka \T^{(i+1)}$, $\cI(\cB_1^{(i)},\T^{(i)};\cG^{(i)})\le \ka \cI(\cB_1^{(i+1)},\T^{(i+1)};\cG^{(i+1)})$.
\end{enumerate}
As a consequence of this and items (i)-(iv) and (viii) in List \ref{list1}, item (6) in List \ref{list2}, when $i=N$, we have
\begin{enumerate}
    \item[(2)] $\#\cB_1^{(i)}\mu^{(i)}\lesssim \ka^2\#\T^{(i)}\beta^{(i)}\la^{(i)}$.
    \item[(3)] For all $Q\in\cQ^{(i)}$, $\si^{(i)}\mu^{(i)}\lesssim \ka\rho^{(i)}m^{(i)}$.
    \item[(4)] $\si^{(i)}\lesssim K^{O(1)}\ka^3R^{\frac{n-2}2}( m^{(i)})^3(\mu^{(i)})^{-3}$.
\end{enumerate}
\end{List}

This, in particular, implies the following two lemmas:
\begin{lemma}
\label{upperboundmu}
Let $N$ be such that List \ref{list3} is true for $i=N$. 
Then when $\be^{(N)}>100$,
\begin{equation}
\nonumber
    \mu^{(N)}\lesssim R^{\e^2}(\#\T^{(N)})^{1/2}(l^{(N)}\si^{(N)})^{-1/2}(\rho^{(N)})^{1/2}.
\end{equation}
\end{lemma}
\begin{proof}
By item (v) in List \ref{list1} and item (2) in List \ref{list3},
    \[(K\Kc)^{-O(1)}l^{(N)}m^{(N)}\la^{(N)}\mu^{(N)}\lesssim \ka^2\#\T^{(N)}\beta^{(N)}\la^{(N)}.\]
This concludes the Lemma by item (3) in List \ref{list3} and since $(\ka K\Kc)^{O(1)}\lesssim R^{\e^2}$.
\qedhere

\end{proof}

\begin{lemma}
\label{upperboundmu2}
Let $N$ be such that List \ref{list3} is true for $i=N$. 
Then when $\be^{(N)}>100$,
\begin{equation}
\nonumber
    \mu^{(N)}\lesssim R^{\e^2}(\#\T^{(N)})^{1/2}(l^{(N)})^{-1/2}\min\{1, \bigg(  \frac{R^{\frac{n-2}{2}}}{\sigma^{(N)}}\bigg)^{\frac{1}{6}}\}.
\end{equation}
\end{lemma}
\begin{proof}
By item (v) in List \ref{list1} and item (2) in List \ref{list3},
    \[(K\Kc)^{-O(1)}l^{(N)}m^{(N)}\la^{(N)}\mu^{(N)}\lesssim \ka^2\#\T^{(N)}\beta^{(N)}\la^{(N)}.\]
Since $(\ka K\Kc)^{O(1)}\lesssim R^{\e^2}$, this concludes the Lemma by item (4) in List \ref{list3} and item (d) in List \ref{list1}.
\end{proof}

\medskip

Let us summarize what we have into a theorem.
\begin{theorem}
\label{after-iteration-thm}
Let $A\geq1$ be an integer.
If $A\geq \log K$, then there exist 
\begin{enumerate}
    \item parameters $\mu, \beta, \rho, \si, l$,
    \item a triple $(\bB, \bT, \cG)$,
    \item a collection of $R^{1/2}$-balls $\cq$,
\end{enumerate}
such that the following is true:
\begin{enumerate}[(a)]
    \item $\cup_{\bB}\subset \cup_\cq$.
    \item $\#\bB|_Q\sim \si$, and $\bB|_Q$ is $\rho$-regular for all $Q\in\cq$. Here $\bB|_Q=\{B\in\bB:B\subset Q\}$.
    \item $l=\sup_{T\in\bT}\#\{Q\in \cQ: Q\cap T\neq\emptyset \}$.
\end{enumerate}
Moreover, we have the following estimates:
\begin{enumerate}[(i)]
    \item $\#_{\bbr,A} \bT(B;\cG)\lesssim \mu$ for all $B\in\bB$.

    \item $L^p(\cB^{(0)},\T^{(0)};\cG^{(0)};f)\lessapprox L^p(\bB,\bT;\cG;f)$ (recall \eqref{goal-3}).

    \item $\cg(T)\lesssim\be$ for all $T\in\bT$, and if $\be\geq100$, we have
    $ \mu\lesssim R^{\e^2}(\#\T)^{1/2}(l\si)^{-1/2}\rho^{1/2}$, and
    $ \mu\lesssim R^{\e^2}(\#\T)^{1/2}l^{-1/2}\min\{1, (R^{(n-2)/2}\si^{-1})^{1/6}\}$.
\end{enumerate}

\end{theorem}

\begin{proof}
Let $N\lesssim_\e 1$ be the natural number obtained in List \eqref{list3}.
Take $(\bB,\bT,\cg)=(\cB^{(N+1)},\T^{(N+1)};\cG^{(N+1)})$ and $(\mu, \beta, \rho, \si, l, \cq)=(\mu^{(N)}, \beta^{(N)}, \rho^{(N)}, \si^{(N)}, l^{(N)},\cq^{(N)})$.
Then item (a) follows by definition, and items (b) and (c) follow from items (b) and (c) in List \ref{list1}.
Items (i), (iii) follow from items (vi), (vii) in List \ref{list1} and Lemmas \ref{upperboundmu}, \ref{upperboundmu2}.
Finally, item (ii) follows from item (b) in List \ref{list2}, and the fact that $N\lesssim_\e1$.
\qedhere

\end{proof}

\bigskip

\section{Reduction to a two-ends broad estimate}

We prove Theorem~\ref{thmmixednorm} in this and the following sections. As noted in the previous section, our proof proceeds by induction, assuming that \eqref{mixnormineq} holds for all $r \le R/2$.  In this section, we reduce the problem to a two-end broad estimate using a two-end reduction argument, combined with a standard broad-narrow analysis.

Let $A=\log K$.
Let the triple $(\bB,\bT,\cg)$, the set of $R^{1/2}$-balls, and the parameters $(\mu, \beta, \rho, \si, l)$ be given by Theorem \ref{after-iteration-thm}.
As a consequence of item (ii) in Theorem \ref{after-iteration-thm} and recall  \eqref{goal-3},  \eqref{goal-2}, and \eqref{goal-1}, we have 
\begin{equation}
\label{goal-final}
    \|\cFl f\|_{L^p(B_R^{n+1})}^p\lessapprox\sum_{B\in\bB}\|\sum_{T\in \bT(B;\cG)}\cFl f_T\|_{L^p(B)}^p.
\end{equation}

We partition each horizontal region $\omega \cap B_R^{n+1}$ ($\omega \in \Omega$) into a collection of $R/\Kc$-balls, resulting in a partition of $B_R^{n+1}$ into $R/\Kc$-balls, which we denote by $\{B_k\}$.
Using this partition, we have
\begin{equation}
\nonumber
    \text{R.H.S. of }\eqref{goal-final}=\sum_k\sum_{B\in\bB, B\subset B_k}\|\sum_{T\in\bT(B;\cG)}\cFl f_T\|_{L^p(B)}^p.
\end{equation}
Denote
\begin{equation}
\label{bT_k}
    \bT_k=\{ T\in\bT: \cG(T)\cap B_k\neq\emptyset \}, 
\end{equation}
and let 
\begin{equation}
\label{f_k}
    f_k:= \sum_{T\in\bT_k}f_T.
\end{equation}
Since $\cFl f_T$ is essentially supported in $T$, when $B\subset B_k$, we have
    \[\|\sum_{T\in\bT(B;\cG)}\cFl f_T\|_{L^p(B)} \lesssim\|\cFl f_k\|_{L^p(B)}. \]
Therefore, we obtain 
\begin{equation}
\label{B_k}
    \text{R.H.S. of }\eqref{goal-final}\lesssim \sum_k\sum_{B\in\bB, B\subset B_k}\|\cFl f_k\|_{L^p(B)}^p\lesssim\sum_k \|\cFl f_k\|_{L^p(B_k)}^p.
\end{equation}

\smallskip

Denote by $\beta=\beta^{(N)}$.
We will discuss the following two cases separately: 
\begin{enumerate}
    \item \textbf{One-end}: $\beta\le 100$.
    \item \textbf{Two-ends}: $\beta > 100$. 
\end{enumerate}
In the rest of the section, we prove Theorem \ref{thmmixednorm}.
The numerology of the two cases, $n = 3$ and $n \geq 4$, differs slightly but follows the same strategy. 
Therefore, we will address them simultaneously.

\bigskip

\subsection{One-end case} In this subsection, we prove Theorem \ref{thmmixednorm} assuming $\beta\le 100$. 
Recall Theorem \ref{after-iteration-thm} item (iii) that $\cg(T)\lesssim\be\le 100$ for all $T\in\bT$. 
Since $\beta\le 100$, each wave packet $f_T$ is included in the summation of at most $O(1)$ different $f_k$.
Thus, by $L^2$-orthogonality,
\begin{equation}
\label{oneL2}
    \sum_k\|f_k\|_2^2\lesssim \|f\|_2^2. 
\end{equation}

Next, we use the induction hypothesis at the scales $(\lambda, R/\Kc)$ and apply \eqref{mixnormineq} to each $B_k$.
We note that \eqref{mixnormineq} is stated for balls centered at the origin, but a change of coordinates allows the result to hold for balls not necessarily centered at the origin.
Suppose $B_k=B^{n+1}_{R/\Kc}(z_k)$.
Apply \eqref{mixnormineq} to obtain
\begin{equation}
\nonumber
    \|\cFl f_k\|_{L^p(B_k)}^p\le C_\e (\frac{R}{\Kc})^{(n-1)(\frac{p}{2}-1)+p\e}\|f_k\|_2^2\,\cW(f_k, B^{n+1}_{R/\Kc}(z_k))^{p-2}.   
\end{equation}
By Lemma \ref{leminduct} in which $r=R\Kc^{-1}$, we have
\begin{equation}
\nonumber
    \cW( f_k,B^{n+1}_{R/\Kc}(z_k))\lesssim R^{O(\de)}\Kc^{\frac{n-1}{2}} \cW(f_k,B^{n+1}_R).
\end{equation}
Since the wave packets summed in $f_k$ are a subset of those summed in $f$, it follows from Lemma \ref{wpd-less-lem} that
    \[\cW(f_k,B^{n+1}_R)\lesssim \cW(f,B^{n+1}_R). \]

Combining the calculations above and recalling \eqref{B_k}, we obtain
\begin{align*}
    \|\cFl f\|_{L^p(B_R^{n+1})}^p&\lessapprox 
    \sum_{k}C_\e (\frac{R}{\Kc})^{(n-1)(\frac{p}{2}-1)+p\e}R^{O(\de)}\Kc^{\frac{n-1}{2}(p-2)} \|f_k\|_2^2\,\cW(f,B^{n+1}_R)^{p-2}\\
    &\lesssim C_\e R^{O(\de)}\Kc^{-p\e}  R^{(n-1)(\frac{p}{2}-1)+p\e}\|f\|_2^2\,\cW(f,B^{n+1}_R)^{p-2}.
\end{align*}
The second inequality follows from \eqref{oneL2}. 
This closes the induction since $\Kc=R^{\e^{100}}$ is large enough and hence proves Theorem \ref{thmmixednorm}.

For the following, we assume the two-ends case, where $\beta > 100$.

\bigskip
\subsection{Two-ends: broad narrow reduction}
By pigeonholing on the $R/\Kc$ balls summed in \eqref{B_k}, there exists a $B_k$ so that
\begin{equation}
\label{two1}
    \|\cFl f\|_{L^p(B_R^{n+1})}^p\lessapprox \Kc^{O(1)}\sum_{B\in\bB, B\subset B_k}\|\cFl f_k\|_{L^p(B)}^p.
\end{equation}

What follows is a somewhat standard broad-narrow method. 
Recall the broad norm $\|\cdot \|_{\BL^p_{k,A}}^p$ defined in Subsection \ref{sectionbroadnorm}.
Let
\begin{equation}
\label{k-n}
    k=k_n:=\left\{
    \begin{array}{ll} 
    \frac{n+5}2, &\textrm{$n\ge 3$ is odd},\\[1ex]
    \frac{n+4}{2},& \textrm{$n\ge 2$ is even}.
\end{array}
\right.
\end{equation}
We remark that the $k$ in the broad norm is irrelevant to the $k$ in $B_k$.

Choose $A=(10n\log K)^2$. 
For any $K^2$-ball $B\subset B_k$, let $V_1,\dots,V_A$ be the set of subspaces of dimension $k_n$ where the broad norm is attained:
    \[\|\cFl f_k\|_{\BL^p_{k_n,A}(B)}^p=\max_{\tau\notin V_i, 1\le i\le A}\int_B |\cFl f_{k,\tau}|^p\]
Here ``$\tau\notin V_i, 1\le i\le A$" means ``$\tau$  ranges over $\tau\in\Theta_{K^{-1}}$ with $\ang(G(\tau),V_i)>K^{-2}$ for all $1\le i\le A$".
Also, $f_{k,\tau}=\sum_{T\in\bT_k[\tau]}f_T$ (recall Definition \ref{T[S]}).
Thus, we have
\begin{equation}
\label{broadnarrow}
    \|\cFl f_k\|_{L^p(B)}^p\lesssim K^{O(1)} \|\cFl f_k\|^p_{\BL^p_{k_n,A}(B)}+\sum_{i=1}^A\int_B | \sum_{\tau\in V_i}\cFl f_{k,\tau} |^p.
\end{equation}
Here ``$\tau\in V_i$" means $\ang(G^\la(z_B,\tau),V_i)\le K^{-2}$, where $z_B$ is the center of $B$.

\smallskip

By using Theorem \ref{l2-decoupling} at $R=K^2, p=p(n)$, Lemma \ref{numberofcaps-lem}, and H\"older's inequality, we can bound the right-hand side of \eqref{broadnarrow} by (the rapidly decreasing term can be discarded, as we have done the normalization in \eqref{normalization})
    \[\sum_{i=1}^A\int_B | \sum_{\tau\in V_i}\cFl f_{k,\tau} |^p\lessapprox K^{(k_n-3)(\frac12-\frac1p)p}\sum_{\tau\in\Theta_{K^{-1}}}\int_{w_B} |\cFl f_{k,\tau}|^p,\]
where $w_B$ is a weight that is $\sim1$ on $B$ and decreases rapidly outside $B$.
Denote $X:=\bigcup_{B\in \bB: B\subset B_k} B$.
Summing up all $\{B\in\bB: B\subset B_k\}$ in \eqref{broadnarrow} recalling \eqref{two1}, we have
\begin{align}\label{broadnarrow2}
    \|\cFl f\|^p_{L^p(B_R^{n+1})}\lessapprox\, & (K\Kc)^{O(1)}\|\cFl f_k\|^p_{\BL^p_{k_n,A}(X)}\\ \nonumber
    &+ \Kc^{O(1)} K^{(k_n-3)(\frac12-\frac1p)p }\sum_{\tau}\int_{B_R^{n+1}}|\cFl f_{k,\tau}|^p.
\end{align}

If the first term on the right-hand side of \eqref{broadnarrow2} dominates, we say we are in the \textbf{broad case}.
If the second term on the right-hand side of \eqref{broadnarrow2} dominates, we say we are in the \textbf{narrow case}.
We will prove Theorem \ref{thmmixednorm} for the narrow case below, and leave the discussion about the broad case in the next section.

\bigskip

\subsection{Two-ends: narrow case}
Recall \eqref{broadnarrow2}.
In this case, we have
    \[\|\cFl f\|_{L^p(B_R^{n+1})}^p\lessapprox \Kc^{O(1)} K^{(k_n-3)(\frac12-\frac1p)p}\sum_\tau \int_{B_R^{n+1}}|\cFl f_{k,\tau}|^p. \]
For each $\tau$, by Lemma \ref{lemrescaling} and by using the induction hypothesis of Theorem \ref{thmmixednorm} at scale $R/K^2$, we have
    \[\|\cFl f_{k,\tau}\|_{L^p(B_R^{n+1})}\lesssim K^{\frac{2}{p}}(\frac{R}{K^2})^{(n-1)(\frac12-\frac1p)+\e}\|f_{k,\tau}\|_2^{\frac{2}{p}}\cW(f_{k,\tau},B_R^{n+1})^{1-\frac2p}. \]
By Lemma \ref{wpd-less-lem}, we get
    \[\cW(f_{k,\tau},B_R^{n+1})\lesssim \cW(f,B_R^{n+1}). \]
Thus, $\|\cFl f\|_{L^p(B_R^{n+1})}^p$ is bounded above by 
\begin{align}
\nonumber
     \Kc^{O(1)} R^{(n-1)(\frac12-\frac1p)p+p\e} &K^{(k_n-3)(\frac12-\frac1p)p+2-2(n-1)(\frac12-\frac1p)p-2p\e}\\ \nonumber
     &\cdot \sum_{\tau}\|f_{k,\tau}\|_2^2\cW(f,B_R^{n+1})^{p-2}.
\end{align}
Note that $p(n)\geq 2+ \f4{2n-k_n+1}$, and when $p\geq 2+ \f4{2n-k_n+1}$ and $k_n\ge 3$, we have $(k_n-3)(\frac12-\frac1p)p+2-2(n-1)(\frac12-\frac1p)p\leq 0$.
Therefore, when $p=p(n)$,
\begin{equation}
\nonumber
    \|\cFl f\|_{L^p(B_R^{n+1})}^p\lessapprox \Kc^{O(1)}  R^{(n-1)(\frac12-\frac1p)p+p\e} K^{-2p\e}\|f\|_2^2\,\cW(f,B_R^{n+1})^{p-2}.
\end{equation}
Since $K=R^{\e^{50}}$ and since $\Kc=R^{\e^{100}}$, we have $K^{-2p\e}\Kc^{O(1)}\lesssim R^{-\e^{100}}$.
Thus, the induction is closed, and we finish the proof of Theorem \ref{thmmixednorm} in the narrow case.

\bigskip

\section{Finish the proof: broad case study}
Recall \eqref{broadnarrow2} and $A=(10n\log K)^2$.
In the broad case, we have
\begin{equation}
\label{twoends1}
    \|\cFl f\|^p_{L^p(B_R^{n+1})}\lessapprox (K\Kc)^{O(1)}\|\cFl f_k\|^p_{\BL^p_{k_n,A}(X)}.
\end{equation}
Recall \eqref{f_k}, \eqref{bT_k}, \eqref{goal-final}, and Theorem \ref{after-iteration-thm}.

\bigskip

\subsection{\texorpdfstring{$L^{q_n}$}{} estimate}
Let $q_n=\frac{2(n+1)}{n-1}$ be the decoupling exponent.
Note that $X=(\cup_{\bB})\cap B_k$. 
Since $A/2\geq \log K$, for each $K^2$-ball $B\subset X$, by item (i) in Theorem \ref{after-iteration-thm}, we have
    \[\#_{\bbr,A} \{T\in\bT_k: T\cap B\neq\emptyset\} \le \#_{\bbr,A} \bT(B;\cG)\lesssim \mu. \]
Therefore, we can find $\tau'\in\Theta_{K^{-1}}$ such that
    \[\#\{ T\in\bT_k: T\cap B\neq\emptyset,\, \theta(T)\not\subset 3\tau' \}\lesssim \mu. \]
By the definition of broad norm and noting $A/2\geq10n$, we have
    \[\|\cFl f_k\|^{q_n}_{\BL^{q_n}_{k_n,A/2}(B)}\lesssim K^{O(1)}\max_{\tau\not\subset 3\tau'}\int_B |\cFl f_{k,\tau}|^{q_n}, \]
where, recall that $f_{k,\tau}=\sum_{T\in\bT_k[\tau]}f_T$.
Hence, there is a $\tau(B)\not\subset 3\tau'$ such that
\begin{equation}
\label{pigeontau}
    \|\cFl f_k\|^{q_n}_{\BL^{q_n}_{k_n,A/2}(B)}\lesssim K^{O(1)}\int_B |\cFl f_{k,\tau(B)}|^{q_n},
\end{equation}
and since $\tau(B)\not\subset 3\tau'$, we have
\begin{equation}\label{uppermu}
    \#\{\bT_k: T\cap B\neq\emptyset,\, \theta(T)\subset \tau(B)\}\lesssim \mu.
\end{equation}

By pigeonholing on the set $\{\tau(B)\}_{B\subset X}$, we can find a uniform $\tau\in\Theta_{K^{-1}}$, a union of sub-collection of $K^2$-balls $X'\subset X$, such that $\tau(B)=\tau$ for $B\subset X'$ and
    \[\|\cFl f_k\|_{\BL^{q_n}_{k_n,A/2}(X)}^{q_n}\lesssim K^{O(1)} \|\cFl f_k\|_{\BL^{q_n}_{k_n,A/2}(X')}^{q_n}.\]
Hence, by \eqref{pigeontau}, we get
\begin{equation}
\nonumber
    \|\cFl f_k\|_{\BL^{q_n}_{k_n,A/2}(X)}^{q_n}\lessapprox K^{O(1)}\sum_{B\subset X'}\int_B|\cFl f_{k,\tau}|^{q_n}=K^{O(1)}\| \cFl f_{k,\tau}\|^{q_n}_{L^{q_n}(X')}.
\end{equation}
Recall \eqref{normalization} and Lemma \ref{local-L2-lem-2}.
Since $\cFl f_T$ is essentially constant on $T$, we have 
    \[R^\frac{n+1}{2}\sim\|\cFl f_T\|_{L^{q_n}(w_{B^{n+1}_R})}^{q_n}\lesssim R^\frac{n+1}{n-1}\|\cFl f_T\|_{L^{2}(w_{B^{n+1}_R})}^{q_n}\lesssim \|f_T\|_2^{q_n}.\]
Also, by Lemma \ref{wpd-lem}, we have
    \[\cW(f,B_R^{n+1})^2\gtrsim \max\{ 1, R^{-\frac{n+1}{2}}(\#\bT)\}\gtrsim R^{-\frac{n+1}{4}}(\#\bT)^{1/2}.\]
Therefore, by \eqref{uppermu} and recalling \eqref{normalization}, we can apply Theorem \ref{refdecthm} to get
\begin{align}
\label{Lqnest}
    &\|\cFl f_k\|_{\BL^{q_n}_{k_n,A/2}(X)}^{q_n}\lessapprox K^{O(1)}\|\cFl f_{k,\tau}\|_{L^{q_n}(X')}^{q_n}\\[1ex]
    \nonumber
    \lessapprox &\, K^{O(1)}\mu^{\frac{2}{n-1}}\sum_{T\in\bT_k[\tau]}\|\cFl f_T\|_{L^{q_n}(w_{B^{n+1}_R})}^{q_n}\lesssim K^{O(1)}\mu^{\frac{2}{n-1}}(\#\bT\cdot R^{\frac{n+1}{2}})\\
    \nonumber
    \lessapprox &\,K^{O(1)} R^2\cdot(R^{-\frac{n-3}{2(n-1)}}\mu^\frac{2}{n-1} (\#\bT)^{-\frac{1}{n-1}})\|f\|_2^2\,\cW(f,B_R^{n+1})^{p_n-2}. 
\end{align}
This is the $L^{q_n}$ estimate we need.

\bigskip

\subsection{Proof for the case \texorpdfstring{$n=3$}{}}

Write $\frac{1}{p(n)}=\al(n)\frac{1}{2}+(1-\al(n))\frac{1}{q_n}$.
Notice that $q_3=4$ and $\al(3)=1/5$.
By Lemma \ref{lemholder}, we have
\begin{equation}
\label{twoends2}
    \|\cFl f_k\|_{\BL^p_{4,A}(X)}\lesssim \|\cFl f_k\|^{1/5}_{\BL^2_{4,A/2}(X)} \|\cFl f_k\|^{4/5}_{\BL^{4}_{4,A/2}(X)}.
\end{equation}
It remains to estimate $\|\cFl f_k\|_{\BL^2_{4,A/2}(X)}$ and $\|\cFl f_k\|_{\BL^{4}_{4,A/2}(X)}$.

We are going to estimate $\|\cFl f_k\|_{\BL^2_{4,A/2}(X)}$ inside each 
$R^{1/2}$-ball $Q\in\cQ$ separately and then sum them up.
Recall that $X=\cup_{\bB}\cap B_k$.
By Theorem \ref{after-iteration-thm}, $X|_Q$ is $\rho$-regular for all $Q\in\cq$.
Thus, for each $Q\in\cQ$, there is a set of $10K^2\times R^{1/2}\times R^{1/2}\times R^{1/2}$-slabs $\S(Q)$ with  $\#\S(Q)\lesssim \si/\rho$ such that $X\cap Q\subset \cup_{\S(Q)}$. 
Hence,
\begin{equation}\label{proceedas}
    \|\cFl f_k\|^2_{\BL^2_{4,A/2}(X\cap Q)}\le \sum_{S\in\S(Q)}\|\cFl f_k\|^2_{\BL^2_{4,A/2}(S)}. 
\end{equation} 
Fix an $S\in\ZS(Q)$ and let $V$ be the 3-dimensional subspace that is parallel to $S$. By the definition of the broad norm, we have
    \[\|\cFl f_k\|^2_{\BL^2_{4,A/2}(S)}\le \sum_{B\subset S}\max_{\tau\notin V}\int_B |\cFl f_{k,\tau}|^2.  \]
By pigeonholing on the caps $\{\tau\}$, there exists a uniform $\tau\notin V$ such that
    \[\|\cFl f_k\|^2_{\BL^2_{4,A/2}(S)}\lesssim K^{O(1)} \int_S |\cFl f_{k,\tau}|^2. \]
Here $f_{k,\tau}=\sum_{T\in\bT_k[\tau]}f_T$.

% and note that     
%     \[\int_S |\cFl f_{k,\tau}|^2\lesssim\int_S\big|\sum_{T\in\bT_k[\tau],T\cap Q\not=\varnothing}f_T\big|^2\]
Apply Lemma \ref{local-L2-lem-1} and note that $\bT_k[\tau]\subset\bT_k\subset \bT$. We get
    \[\int_S |\cFl f_{k,\tau}|^2\lesssim K^{O(1)} \sum_{T\in\bT_k[\tau], T\cap Q\not=\varnothing} \|f_T\|_2^2\lesssim K^{O(1)} \sum_{T\in\bT, T\cap Q\not=\varnothing} \|f_T\|_2^2. \]
Plugging it back to \eqref{proceedas} and summing up all $Q\in\cq$, we get
\begin{equation}
\nonumber
\begin{split}
     \|\cFl f_k\|^2_{\BL^2_{4,A/2}(X)}&\lesssim K^{O(1)}\sum_{Q\in\cQ} \sum_{S\in\S(Q)}\sum_{T\in\bT: T\cap Q\neq\emptyset} \|f_T\|_2^2\\
     &\lesssim K^{O(1)}\#\S(Q) \sum_{T\in\bT}\sum_{Q\in\cQ:T\cap Q\neq\emptyset} \|f_T\|_2^2\lesssim K^{O(1)} l(\si/\rho)\|f\|_2^2.
\end{split}
\end{equation}
In the last inequality, we use item (c) in Theorem \ref{after-iteration-thm} and $L^2$-orthogonality.

\medskip

Now we can finish the proof.
By \eqref{twoends1}, \eqref{twoends2}, \eqref{Lqnest} and \eqref{proceedas}, we get
    \[\|\cFl f\|_{L^p(B_R^{4})}\lessapprox (K\Kc)^{O(1)}R^{2(\frac{1}{2}-\frac{1}{p})}(\mu (\#\bT)^{-1/2})^{1/5}(l\si/\rho)^{1/10}\|f\|_2^\frac{2}{p}W(f,B_R^4)^{1-\frac{2}{p}}. \]
This implies \eqref{mixnormineq} since $ (K\Kc)^{O(1)}\lesssim R^{\e^2}$ and since $\mu (\#\bT)^{-1/2}(l\si/\rho)^{1/2}\lesssim R^{\e^2}$, which follows from item (iii) in Theorem \ref{after-iteration-thm}.

\bigskip

\subsection{Proof for the case \texorpdfstring{$n\geq 5$ for odd $n$}{}}
\label{odd-n-section}

In this case, we need to use the $k$-broad estimate obtained by \cite{Schippa}.
Let 
\[r_{n} = \frac{2(n+k_n+1)}{n+k_n-1}.\]

\begin{theorem}[\cite{Schippa}, Theorem 1.2]
\label{k-brd}
For any $\e>0$, there exists constants $C_\e$ and $A\geq[1,\infty)\cap\ZZ$ such that for all $f$ with $\supp\wh f\subset\A^n(1)$,
    \[\|\cFl f\|_{\BL^{r_n}_{k_n,A}(B_R^{n+1})}\le C_\e R^\e \|f\|_{2}.\]
\end{theorem}

As two direct corollaries, we have
\begin{corollary}
For any $\e>0$, there exists constants $C_\e$ and $A\geq[1,\infty)\cap\ZZ$ such that for all $f$ with $\supp\wh f\subset\A^n(1)$,
\begin{equation}
\label{k-brd1}
    \|\mathcal F^\la f\|_{\BL^{r_n}_{k_n,A}(B_R^{n+1})}\le C_\e R^\e R^{n(\f12-\frac{1}{r_n})} \|f\|_{2}^{\frac{2}{r_n}}\cW(f,B^{n+1}_R)^{1-\frac{2}{r_n}}.
\end{equation}
\end{corollary}

\begin{proof}
Recall $\cW(f,B^{n+1}_R)$ in Definition \ref{wpd-def}.
Thus, we have
\[
     \|f\|_{2}\lesssim \Big( \sum_{T\in\ZT}\|f_T\|_{2}^2\Big)^{1/2} \lesssim R^{n/2} \cW(f,B^{n+1}_R). 
\]
This proves the corollary by Theorem \ref{k-brd}. \qedhere

\end{proof}

\begin{corollary}
\label{cor-2}
Let $\ZT$ be a family of $R$-planks, and let $f=\sum_{T\in\ZT} f_T$ be a sum of wave packets. 
Let $\si, l\geq1$.
Let $X$ be a union of $K^2$-balls, and let $\cq$ be a family of $R^{1/2}$-balls such that the following hold:
\begin{enumerate}
    \item $X\subset\cup_\cq$, and $|X\cap Q|\sim \si$ for all $Q\in\cq$.
    \item $\sup_{T\in\ZT}\#\{Q\in \cQ: Q\cap T\neq\emptyset \}\leq l$.
\end{enumerate}
Then we have
\begin{equation}
\label{L2e}
    \|\cFl f\|_{\BL^2_{k_n,A}(X)}\lessapprox l^{1/2}\si^{\f 1{n+k_n+1}}\|f\|_{L^2}.
\end{equation}

\end{corollary}
\begin{proof}
Consider each $Q\in\cq$.
Let $\ZT(Q)=\{T\in\ZT:T\cap 2Q\not=\varnothing\}$ and $f_Q=\sum_{T\in\ZT_{Q}}f_T$.
By H\"older's inequality and Theorem \ref{k-brd}, we have
\begin{align}
\nonumber
    \|\cFl f\|^2_{\BL^2_{k_n,A}(X\cap Q)}&\lesssim \|\cFl f_Q\|^2_{\BL^2_{k_n,A}(X\cap Q)}\\ \nonumber
    &\lesssim \si^{1-\frac{2}{r_n}}\|\cFl f_Q\|^2_{\BL^{r_n}_{k_n,A/2}(Q)}\lessapprox \si^{1-\frac{2}{r_n}}\|f_Q\|_2^2.
\end{align} 
Since $1-2/r_n=2/(n+k_n+1)$, we can sum up all $Q\in\cq$ using assumption (2) to conclude the corollary.
\qedhere

\end{proof}

\smallskip

Recall \eqref{twoends1}.
Let $\al(n)=1-\frac{(n-3)(n+k_n+1)}{(3n+1)(k_n-1)}$ and $A'=A/3$.
By Lemma \ref{lemholder}, we have for $p=p(n)=p_n^+$,
\begin{equation}
\label{holder-3-term}
    \|\cFl f_k\|_{\BL^p_{k_n,A}(X)}\lesssim \|\cFl f_k\|_{\BL^{q_n}_{k_n,A'}(X)}^\frac{\al(n)(n+1)}{n+2}\|\cFl f_k\|_{\BL^{2}_{k_n,A'}(X)}^\frac{\al(n)}{n+2}\|\cFl f_k\|_{\BL^{r_n}_{k_n,A'}(X)}^{1-\al(n)}.
\end{equation}
Note that the datum $(f_k,X)$ obeys the assumption of Corollary \ref{cor-2}.
Let $p_n=\frac{2(n+2)}{n}$.
Thus, by \eqref{Lqnest}, \eqref{L2e}, and the second incidence estimate in (iii) of Theorem \ref{after-iteration-thm}, we have
\begin{align}
\nonumber
    &\|\cFl f_k\|_{\BL^{q_n}_{k_n,A'}(X)}^\frac{n+1}{n+2}\|\cFl f_k\|_{\BL^{2}_{k_n,A'}(X)}^\frac{1}{n+2}  \\ \nonumber
    \lesssim&\, R^{O(\e^2)}R^{\frac{3n-1}{4(n+2)}}\min\{1, R^{\frac{n-2}{2}}\si^{-1}\}^\frac{1}{6
    (n+2)}\cdot\si^{\f 1{(n+k_n+1)(n+2)}} \|f\|_{L^2}^{\f2{p_n}}\,\cW(f,B_R^{n+1})^{1-\f2{p_n}}\\ \nonumber
    \lesssim &\, R^{O(\e^2)}\min\{R^{\frac{3(3n-1)}{2}}\si^\frac{6}{n+k_n+1}, R^{\frac{10n-5}{2}}\si^{-\frac{n+k_n-5}{n+k_n+1}}\}^\frac{1}{6
    (n+2)} \|f\|_{L^2}^{\f2{p_n}}\,\cW(f,B_R^{n+1})^{1-\f2{p_n}}.
\end{align}
Plugging this back to \eqref{holder-3-term} and employing \eqref{k-brd1}, we get that when $p=p(n)$,
\begin{align}
\nonumber
    \|\cFl f_k\|_{\BL^p_{k_n,A}(X)}\lesssim& R^{O(\e^2)}\|f\|_{L^2}^{\f2{p}}\,\cW(f,B_R^{n+1})^{1-\f2{p}}\cdot \\ \label{numerology1}
    & \min\{R^{\frac{3(3n-1)}{2}}\si^\frac{6}{n+k_n+1}, R^{\frac{10n-5}{2}}\si^{-\frac{n+k_n-5}{n+k_n+1}}\}^\frac{\al(n)}{6
    (n+2)}\cdot R^{\frac{n(1-\al(n))}{n+k_n+1}}.
\end{align}
Recall \eqref{k-n} that $k_n=\frac{n+5}{2}$ when $n$ is odd. 
Thus, when $p=p(n)$ and $n\geq 5$,
\begin{align}
\nonumber
    \eqref{numerology1}&=\min\{R^\frac{9n-3}{2}\si^\frac{12}{3n+7},R^\frac{10n-5}{2}\si^{-\frac{3n-5}{3n+7}}\}^\frac{2}{3n^2+10n+3}\cdot R^\frac{2n^2-6n}{3n^2+10n+3}\\ \nonumber
    &=R^{(n-1)(\frac{1}{2}-\frac{1}{p})}\min\{R^\frac{9n-3}{2}\si^\frac{12}{3n+7},R^\frac{10n-5}{2}\si^{-\frac{3n-5}{3n+7}}\}^\frac{2}{3n^2+10n+3}\cdot R^\frac{6-10n}{3n^2+10n+3}\\ \nonumber
    &=R^{(n-1)(\frac{1}{2}-\frac{1}{p})}\min\{R^{\frac{3-n}{2}}\si^{\frac{12}{3n+7}},R^{\frac{1}{2}}\si^{-\frac{3n-5}{3n+7}}\}^\frac{2}{3n^2+10n+3}\leq R^{(n-1)(\frac{1}{2}-\frac{1}{p})}.
\end{align}
Thus, by \eqref{twoends1}, this proves Theorem \ref{thmmixednorm} when $n\geq5$ and $n$ is odd.

\bigskip

\subsection{Proof for the case \texorpdfstring{$n\geq 4$ for even $n$}{}}

The idea behind this proof is similar to that in Section \ref{odd-n-section}, so we only provide a few remarks regarding the choice of parameters.

Similar to \eqref{holder-3-term}, choosing $\al(n)=1-\frac{(3n-10)(n+1)}{(3n^2+n+6)}$ and $A'=A/3$, we have by Lemma \ref{lemholder} that when $p=p(n)$, 
\begin{equation}
\nonumber
    \|\cFl f_k\|_{\BL^p_{k_n,A}(X)}\lesssim \|\cFl f_k\|_{\BL^{q_n}_{k_n,A'}(X)}^\frac{\al(n)(n+1)}{n+2}\|\cFl f_k\|_{\BL^{2}_{k_n,A'}(X)}^\frac{\al(n)}{n+2}\|\cFl f_k\|_{\BL^{r_n}_{k_n,A'}(X)}^{1-\al(n)}.
\end{equation}
The remainder of the calculation proceeds as before, using the choice $k_n=\frac{n+4}{2}$.

\bigskip

\appendix 
\numberwithin{equation}{section}
\section{Decoupling inequalities}

Suppose $f(x)$ is a function in $\R^{n }$ with $\supp\wh f\subset \A^n(1)$ with the wave packet decomposition inside $B_R$:
    \[\cFl f=\sum_{\theta\in\Theta_{R^{-1/2}}}\sum_{T\in\T_\theta}\cFl f_T. \]
Beltran, Hickman, and Sogge \cite{beltran2020variable} proved the following variable coefficient version of $\ell^pL^p$-decoupling. 
We remark that $\cFl f=T^\la\wh f$, where $T^\la$ is the oscillatory integral operator as in \cite{beltran2020variable}. 
Also, as we will focus exclusively on $(n+1)$ dimensions in the appendix, we simplify the notation by writing $B_r^{n+1}(z)$ as $B_r(z)$.

\begin{theorem}[$\ell^p$-decoupling]
\label{lp-decoupling}
    Let $2\le p\le \frac{2(n+1)}{n-1}$. For any $\e>0$, one has
    \begin{equation}
    \nonumber
        \|\cFl f\|_{L^p(B_R )} \lesssim_{\e,M,\phi,a} R^{\frac{n-1}{2}(\frac12-\frac1p)+\e} (\sum_{\theta\in\Theta_{R^{-1/2}}}\|\cFl f_\theta\|_{L^p(w_{B_R })}^p)^{1/p}+R^{-M}\|f\|_2. 
    \end{equation}
Here $w_{B_R }$ is a weight that is $\sim1$ on $B_R $ and decreases rapidly outside $B_R $. 
\end{theorem}

A slight modification of the proof in \cite{beltran2020variable} also gives the $\ell^2$-decoupling inequality.

\begin{theorem}[$\ell^2$-decoupling]
\label{l2-decoupling}
    Let $2\le p\le \frac{2(n+1)}{n-1}$. For any $\e>0$, one has
    \begin{equation}
    \nonumber
        \|\cFl f\|_{L^p(B_R )} \lesssim_{\e,M,\phi,a} R^{\e} (\sum_{\theta\in\Theta_{R^{-1/2}}}\|\cFl f_\theta\|_{L^2(w_{B_R })}^p)^{1/p}+R^{-M}\|f\|_2. 
    \end{equation}
\end{theorem}

\medskip

In the rest of the appendix, we prove the following refined decoupling inequality, which also implies Theorem \ref{l2-decoupling}. 
The implication can be found below Theorem 4.2 in \cite{guth2020falconer}. 
The refined decoupling inequality for the cone was proved in \cite[Appendix A]{gan2022restricted}, when $\cf^\la$ is of constant coefficient.
Similarly, the proof for the variable coefficient case is based on the $\ell^2$-decoupling theorem and an induction on scales argument.

\begin{theorem}
\label{refdecthm}     
Fix $\nu=\e^{50}$ and let $K=R^{\nu}$. 
Suppose we have the wave packet decomposition at scale $R$ in $B_R^{n+1}$:
    \[\cFl f=\sum_{T\in \W} \cFl f_T. \]
Assume that $\|\cFl f_T\|_{L^p(w_{B_R })}$ are about the same for all $T\in\W$. Let $Y$ be a disjoint union of $K^2$-balls in $B_R $, each of which intersects $\le M$ many $T\in \W$. Then for $2\le p\le \frac{2(n+1)}{n-1}$ and any $\e>0$,
    \begin{equation}
    \nonumber
        \|\cFl f\|_{L^p(Y)}\le C_{\e,\nu} R^\e (\frac{M}{|\W|})^{\frac12-\frac1p}(\sum_{T\in\W} \|\cFl f_T\|_{L^p(w_{B_R })}^2)^{1/2}+\rap(R) \|f\|_2. 
    \end{equation}
Since $\|\cFl f_T\|_{L^p(w_{B_R })}$ are about the same for all $T\in\W$, we equivalently have
\begin{equation}
\label{refdecineq}
    \|\cFl f\|_{L^p(Y)}\le C_{\e,\nu} R^\e M^{\frac12-\frac1p}(\sum_{T\in\W} \|\cFl f_T\|_{L^p(w_{B_R })}^p)^{1/p}+\rap(R) \|f\|_2. 
\end{equation}
\end{theorem}

\begin{proof}
We will prove \eqref{refdecineq} by induction on the scale $R$. 
Suppose \eqref{refdecineq} holds for scales $\la'\le \la/K^2, R'\le R/K^2$. 
Let 
    \[ R_1=R/K^2=R^{1-2\nu}\ \textup{and}\  K_1=R_1^\nu=R^{\nu-2\nu^2}.\] 
For $\tau\in \Theta_{K^2}$, we define $\Box_{\tau}^\flat=(R/K^{-2})  \tau^*$ to be the isotropic dilate of $\tau^\ast$, a box centered at the origin in $\R^n$ of dimensions $RK^{-2}\times RK^{-1}\times \dots\times RK^{-1}$. 
Cover $B_R^n$ by translated copies of $\Box_\tau^\flat$ and denote this cover by $\B^\flat_\tau=\{\Box^\flat\}$.
For each $\Box^\flat$, as in \eqref{defgaVeq}, we define 
    \[\Box=\Ga_{\Box^\flat}(\xi_\tau;R).\] 
Let
$\B_\tau=\{\Box:\Box^\flat\in\B^\flat_\tau\}$, so the fat curved planks in $\B_\tau$ form a finitely overlapping covering of the $n+1$ dimensional ball $B_R$. 
This induces a partition of wave packets $\W=\bigsqcup_{\tau\in\Theta_{K^2}}\bigsqcup_{\Box\in\B_\tau} \W_\Box$. 
Here for each $T\in\W$, we assign $T$ to $\W_\Box$ if $T\subset \Box$. 
In cases where there are multiple valid choices for $\Box$, we simply choose one.

\begin{figure}[ht]
\centering
\includegraphics[width=10cm]{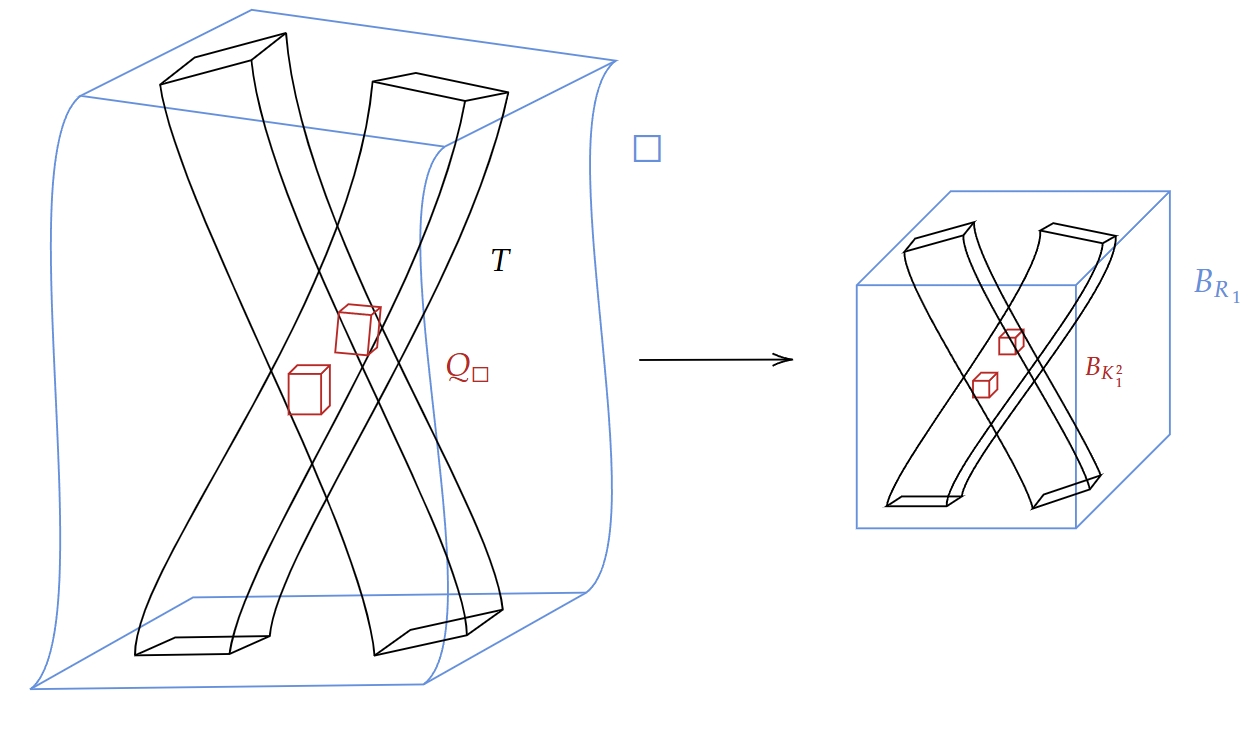}
\caption{Rescaling}
\label{refdec}
\end{figure}

For each $\Box\in \B_\tau$, 
let $\cQ_\Box=\{Q_\Box\}$ be a set of planks that form a finitely overlapping covering of $100\Box$, where each $Q_\Box$ is of the form 
\begin{equation}
\label{Qbox}
    \Ga_W(\bar z;\xi_\tau;K^2K_1^2) 
\end{equation} 
(see \eqref{labeldeftau}) for some point $\bar z\in 100\Box$ and some $K_1^2\times KK_1^2\times \dots\times KK_1^2$-slab $W$ parallel to $K_1^2 \tau^*$, an isotropic dilate of $\tau^\ast$.
The size of $Q_\Box$ is chosen so that when $\Box$ is rescaled to become $B_{R_1}$,  each $Q_\Box$ becomes a $K_1^2$-ball, enabling us to perform induction. 
See Figure \ref{refdec}.
Let $\{\eta_{Q_\Box}\}_{Q_\Box\in\cQ_\Box}$ be a partition of unity adapted to $\cQ_\Box=\{Q_\Box\}$.

\smallskip

Write $\cFl f$ as
    \[\cFl f=\sum_{\W}\cFl f_T=\sum_{\Box}\sum_{T\in\W_\Box}\cFl f_T. \]
For each $\Box$, by dyadic pigeonholing on $\{\#\{T\in\W_\Box: T\cap Q_\Box\not=\varnothing\}:Q_\Box\in\cq_\Box\}$, there exists a number $M'(\Box)$ and a set $Y_\Box$, which is a union of $Q_\Box$, such that
    \begin{enumerate}
        \item $\#\{T\in\W_\Box: T\cap Q_\Box\not=\varnothing\}\sim M'(\Box)$ for all $Q_\Box\subset Y_\Box$.
        \item We have
        \begin{equation}
        \label{pigeonhole-1}
            \|\cFl f\|_{L^p(Y)}\lesssim (\log R)\bigg\| \sum_\Box\sum_{T\in\W_\Box}\eta_{Y_\Box}\cFl f_T \bigg\|_{L^p(Y)}.
        \end{equation}
    \end{enumerate}
Here, $\eta_{Y_\Box}=\sum_{Q_\Box\subset Y_\Box}\eta_{Q_\Box}$.
By dyadic pigeonholing on $\{M'(\Box)\}_\Box$, there exists a set $\B$ and a uniform number $M'$ such that 
    \begin{enumerate}
        \item $M'(\Box)\sim M'$ for all $\Box\in\B$.
        \item We have     
        \begin{equation}
        \label{pigeonhole-2}
            \bigg\| \sum_\Box\sum_{T\in\W_\Box}\eta_{Y_\Box}\cFl f_T \bigg\|_{L^p(Y)}\lesssim (\log R)^2 \bigg\| \sum_{\Box\in \B}\sum_{T\in\W_\Box}\eta_{Y_\Box}\cFl f_T \bigg\|_{L^p(Y)}.
        \end{equation}
    \end{enumerate}
Finally, by dyadic pigeonholing on $\{\#\{Y_\Box:Y_\Box\cap Q\not=\varnothing, \Box\in\B\}: Q\subset Y\}$, there exists a number $M''$ and a subset $Y'\subset Y$ such that 
\begin{enumerate}
    \item $\#\{Y_\Box:Y_\Box\cap Q\not=\varnothing, \Box\in\B\}\sim M''$ for all $Q\subset Y'$.
    \item We have
    \begin{equation}
    \label{pigeonhole-3}
        \bigg\| \sum_{\Box\in\B}\sum_{T\in\W_\Box}\eta_{Y_\Box}\cFl f_T \bigg\|_{L^p(Y)}\lesssim (\log R) \bigg\| \sum_{\Box\in\B}\sum_{T\in\W_\Box}\eta_{Y_\Box}\cFl f_T \bigg\|_{L^p(Y')}.
    \end{equation}
\end{enumerate}

\smallskip

We want to estimate the right-hand side of \eqref{pigeonhole-3} using Bourgain-Demeter's $\ell^2$ decoupling theorem (\cite[Theorem 1.2]{Bourgain-Demeter}), that is, the $\ell^2$ decoupling theorem in the constant coefficient setting.
This requires us to study the Fourier transform of the function in the right-hand side of \eqref{pigeonhole-3}.

For each $K^2$-ball $Q\subset Y'$, let $\Id_Q^*$ be a smooth bump function adapted to $Q$, so
\begin{equation}
\label{bump-fcn}
    \bigg\|\sum_{\Box\in\B}\sum_{T\in\W_\Box}\eta_{Y_\Box}\cFl f_T \bigg\|_{L^p(Q)}\lesssim\bigg\|\Id_Q^*\sum_{\Box\in\B}\sum_{T\in\W_\Box}\eta_{Y_\Box}\cFl f_T \bigg\|_{p}.
\end{equation}
Suppose $Q=B_{K^2}(z_0)$ is centered at $z_0$. 
For each $\tau\in\Theta_{K^{-1}}$, define the cap
    \[ P(z_0;\tau):=N_{K^{-2+\de}}(\{ \partial_z\phl(z_0;\xi):\xi\in 2\tau \}). \]
Here, $\de=\e^{1000}$.
Note that $\{P(z_0;\tau)\}_{\tau\in\Theta_{K^{-1}}}$ is a set of $1\times K^{-1}\times \dots\times K^{-1}\times K^{-2+\de}$-caps that form a finite-overlapping cover of the $K^{-2+\de}$-neighborhood of the conical surface
    \[ \Si_{z_0}:=\{ \partial_z\phl(z_0;\xi):\xi\in\A^n(1) \}. \]
One may view $P(z_0;\tau)$ as the $K^{-2+\de}$-neighborhood of the tangent plane of $\Si_{z_0}$ at $\xi_\tau$. 
Note that the normal direction of the tangent plane is $(\partial_t\gal(u_0,t_0;\xi_\tau,t_0),1)$, where $u_0=\partial_\xi\phl(z_0;\xi_\tau)$. 
Thus, we can think of $P(z_0;\tau)$ as
\begin{equation}
\nonumber
\begin{split}
    \partial_z\phl(z_0;\xi_\tau)+
    \bigg(\Big\{  \partial_\xi\partial_z\phl(z_0;\xi_\tau)\cdot(\xi-\xi_\tau):\xi\in \tau\Big\}\\
    \oplus\Big\{r(\partial_t\gal(u_0,t_0;\xi_\tau,t_0),1): -K^{-2+\de}\le r\le K^{-2+\de}  \Big\}\bigg).
\end{split}
\end{equation}
Here $\oplus$ denotes the orthogonal sum.
This shows that a dual box of $P(z_0;\tau)$ is parameterized by
\begin{equation}
\nonumber
\begin{split}
    &\bigg(\proj_{\R^n}( \{  \partial_\xi\partial_z\phl(z_0;\xi_\tau)\cdot(\xi-\xi_\tau):\xi\in \tau \})\bigg)^*\\
    &+\Big\{r(\partial_t\gal(u_0,t_0;\xi_\tau,t_0),1): -K^{2-\de}\le r\le K^{2-\de}  \Big\}.
\end{split}
\end{equation}
Since $\partial_\xi\partial_x\phl(z_0;\xi_\tau)\cdot \partial_u\gal(u_0,t_0;\xi_\tau)=I_n$ and since $\phi$ obeys the  quantitative condition $(\text{H1}_{\bA})$, we have
\begin{align*}
    &\bigg(\proj_{\R^n}( \{  \partial_\xi\partial_z\phl(z_0;\xi_\tau)\cdot(\xi-\xi_\tau):\xi\in \tau \})\bigg)^* \\
   =&\,\bigg(\{  \partial_\xi\partial_x\phl(z_0;\xi_\tau)\cdot(\xi-\xi_\tau):\xi\in \tau \}\bigg)^*\sim\Big\{  \partial_u\gal(u_0,t_0;\xi_\tau)\cdot u:u\in \tau^* \Big\}.
\end{align*}
Hence, we get
\begin{equation}
\label{tangentplane}
\begin{split}
    P^*(z_0;\tau)\sim &\,\Big\{  \partial_u\gal(u_0,t_0;\xi_\tau)\cdot u:u\in \tau^* \Big\}
    \\
    &+\Big\{r(\partial_t\gal(u_0,t_0;\xi_\tau),1): -K^{2-\de}\le r\le K^{2-\de}  \Big\}.
\end{split}
\end{equation}

\begin{lemma}
Given $\tau\in\Theta_{K^{-1}}$ and $\Box\in\B_\tau$,
    \[ \Id_Q^*\sum_{T\in\W_\Box}\eta_{Y_\Box}\cFl f_T, \]
as a function in $\R^{n+1}$, has Fourier transform essentially supported in $2P(z_0;\tau)$.
\end{lemma}

\begin{proof}
Note that, up to a rapidly decreasing term,
    \[\Id_Q^*\sum_{T\in\W_\Box}\eta_{Y_\Box}\cFl f_T=\Id_Q^*\sum_{Q_\Box\subset Y_\Box,Q_\Box\cap 2Q\neq\emptyset}\eta_{Q_\Box}\sum_{T\in\W_\Box}\cFl f_T.\]
Let $P_0(z_0;\tau)$ be a translated copy of $P(z_0;\tau)$ that is centered at the origin.
Since $(\Id^*_Q\eta_{Q_\Box}\cFl f_T)^\wedge=\wh \eta_{Q_\Box}*(\Id^*_Q\cFl f_T)^\wedge$, it suffices to show that
\begin{enumerate}
    \item $(\Id_Q^*\cFl f_T)^\wedge$ is essentially supported in $P(z_0;\tau)$.
    \item $\wh \eta_{Q_\Box}$ is essentially supported in $P_0(z_0;\tau)$.
\end{enumerate}

We write
\begin{equation}
\label{computeFsupp}
    (\Id_Q^*\cFl f_T)^\wedge(\eta)=\int \Id_Q^*e^{i(\phl(z;\xi)-\langle z,\eta\rangle)}\mathrm{d}z  \int a^\la(z;\xi) \wh f_T \mathrm{d}\xi.  
\end{equation} 
Note that $\wh f_T$ is supported in $2\theta\subset 2\tau$.
Thus, to show item (1), we just need to show that for all $\xi\in2\tau$, 
\begin{equation}
\label{compute2}
    \int \Id_Q^*e^{i(\phl(z;\xi)-\langle z,\eta\rangle)}\mathrm{d}z
\end{equation}
decays rapidly outside $P(z_0;\tau)$.
By stationary phase method, this is true if
    \[ |\partial_z\phl(z;\xi)-\eta|\ge K^{-2+\de}\ \textup{~for~all~}z\in 2Q. \]
Note that $|\partial_z\phl(z;\xi)-\partial_z\phl(z_0;\xi)|\lesssim |z-z_0|\la^{-1}\le K^{-3}$, as $K$ is a small power of $\la$. 
Hence, \eqref{compute2} is negligible if
    \[ |\partial_z\phl(z_0;\xi)-\eta|\ge K^{-2+\de}. \]
Therefore, \eqref{computeFsupp} is essentially supported in
    \[ N_{K^{-2+\de}}(\{ \partial_z\phl(z_0;\xi): \xi\in2\tau \})=P(z_0;\tau). \]
This proves item (1).
\medskip

Regarding item (2), we study $\eta_{Q_\Box}$ for $Q_\Box\cap 2Q\neq\emptyset$. 
By \eqref{Qbox}, for some slab $W$,
    \[ Q_\Box=\{ (\gal(u+\partial_z\phl(\bar z;\xi_\tau),t;\xi_\tau),t):u\in W, |t-\bar t|\le K^2K_1^2 \}, \]
where $\bar z=(\bar x,\bar t)$.
We first show that $Q_\Box$ is morally a box. 
Let $\bar u$ be the center of $W$. 
By Taylor's expansion, we have
\begin{align}
\nonumber
    \gal(u+\partial_z\phl(\bar z;\xi_\tau),t;\xi_\tau)
    =&\,\gal(\bar u+\partial_z\phl(\bar z;\xi_\tau),\bar t;\xi_\tau)\\ \label{linearapproximate}
    +&\partial_u\gal(\bar u+\partial_z\phl(\bar z;\xi_\tau),\bar t;\xi_\tau)\cdot(u-\bar u)\\
    \nonumber
    +&\partial_t\gal(\bar u+\partial_z\phl(\bar z;\xi_\tau),\bar t;\xi_\tau)(t-\bar t)+O(\la^{-1}(KK_1)^2).
\end{align} 
Since $K,K_1\leq R^{\nu}\leq\la^{\nu}$, $O(\la^{-1}(KK_1)^2)=O(1)$.
Hence, $Q_\Box$ can be approximated by the linear term in \eqref{linearapproximate}, yielding that $Q_\Box$ is morally a box. 
As a result, the translation of $Q_\Box$ to the origin can be essentially parameterized by 
\begin{align}
\nonumber
    \Bigg\{\Bigg(\partial_u\gal(\bar u+\partial_z\phl(\bar z;\xi_\tau),\bar t;\xi_\tau)\cdot u
    +&\partial_t\gal(\bar u+\partial_z\phl(\bar z;\xi_\tau),\bar t;\xi_\tau)t,t\Bigg)\\ \nonumber
    &:u\in K_1^2\tau^*,|t|\le K^2K_1^2\Bigg\}. 
\end{align}
Again, since $|\bar u|\le K^2$ and $|\bar z-z_0|\lesssim K^2$, by using Taylor's expansion to get rid of the second order term, the above rectangle can be essentially parameterized by
\begin{equation}\label{samerectangle}
    \Bigg\{\Bigg(\partial_u\gal(u_0,t_0;\xi_\tau)\cdot u
+\partial_t\gal(u_0,t_0;\xi_\tau)t,t\Bigg):u\in K_1^2\tau^*,|t|\le K^2K_1^2\Bigg\},
\end{equation}
where replace $\bar u$ by $0$ and replace $\bar z$ by $z_0$ (recall that $u_0=\partial_\xi\phl(z_0;\xi_\tau)$). 

\smallskip

To show that $\wh \eta_{Q_\Box}$ is essentially supported in $P_0(z_0;\tau)$, we just need to show that $Q_\Box$ contains a translated copy of $P^*_0(z_0;\tau)$. 
In other words, we just need to show \eqref{samerectangle} contains \eqref{tangentplane}, which is true.
\qedhere

\end{proof}

Now we can use Bourgain-Demeter's $\ell^2$ decoupling theorem in $Q$ to have
    \[\text{R.H.S of }\eqref{bump-fcn}\le C_\e K^{\e/100}(M'')^{\frac12-\frac1p}\bigg( \sum_{\Box\in\B} \bigg\| \sum_{T\in\W_\Box}\eta_{Y_\Box}\cFl f_T \bigg\|_{L^p(\om_Q)}^p \bigg)^{1/p}. \]
Recall \eqref{pigeonhole-1}, \eqref{pigeonhole-2}, and \eqref{pigeonhole-3}.
Sum up $Q\subset Y'$ so that
\begin{equation}
\nonumber
    \|\cFl f\|_{L^p(Y)}\lesssim C_\e (\log R)^{100} K^{\e/100}(M'')^{\frac12-\frac1p} \bigg( \sum_{\Box\in\B}\bigg\|\sum_{T\in\W_\Box}\cFl f_T\bigg\|^p_{L^p(Y_{\Box})} \bigg)^{1/p}. 
\end{equation}
Similar to Lemma \ref{lemrescaling}, use \eqref{refdecineq} at scale $R_1=R/K^2$ as an induction hypothesis (see Figure \ref{refdec}) to get
\begin{equation}
\nonumber
    \bigg\|\sum_{T\in\W_\Box}\cFl f_T\bigg\|_{L^p(Y_{\Box})} \lesssim C_{\e,\de} R^\e K^{-2\e}\bigg(\frac{M'}{|\W_\Box|}\bigg)^{\frac12-\frac1p}\bigg(\sum_{T\in\W_\Box}\|\cFl f_T\|_{L^p(\om_{B_R})}^2 \bigg)^{1/2}.
\end{equation}
Therefore, $\|\cFl f\|_{L^p(Y)}$ is bounded above by
    \[C_\e (\log R)^{100} K^{\e/100} C_{\e,\de} R^\e K^{-2\e} \bigg( \frac{M'M''}{|\W_\Box|} \bigg)^{\frac12-\frac1p}\bigg(\sum_{\Box\in\B}\bigg(\sum_{T\in\W_\Box}\|\cFl f_T\|_{L^p(\om_{B_R})}^2\bigg)^{p/2} \bigg)^{1/p}. \]
Since $\|\cFl f_T\|_p$ are comparable for all $T\in\W$, we thus have that $\|\cFl f\|_{L^p(Y)}$ is bounded above by
\begin{equation}
\nonumber
    C_\e (\log R)^{100} K^{\e/100} C_{\e,\de} R^\e K^{-2\e} \bigg( \frac{M'M''}{|\W|} \bigg)^{\frac12-\frac1p} \bigg(\frac{|\B||\W_\Box|}{|\W|}\bigg)^{\frac1p}\bigg(\sum_{T\in\W}\|\cFl f_T\|_{L^p(\om_{B_R})}^2 \bigg)^{1/2}.
\end{equation} 
Recall $K=R^{\e^{50}}$.
Thus, when $R$ is large enough, $C_\e (\log R)^{100}K^{\e/100}K^{-2\e}\le 1/1000$.
Since $|\B||\W_\Box|\le 2 |\W|$ and since $M'M''\le 4 M$, we therefore obtain
\begin{equation*}
    \|\cFl f\|_{L^p(Y)}\le C_{\e,\de} R^\e (\frac{M}{|\W|})^{\frac12-\frac1p}(\sum_{T\in\W} \|\cFl f_T\|_{L^p(\om_{B_R})}^2)^{1/2}. \qedhere
\end{equation*}

\end{proof}

\bigskip

\bibliographystyle{abbrv}
\bibliography{ref}

\end{document}